\documentclass[11pt,a4paper]{article}
\usepackage{amsmath,amsfonts,amssymb,amsthm,enumerate,graphicx,mathtools,titling,caption,tabularx}
\usepackage{hyperref}
\usepackage[margin=2.5cm]{geometry}

\usepackage{tikz}
\usetikzlibrary{shapes.geometric}
\usetikzlibrary{decorations.markings}

\usepackage{xypic} 
\usepackage{mathabx}
\usepackage{mathrsfs} 
\usepackage{xspace}
\usepackage{fixltx2e}

\usepackage{authblk}
\usepackage{tocloft}
\usepackage[normalem]{ulem}

\usepackage{pdflscape}
\usepackage[toc,page,header]{appendix}
\usepackage{listings}

\usepackage{color}

\usepackage{makeidx}
\makeindex

\newtheorem{thm}{Theorem}[section]
\newtheorem{lem}[thm]{Lemma}
\newtheorem{prop}[thm]{Proposition}
\newtheorem{cor}[thm]{Corollary}

\theoremstyle{definition}
\newtheorem{defn}[thm]{Definition}
\newtheorem{const}[thm]{Construction}
\newtheorem{ex}[thm]{Example}

\newtheorem{rem}[thm]{Remark}

\newcommand{\ol}{\overline}

\newcommand{\defbold}{\textit}

\newcommand{\inv}{^{-1}}

\newcommand{\CC}{\mathrm{C}}

\newcommand{\tdlc}{t.d.l.c.\@\xspace}

\newcommand{\triv}{\{1\}}

\newcommand{\rist}{\mathrm{rist}}

\newcommand{\Aut}{\mathrm{Aut}}
\newcommand{\Inn}{\mathrm{Inn}}
\newcommand{\Alt}{\mathrm{Alt}}
\newcommand{\Sym}{\mathrm{Sym}}

\newcommand{\Out}{\mathrm{Out}}

\newcommand{\SL}{\mathrm{SL}}
\newcommand{\PSL}{\mathrm{PSL}}
\newcommand{\GL}{\mathrm{GL}}
\newcommand{\PGL}{\mathrm{PGL}}

\newcommand{\GaL}{\Gamma\mathrm{L}}
\newcommand{\PGaL}{\mathrm{P}\Gamma\mathrm{L}}

\newcommand{\propP}[1]{(\mathrm{P}_{#1})}

\newcommand{\Fb}{\mathbb{F}}
\newcommand{\Nb}{\mathbb{N}}

\newcommand{\Qb}{\mathbb{Q}}

\newcommand{\Zb}{\mathbb{Z}}

\newcommand{\mc}[1]{\mathcal{#1}}

\newcommand{\grp}[1]{\langle #1 \rangle}

\begin{document}

\title{Rigid stabilizers and local prosolubility for boundary-transitive actions on trees}

\preauthor{\large}
\DeclareRobustCommand{\authoring}{
\renewcommand{\thefootnote}{\arabic{footnote}}
\begin{center}Colin D. Reid\textsuperscript{1}\footnotetext[1]{Research supported in part by ARC grant FL170100032.}\\ \bigskip
The University of Newcastle \\ School of Information and Physical Sciences \\ Callaghan, NSW 2308, Australia \\ Email: \href{mailto:colinreid29@gmail.com}{colinreid29@gmail.com} 
\end{center}
}
\author{\authoring}
\postauthor{\par}

\maketitle

\begin{abstract}
Let $G$ be a group acting $2$-transitively on the boundary of a locally finite tree, and exclude the situation (which is a genuine exception) where $G$ has both $\PGaL_3(4)$ and $\PGaL_3(5)$ as local actions.  We show that for each half-tree $T_a$, the local action of the rigid stabilizer of $T_a$ at the root of $T_a$ contains the soluble residual of the point stabilizer of the local action of $G$.  In particular, $G$ is locally prosoluble if and only if its local actions have soluble point stabilizers; if $G$ is not locally prosoluble, then it has micro-supported action on the boundary.  We also prove some strong restrictions on the local actions of end stabilizers in $G$.  These results are partly inspired by Radu's classification of groups acting boundary-$2$-transitively on trees with local action containing the alternating group, and partly based on the author's recent classification of finite permutation groups that preserve an equivalence relation, act faithfully on blocks and act transitively on pairs of points from different blocks.
\end{abstract}

\section{Introduction}

\subsection{Background and motivation}

Let $T$ be a locally finite tree (in the combinatorial sense, or its geometric realization; we will use $T$ in either sense as convenient where there is no danger of confusion).  Then $T$ is naturally equipped with a compact space $\partial T$, called the \defbold{boundary} or \defbold{space of ends}, which can be defined in various equivalent ways: for example, it is the set of equivalence classes of geodesic rays in $T$, where two rays are equivalent if they coincide up to removing an initial segment and shifting indices.  (One can equivalently say that two rays belong to the same end if and only if they stay a bounded distance apart.)  For the purposes of this article we will generally assume $T$ is \defbold{thick}, meaning each vertex has degree at least $3$.  We will also make use of the natural partition of the vertices of $T$ into two \defbold{types}\index{type (of vertex)}, $VT = V_0T \sqcup V_1T$, where each vertex only has neighbours of the opposite type.  An \defbold{arc}\index{arc} of $T$ is an ordered pair $(v,w)$ of adjacent vertices of $T$.

Given a group $G$ acting on a set $X$ we will write $G(x)$ for the stabilizer of $x \in X$ in $G$ and $G(x_1,\dots,x_n) = \bigcap^n_{i=1}G(x_i)$.  Suppose $G$ is a closed subgroup of $\Aut(T)$.  We say $G$ is \defbold{boundary-$k$-transitive} if all $k$-tuples of distinct points in $\partial T$ lie in the same $G$-orbit.  The group $\Aut(T)$ itself is boundary-$2$-transitive if and only if $T$ is \defbold{semiregular}, meaning any two vertices of the same type $t$ have the same number $d_t$ of neighbours.  The class $\mc{H}_T$ of closed boundary-$2$-transitive subgroups of $\Aut(T)$ is interesting to study for a number of reasons.

\begin{itemize}

\item Given a closed subgroup $G$ of $\Aut(T)$, $2$-transitivity on the boundary is a natural condition for specifying that $G$ is large from a permutational perspective, which can be expressed in several equivalent ways; see Corollary~\ref{cor:distance_transitivity} below.  For example, it is equivalent to the condition that $G$ is transitive on $\partial T$ and does not fix a vertex or invert an edge.  Another interpretation, if we assume $G$ is type-preserving and regard the tree as a building of type $\widetilde{\mathrm{A}}_1$ in which all lines are apartments, is that $G$ is boundary-$2$-transitive if and only if it is strongly transitive, that is, transitive on pairs consisting of an apartment and a chamber (geometric edge) contained in that apartment.

\item The condition serves as a source of groups in the class $\mathscr{S}$ of nondiscrete, compactly generated topologically simple \tdlc groups.  Specifically, given $G \in \mc{H}_T$, there is a cocompact normal subgroup $G^{(\infty)} \in \mc{H}_T \cap \mathscr{S}$; see \cite[Proposition~3.1.2]{BurgerMozes}.

\item Regarded as a subspace of $\mathrm{Sub}(\Aut(T))$ in the Chabauty topology, the boundary-$2$-transitive groups are known to be relatively well-behaved, as shown by Caprace--Radu in \cite{CapraceRadu}: they form a closed subspace, with closed isomorphism classes, and for every $G \in \mc{H}_T$ there is $H \le G$ that is minimal among closed boundary-$2$-transitive subgroups of $\Aut(T)$.  The groups $H$ that arise this way are a special case of the groups in $\mc{H}_T \cap \mathscr{S}$ from the previous bullet point.  For example, if $T$ is a regular tree of degree $p+1$ for some prime $p$, then the simple group $\Aut(T)^{(\infty)}$ is the type-preserving subgroup of $\Aut(T)$, but $\Aut(T)^{(\infty)}$ properly contains a copy of $\PSL_2(\Qb_p)$ acting on its Bruhat--Tits tree, so $\Aut(T)^{(\infty)}$ is not minimal among closed boundary-$2$-transitive subgroups of $\Aut(T)$.

\item Groups $G \in \mc{H}_T$ give examples of locally compact $k$-transitive groups of homeomorphisms of the Cantor set where $k \in \{2,3\}$; one can then define the topological full group $F(G)$ of the action of $G$ on $\partial T$, consisting of all homeomorphisms $h: \partial T \rightarrow \partial T$ such that for every $\xi \in \partial T$, there is some neighbourhood $U$ of $\xi$ and $g_\xi \in G$ such that $g_\xi$ and $h$ restrict to the same map on $U$.  The group $F(G)$ would then be an example of a highly transitive group of homeomorphisms of the Cantor set.  In some cases the topology of $G$ will extend to a locally compact group topology on $F(G)$, in the same vein as Neretin's groups \cite{Neretin} or their generalization due to Lederle \cite{Lederle}.  For such a group $F(G)$, the derived group $D$ will be simple by \cite[Theorem~4.16]{MatuiSimple}, and one can also show that $D$ is open and compactly generated (work in preparation), so one obtains many examples of groups in $\mathscr{S}$ with highly transitive action on the Cantor set.

\item Given $G \in \mc{H}_T$, and $\xi \in \partial T$, we obtain an action of $G(\xi)$ that is a \defbold{scale group} action in the sense of Willis, meaning a closed action on a locally finite tree that fixes one end and is transitive on the other ends (equivalently, that fixes one end and is vertex-transitive).  Specifically, if $G$ is vertex-transitive we can take $G(\xi)$ acting on $T$ itself, and otherwise we have vertex-transitive actions of $G(\xi)$ on trees $T_{t,\xi}$ with vertex set $V_tT$ ($t \in \{0,1\}$), where $v,w \in V_tT$ are adjacent in $T_{t,\xi}$ if they are distance $2$ apart on a ray of $T$ representing the fixed end.  Scale groups play an important role in the general theory of dynamics of automorphisms of arbitrary \tdlc groups; see \cite{WillisScale}.  Here it seems likely that scale groups arising as end stabilizers of groups in $\mc{H}_T$ will have special properties beyond those of an arbitrary scale group, but little is known in general.

\item Associated to any $G \in \mc{H}_T$ are its \defbold{local actions}\index{local action}, where the local action of type $t$ is the subgroup $F_t$ of $\Sym(\Omega_t)$ induced by the stabilizer of a vertex of type $t$ on the neighbours of that vertex, where $\Omega_t$ is a set of size $d_t$.  (Here we regard the local action as specified up to permutational equivalence.)  Given $G \in \mc{H}_T$, one sees that $F_0$ and $F_1$ are both $2$-transitive.  By general results, $G$ is a cocompact subgroup of a ``universal group'' with the same local actions and the same orbits on vertices, either the Smith universal group $\mathbf{U}(F_0,F_1)$ (\cite[Theorem~6(iii)]{SmithDuke}), or in the case that $G$ is vertex-transitive, the Burger--Mozes universal group $\mathbf{U}(F_0)$ (\cite[Proposition~3.2.2]{BurgerMozes}), which contains $\mathbf{U}(F_0,F_0)$ as its unique subgroup of index $2$\index{U@$\mathbf{U}(F_0,F_1)$}.  (We will recall the construction of $\mathbf{U}(F_0,F_1)$ in Section~\ref{sec:examples} in order to construct a variant of it.)  Conversely, for any pair $(F_0,F_1)$ of finite $2$-transitive permutation groups, then $\mathbf{U}(F_0,F_1) \in \mc{H}_T$.  Universal groups in the sense of Burger--Mozes or Smith are among the best-understood examples of noncompact, nondiscrete \tdlc groups.  Since the finite $2$-transitive permutation groups have been classified, it is plausible to try to classify $\mc{H}_T$ via a local-to-global approach (see also Remark~\ref{rem:Pk} below).  Such a classification has been achieved in certain cases, most notably by Radu (\cite{Radu}) in the case where $d_0,d_1 \ge 6$ and both local actions are symmetric or alternating, and by Burger--Mozes and Radu in the case where both local actions have nonabelian simple point stabilizer.  Given the enormous diversity of totally disconnected locally compact groups, the prospect of classifying a class of such groups specified by a simple permutational condition is tantalizing.

\item The known examples have an interesting diversity of structure of locally normal subgroups, from the perspective of the local structure theory developed in \cite{CRW-Part1}, \cite{CRW-Part2}.  On the one hand, the class includes rank $1$ semisimple algebraic groups over local fields, in which every closed locally normal subgroup is open.  At the other extreme, universal groups $G$ in the sense of Burger--Mozes or Smith have the property that 
\[
G(a) = \rist_G(T_a) \times \rist_G(T_{\ol{a}}),
\]
where $T_a$ is the \defbold{half-tree}\index{half-tree} specified by the arc $a$, consisting of all vertices $v$ that are closer to the terminus of $a$ than they are to the origin of $a$, the reverse arc of $a$ is denoted $\ol{a}$, and $\rist_G(T_a)$ denotes the \defbold{rigid stabilizer}\index{rigid stabilizer} of $T_a$, that is, the pointwise stabilizer of the complement of $T_a$.  The class $\mc{H}_T$ is thus an interesting test case for the local structure theory of \tdlc groups, particularly the local structure of compactly generated almost simple \tdlc groups.  Closely related is the phenomenon of faithful actions on the Cantor set that are \defbold{micro-supported}\index{micro-supported}, that is, for every nonempty open set there is a nontrivial element that fixes all points outside the open set; such actions have proved to be a powerful tool for constructing and understanding \tdlc groups.  Say an action on a thick locally finite tree is micro-supported if it is micro-supported on the boundary, or equivalently, if and only if the rigid stabilizer of every half-tree is nontrivial.  A sufficient condition for producing a micro-supported action on a tree (only assuming properties of $G$ as a topological group) was recently studied by Caprace, Marquis and the author in \cite{CapraceMarquisReid}.  For the known groups in $\mc{H}_T$, the algebraic groups do not have micro-supported action, but most other known examples do have micro-supported action.  While there are several known consequences of having micro-supported action, in this context it is reasonable to hope that \emph{not} having micro-supported action imposes significant restrictions on a group $G \in \mc{H}_T$.

\item Within the context of groups acting on locally finite trees, there are a number of results and conjectures relating the class $\mc{H}_T$ to classes of nonamenable \tdlc groups $G$ with tame unitary representation theory: from the latter perspective, three conditions one can consider, going from strongest to weakest, are CCR, type I, and amenable von Neumann algebra.  The idea goes back at least to Nebbia \cite{Nebbia}, who showed that if a closed subgroup $G \le \Aut(T)$ is both vertex-transitive and CCR, then $G$ acts transitively on $\partial T$, so in fact $G \in \mc{H}_T$; Nebbia also conjectured the converse in the case that $G$ is vertex-transitive.  Nebbia's conjecture has been proved for $G = \mathbf{U}(F_0)$ by Amann \cite{Amann} and Ciobotaru (unpublished), and for the Radu groups by Semal \cite{Semal}.  (The existence of nonamenable locally compact groups with tame unitary representation theory can be contrasted with an earlier theorem of Thoma \cite[Satz~4]{Thoma} that a discrete group is type I if and only if it is virtually abelian.)  In a similar vein, Caprace--Kalantar--Monod and Houdayer--Raum have proved closely related restrictions on a closed subgroup $G$ of $\Aut(T)$ that is nonamenable and preserves no subtree: Caprace--Kalantar--Monod showed (\cite[Corollary~D]{CapraceKalantarMonod}) that if $G$ is type I, then $G \in \mc{H}_T$, while Houdayer--Raum (\cite[Theorem~C]{HR}) showed that if the von Neumann algebra of $G$ is amenable, then $G$ has $2$-transitive local actions.
\end{itemize}

\subsection{Main results}

Throughout the article, for $t \in \{0,1\}$ then $F_t$ will be a $2$-transitive subgroup of $\Sym(\Omega_t)$, where $\Omega_t$ is a finite set of size $d_t \ge 3$.  Write $\mc{H}_{F_0,F_1}$ for the class of closed boundary-$2$-transitive groups acting on the $(d_0,d_1)$-semiregular tree $T$ with local action $F_t$ at vertices of type $t$.  See Figure~\ref{fig:semiregular} for an illustration of a semiregular tree, with vertices of type $0$ marked by white dots and those of type $1$ by black dots.  Write $F_t(\omega)$ for a point stabilizer in $F_t$, which we regard as a transitive permutation group acting on the set $\Omega_t \setminus \{\omega\}$ of $d_t-1$ points.

\begin{figure}
\caption{Part of the $(4,3)$-semiregular tree}
\label{fig:semiregular}
\begin{center}
\begin{tikzpicture}[scale=1, every loop/.style={}, square/.style={regular polygon,regular polygon sides=4}]

\tikzstyle{every node}=[circle,
                        inner sep=0pt, minimum width=5pt]
                     
\foreach \x in {0,1,2,3} {
\foreach \y in {-1,1} {
\foreach \z in {-1,0,1} {
\draw (0,0) to ++ (90*\x:1) node[draw=black, fill=black]{} to ++ (90*\x+45*\y:1) node[draw=black, fill=white]{} to ++ (90*\x+45*\y+30*\z:0.7) node[draw=black, fill=black]{};
\draw (0,0) ++ (90*\x:1) ++ (90*\x+45*\y:2.5) node[rotate=90*\x+45*\y]{$\ldots$};
}
}
};

\draw (0,0) node[draw=black, fill=white]{} to (0,0);
     
\end{tikzpicture}
\end{center}
\end{figure}
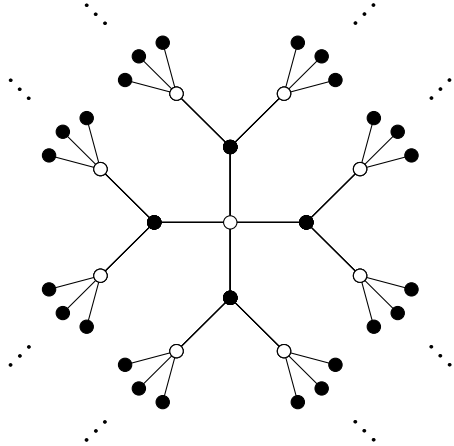

In the present article, we prove restrictions on the local actions of certain important subgroups of $G \in \mc{H}_{F_0,F_1}$, namely: the local action $\Lambda_t(G)$ of $G(\xi)$ at a vertex of type $t$, where $\xi$ is an end of $T$, and the local action $\Theta_t(G)$ at the root of the rigid stabilizer of a half-tree (equivalently, the pointwise stabilizer of the opposite half-tree), where the root is a vertex of type $t$.  Both groups are naturally regarded as subgroups of $F_t(\omega)$, since the local action in each case has a natural fixed point: for $G(\xi)$, the vertex stabilizer also fixes the neighbouring vertex in the direction of $\xi$, and for the rigid stabilizer $\rist_G(T_a)$, then we are considering the local action of $\rist_G(T_a)$ at the terminus of the arc $a$, an action that has the origin of $a$ as a fixed point.  Using the fact that $G$ is contained in $\mathbf{U}(F_0,F_1)$ or $\mathbf{U}(F_0)$ (depending on whether or not $G$ is vertex-transitive), it is easy to see that $G \in \{\mathbf{U}(F_0,F_1),\mathbf{U}(F_0)\}$ if and only if $\Theta_t(G) = F_t(\omega)$ for both $t \in \{0,1\}$.

Previous work has suggested that a key source of complexity for groups in $\mc{H}_{F_0,F_1}$ is the set of quotients of $F_t(\omega)$, particularly soluble quotients.  For example we have \cite[Theorem~F]{Radu}: if $F_t(\omega)$ is a nonabelian simple group for both $t=0,1$, then up to conjugacy the class $\mc{H}_{F_0,F_1}$ consists only of $\mathbf{U}(F_0,F_1)$, together with $\mathbf{U}(F_0)$ if $F_0=F_1$.

We say a topological group is \defbold{locally prosoluble} if it has a prosoluble open subgroup.  In the present context, given $G \le \Aut(T)$ closed and $v_1,\dots,v_n \in VT$, then the open subgroup $G(v_1,\dots,v_n)$ of $G$ is profinite, so it is residually soluble if and only if it is prosoluble, if and only if it is an inverse limit of finite soluble groups.

We make the following basic observations, to be compared with Corollary~\ref{intro:soluble} below.

\begin{lem}
Let $G \in \mc{H}_{F_0,F_1}$.
\begin{enumerate}[(i)]
\item Given an arc $a$ of the tree, then $G(a)$ is prosoluble if and only if both $F_0(\omega)$ and $F_1(\omega)$ are soluble.
\item If $G$ has a residually soluble open subgroup, then $\Theta_0(G)$ and $\Theta_1(G)$ are soluble.
\end{enumerate}
\end{lem}

The main situation we will be interested in is when at least one of $F_0(\omega)$ and $F_1(\omega)$ is insoluble.  Note that any finite group $F$ has a smallest normal subgroup $N$ such that $F/N$ is soluble, the \defbold{soluble residual} $N = O^{\infty}(F)$ of $F$; the soluble residual is always perfect, and is trivial if and only if $F$ is soluble.  Similarly there is a largest soluble normal subgroup of $F$, the \defbold{soluble radical} $O_\infty(F)$, which is equal to $F$ if and only if $F$ is soluble.

Our first main theorem is that, except in a very special situation, $\Theta_t(G)$ always contains the corresponding soluble residual $O^\infty(F_t(\omega))$.

\begin{thm}\label{intro:locally_prosoluble}
Let $G \in \mc{H}_{F_0,F_1}$.  Then exactly one of the following holds.
\begin{enumerate}[(i)]
\item We have $\Theta_t(G) \ge O^\infty(F_t(\omega))$ for all $t \in \{0,1\}$.
\item The local actions are $F_t = \PGaL_3(q_t)$ acting on the projective plane of order $q^2_t + q_t +1$, such that $\{q_0,q_1\} = \{4,5\}$.  Writing $W_t$ for the socle of $F_t(\omega)$, then $\Lambda_t(G)$ is a soluble group of the form $W_t \rtimes A(q_t)$, where
\[
A(4) = \GaL_1(16) \; \text{ and } A(5) \in \{\SL_2(3) \rtimes C_4, \SL_2(3) \rtimes C_2,  \SL_2(3), C_3 \rtimes C_8  \}.
\]
In this case, $\Theta_t(G) \le O_\infty(F_t(\omega))$ for all $t \in \{0,1\}$.
\end{enumerate}
\end{thm}

We say that $G \in \mc{H}_{F_0,F_1}$ is of \defbold{exceptional locally prosoluble type} if $G$ has a prosoluble open subgroup, but at least one of $F_0(\omega)$ and $F_1(\omega)$ is insoluble.  Theorem~\ref{intro:locally_prosoluble} implies that the exceptional locally prosoluble type can only occur with a single pair of local actions on the $(31,21)$-semiregular tree.  This exceptional case does occur, see Example~\ref{ex:exceptional} below.  We emphasize that for $G \in \mc{H}_T$, having local actions $\PGaL_3(4)$ and $\PGaL_3(5)$ does not by itself imply that $G$ falls under case (ii), for instance the group $\mathbf{U}(\PGaL_3(4),\PGaL_3(5))$ falls under case (i).

The next corollary explains the relationship between case (ii) of Theorem~\ref{intro:locally_prosoluble} and whether or not $G$ is locally prosoluble.

\begin{cor}\label{intro:soluble}
Let $G \in \mc{H}_{F_0,F_1}$.  Then the following are equivalent:
\begin{enumerate}[(i)]
\item $G$ is locally prosoluble, in other words $G$ has a residually soluble open subgroup.
\item One of the following holds:
\begin{enumerate}[(a)]
\item In both $F_0$ and $F_1$, point stabilizers are soluble;
\item Case (ii) of Theorem~\ref{intro:locally_prosoluble} holds.
\end{enumerate}
\end{enumerate}
\end{cor}

We have the following consequence for micro-supported actions.

\begin{cor}\label{intro:TMS}
Let $G \in \mc{H}_T$.  Suppose that $G$ does not have a residually soluble compact open subgroup, and let $\xi$ be an end of $T$.  Then $G$ and $G(\xi)$ have micro-supported action on $\partial T$.  In fact, given an arc $a$ directed away from $\xi$, then $\rist_G(T_a)$ is a TMS subgroup of both $G$ and $G(\xi)$ in the sense of \cite{CapraceMarquisReid}.
\end{cor}

If $G$ is of exceptional locally prosoluble type, our methods do not determine whether or not $G$ must have micro-supported action on $\partial T$.  It would be interesting to find $G \in \mc{H}_{\PGaL_3(5),\PGaL_3(4)}$ with trivial rigid stabilizers of half-trees, if such a group exists.

We also generalize \cite[Theorem~F]{Radu}.

\begin{cor}\label{intro:perfect}
Let $G \in \mc{H}_{F_0,F_1}$, acting on the tree $T$.  Suppose that the point stabilizers $F_0(\omega)$ and $F_1(\omega)$ are both perfect.  Then up to conjugacy in $\Aut(T)$, $G$ is either $\mathbf{U}(F_0,F_1)$ or $\mathbf{U}(F_0)$.
\end{cor}

The second main result is that for many pairs $(F_0,F_1)$ of local actions, given $G \in \mc{H}_{F_0,F_1}$ then the local action of type $t$ of an end stabilizer of $G$ is necessarily $F_t(\omega)$ (which is the largest it can be given the local actions of $G$).

\begin{thm}\label{intro:end_stabilizer}
Let $G \in \mc{H}_{F_0,F_1}$, let $t \in \{0,1\}$ and write $O_t = O^{\infty}(F_t(\omega))$.  Suppose $\Lambda_t(G)$ is properly contained in $F_t(\omega)$, but that $G$ is not of exceptional locally prosoluble type.  Then either $O_{1-t}=\triv$, or $F_{1-t}$ is of affine type with $O_{1-t} \cong \SL_2(5)$.  Moreover, one of the following holds:
\begin{enumerate}[(i)]
\item We have $\PSL_{n+1}(q) \unlhd F_t \le \PGaL_{n+1}(q)$ for some $n \ge 2$, with $F_t$ acting on the projective space $P_n(q)$;
\item The socle of $F_t$ is of rank $1$ simple Lie type;
\item We have $(F_t,F_t(\omega),\Lambda_t(G)) = (\mathrm{M}_{11},\mathrm{M}_{10},\Alt(6))$.
\end{enumerate}
\end{thm}

The cases in Theorem~\ref{intro:end_stabilizer} are not the strongest possible statements; see Theorem~\ref{thm:end_stabilizer} for a more precise statement of the exceptional cases that are still outstanding.  Even outside the exceptional locally prosoluble case, it is not true in general that $\Lambda_t(G) = F_t(\omega)$; see Example~\ref{ex:ldc_end}.

However, if we assume both local actions are the same (which includes, for example, the case that $G$ is vertex-transitive), the conclusion is more definitive:

\begin{thm}\label{intro:end_stabilizer:same}
Let $G \in \mc{H}_{F_0,F_1}$, where $F_0=F_1$ as permutation groups.  Then $\Lambda_t(G) = F_t(\omega)$ for $t=0,1$.
\end{thm}

The main theorems are all based on a combination of two lines of argument.  The first broadly follows a strategy used by Radu, an induction argument on pointwise stabilizers of balls in the tree, albeit in greater generality.  This argument by itself, culminating in Proposition~\ref{prop:soluble_dichotomy}, requires relatively little information about the classification of finite $2$-transitive permutation groups, but is not sufficient to prove Theorem~\ref{intro:locally_prosoluble}.

The second aspect is based on the following observation.  We say a permutation group $F \le \Sym(\Omega)$ is \defbold{$2$-by-block-transitive} if $F$ preserves some equivalence relation $\sim$ with at least two blocks and is transitive on pairs $(\omega,\omega')$ belonging to different blocks.

\begin{lem}[See Corollary~\ref{cor:tree_2bbtrans}]
Let $T$ be a thick locally finite tree, let $G \in \mc{H}_{F_0,F_1}$ and let $t \in \{0,1\}$.  Then $F_t$ has $2$-by-block-transitive action on $F_t/\Lambda_t(G)$, with block stabilizer $F_t(\omega)$.
\end{lem}

We can thus appeal to a classification of finite $2$-by-block-transitive permutation groups acting faithfully on blocks, which is obtained in the article \cite{Reidkblock}, to greatly limit the possibilities for the group $\Lambda_t(G)$ and to complete the proof of Theorem~\ref{intro:locally_prosoluble}.

\begin{rem}\label{rem:Pk}
In principle, for a thick locally finite tree $T$, one could classify groups in $\mc{H}_T$ in terms of $k$-local actions, as follows.  Given a locally finite tree $T$, some $k \ge 0$ and $G \le \Aut(T)$, the \defbold{$k$-local action} of $G$ at $v \in VT$ is the automorphism group induced by the vertex stabilizer $G(v)$ on the tree $B(v,k)$ spanned by vertices at distance at most $k$ from $v$.  One can form the \defbold{$\propP{k}$-closure $G^{\propP{k}}$} of $G$ (also called the $k$-closure), which consists of all automorphisms $\alpha$ of $T$ such that for all vertices $v \in VT$, there is $g \in G$ such that $\alpha w = g_v w$ for every vertex and arc $w$ of $B(v,k)$.

Given $G \le \Aut(T)$, we claim that $G \in \mc{H}_T$ if and only if $G^{\propP{k}} \in \mc{H}_T$ for all $k$.  The $\propP{k}$-closures of $G$ form a descending chain as $k \rightarrow +\infty$, whose intersection is the closure of $G$ in $\Aut(T)$.  We see that $G^{\propP{1}}$ has the same vertex orbits as $G$, so we may suppose any two vertices of the same type are in the same $G$-orbit.  By Corollary~\ref{cor:distance_transitivity} below we have $G \in \mc{H}_T$ if and only if $G$ is transitive on each set $V_{a,b,k}$, where $V_{a,b,k}$ consists of ordered pairs $(x,y)$ of vertices where $x$ and $y$ belongs to the $G$-orbits $a$ and $b$ respectively and $d(x,y)=k$.  The claim follows by observing that $G$ is transitive on $V_{a,b,k}$ if and only if $G^{\propP{2k}}$ is transitive on $V_{a,b,k}$.  A similar argument shows that $G$ is boundary-$3$-transitive if and only if $G^{\propP{k}}$ is boundary-$3$-transitive for all $k$.  Thus a suitably detailed classification of $\propP{k}$-closed groups would in particular classify the multiply transitive closed subgroups of $\Aut(T)$ acting on $\partial T$.

However, for $k \ge 2$ the theory of $\propP{k}$-closed groups is still at a relatively early stage of development.  The main articles developing the theory are \cite{BanksElderWillis} and \cite{Tornier}; in the latter, Tornier provides a universal construction for a special case, namely when $G$ is $\propP{k}$-closed, vertex-transitive and contains an involution that inverts an edge, then $G$ must take the form $\mathbf{U}_k(F^{(k)})$, where $F^{(k)}$ is a subgroup of $\Aut(B(v,k))$ satisfying a certain compatibility condition (C).  For examples of groups $\mathbf{U}_k(F^{(k)})$ that are boundary-$2$-transitive, we would also impose the requirement that $F^{(k)}$ be $2$-by-block-transitive on the leaves of $B(v,k)$.  However, even assuming a given $2$-transitive permutation group for the action on $B(v,1)$, there is not yet a robust strategy for classifying subgroups of $\Aut(B(v,k))$ satisfying (C).  The importance of edge inversions for the arguments in \cite{Tornier} also suggests difficulties in generalizing to the case when $G$ has type-preserving action on the tree.
\end{rem}

\paragraph{Structure of article}

The remainder of the article is divided into four sections.  In Section 2 we establish some key characteristics of the class $\mc{H}_T$, including the relationship with $k$-by-block-transitive actions.  Section 3 collects together the relevant preliminaries on finite groups, including the classification of finite block-faithful $2$-by-block-transitive permutation groups for the intended application to groups acting on trees.  The proofs of the main theorems are in Section 4.  In Section 4.1 we perform the induction argument on pointwise stabilizers of balls in the tree; combined with the preliminaries, this actually proves Theorem~\ref{intro:locally_prosoluble} except for a special case, which is when one of the local actions has socle $\PSL_3(q_t)$ for some prime power $q_t > 3$ and the other local action is either of the same form (with a possibly different $q_t$) or has soluble point stabilizers.  In Section 4.2 we obtain some more restrictions on the local actions of stabilizers of line segments in $G$, and then in Section 4.3 we prove the main theorems, including the more detailed version of Theorem~\ref{intro:end_stabilizer}.  Finally, Section 5 presents a construction with some examples, to show that many of the exceptional cases do occur.

\paragraph{Acknowledgement} I thank the anonymous reviewer, whose suggestions have led to significant improvements in presentation.  I also thank Florian Lehner for explaining to me a construction which proved useful in obtaining examples; Sven Raum for clarifications on the representation-theoretic context; and Pierre-Emmanuel Caprace and Stephan Tornier for other helpful comments related to this article.

\section{General results about boundary-$k$-transitive actions on trees}

\subsection{Notation for permutation groups and groups acting on trees}

For general concepts in permutation group theory, the reader may wish to consult \cite{Cameron} and \cite{DixonMortimer}.

\begin{defn}\label{def:permutation}
Let $G$ be a group.  A \defbold{$G$-set} is a set $X$ equipped with an action of $G$ on $X$ by permutations. Given a $G$-set $X$ and $x_1,x_2,\dots, x_n \in X$, write $G(x_1)$\index{G@$G(x)$} for the stabilizer of $x_1$ in $G$ and $G(x_1,\dots,x_n) = \bigcap^n_{i=1}G(x_i)$.  (We use this notation, instead of the more usual $G_{x_1}$ and $G_{(x_1,\dots,x_n)}$, to avoid an overload of subscripts later.)  Let $\Omega$ be a set equipped with an equivalence relation $\sim$; we consider the equivalence classes as \defbold{blocks} of $\Omega$.  Given $\omega \in \Omega$, write $[\omega]$ for the $\sim$-class of $\omega$.  If $\sim$ is $G$-invariant, the quotient map
\[
\Omega \rightarrow \Omega/\sim \; \omega \mapsto [\omega]
\]
is $G$-equivariant with respect to the natural action of $G$ on $\Omega/\sim$.

Given a $G$-equivariant surjective map $\pi: \Omega \rightarrow X$ between $G$-spaces we say $X$ is a \defbold{factor} of $\Omega$ and $\Omega$ is an \defbold{extension} of $X$; if $\pi$ is invertible, we say $\Omega$ and $X$ are \defbold{equivalent} (as $G$-sets) and write $\Omega \cong_G X$.  Note that every factor of $\Omega$ is equivalent to one of the form $\Omega/\sim$ for $\sim$ a $G$-invariant equivalence relation: namely, if $\pi: \Omega \rightarrow X$ is a $G$-equivariant map then $\Omega/\sim_\pi \cong_G X$, where we set $\omega \sim_\pi \omega'$ if and only if $\pi(\omega) = \pi(\omega')$.

Now suppose $G$ acts transitively on $X$.  A \defbold{standard extension}\index{standard extension} of $X$ is a $G$-set $\Omega$ of the following form $X \times B$:
\begin{enumerate}[(a)]
\item $\Omega = X \times B$ for some set $B$;
\item The map $X \times B \rightarrow X; \; (x,b) \mapsto x$ is $G$-equivariant;
\item For all $x,y \in X$, there exists $g_{x,y} \in G$ such that $g_{x,y}(x,b) = (y,b)$ for all $b \in B$.
\end{enumerate}
\end{defn}

\begin{rem}\label{rem:standard_extension}
We recall the basic fact that isomorphism types of transitive $G$-sets are in one-to-one correspondence with conjugacy classes of subgroups of $G$.  In particular, a transitive $G$-set $X$ is equivalent to the coset space $G/G(x) := \{gG(x) \mid g \in G\}$ for every $x \in X$.  Given two coset spaces $G/H$ and $G/K$ such that $H \ge K$, we observe that $G/K$ is an extension of $G/H$ on which $G$ also acts transitively.  Up to equivalence we can realize $G/K$ as a standard extension of $X = G/H$, as follows:
\begin{enumerate}[(i)]
\item Let $B = H/K$ and let $\Omega = X \times B$.
\item For each left coset $gH$ of $H$ in $G$, choose a representative $c_{gH} \in gH$, with $c_H=1$.
\item Define
\[
\pi: \Omega \rightarrow G/K \; \; (gH,hK) \mapsto c_{gH}hK;
\]
we give $G$ the unique action $\Omega$ with respect to which $\pi$ is $G$-equivariant (such an action exists because $\pi$ is a bijection).
\end{enumerate}
It is clear that conditions (a) and (b) of a standard extension are satisfied.  For (c), given $g_1H,g_2H \in G/H$, let $c = c_{g_2H}c^{-1}_{g_1H}$.  Then for all $h \in H$ we have
\[
c(g_1H,hK) = \pi\inv(cc_{g_1H}hK) = \pi\inv(c_{g_2H}hK) = (g_2H,hK),
\]
in other words, $c(g_1H,b) = (g_2H,b)$ for all $b \in B$.

Given a standard extension $X \times B$ of $X$ and $x \in X$, we can define a \defbold{block restriction map} $\beta_x: G \rightarrow \Sym(B)$ via the formula
\[
\forall g \in G, b \in B: g(x,b) = (gx,\beta_x(g)b).
\]
The restriction of $\beta_x$ to $G(x)$ is a homomorphism.  Applied to $G$ itself, $\beta_x$ depends on the choice of $x$, however because of condition (c) of a standard extension, the image $\beta_x(G)$ does not depend on $x$, and moreover $\beta_x(G) = \beta_x(G(x))$, since
\[
\forall y \in X: \beta_x(g_{x,y}G(x)) = \beta_x(G(x)) = \beta_y(g_{x,y}G(x)g_{y,x}) = \beta_y(G(y)).
\]
The \defbold{block action} is then the subgroup $\beta_x(G(x))$ of $\Sym(B)$.  We note that $G$ acts transitively on $X \times B$ if and only if the action on $X$ and the block action are both transitive.
\end{rem}

\begin{defn}
Let $\Omega$ be a set equipped with an equivalence relation $\sim$.  For $k \ge 1$, define the set $\Omega^{[k]}$ of \defbold{distant $k$-tuples} to consist of those $k$-tuples $(\omega_1,\omega_2,\dots,\omega_k)$ such that no two entries lie in the same block.  We then say $G \le \Sym(\Omega)$ is \defbold{$k$-by-block-transitive}\index{k@$k$-by-block-transitive}\index{2@$2$-by-block-transitive} if $G$ preserves $\sim$, there are at least $k$ blocks, and $G$ acts transitively on $\Omega^{[k]}$.
\end{defn}

\begin{rem}\
\begin{enumerate}
\item If a group action is $k$-by-block-transitive, it is also $k'$-by-block-transitive for $k' < k$; in particular, every $k$-by-block-transitive action for $k \ge 1$ is transitive.
\item If $G \le \Sym(\Omega)$ is $k$-by-block-transitive for some $k \ge 2$ with respect to the equivalence relation $\sim$, then $\sim$ is the unique coarsest $G$-invariant equivalence relation other than the universal relation: see \cite[Lemma~2.3]{Reidkblock}.  In particular, $\sim$ is uniquely determined by the action of $G$ on $\Omega$.
\item Suppose we have a $G$-equivariant surjection $\pi: \Omega \rightarrow X$ of $G$-sets, such that $G$ is $k$-transitive on $X$ with $|X| \ge k \ge 2$.  Then $\sim_\pi$ is a $G$-invariant nonuniversal equivalence relation $\sim_\pi$ on $\Omega$, which cannot be made any coarser subject to these conditions.  Thus if $G$ acts $k$-by-block-transitively on $\Omega$ with respect to some equivalence relation $\sim$, then $\sim = \sim_\pi$.  Conversely, if $G$ acts $k$-by-block-transitively on a set $\Omega$ with respect to an equivalence relation $\sim$, then as a $G$-set, $\Omega$ is an extension of the $k$-transitive $G$-set $\Omega/\sim$.
\end{enumerate}
\end{rem}

\subsection{Characterizations of boundary-$k$-transitivity}\label{sec:trees}

In this subsection, we establish some notation for actions on trees and explain the connection between $k$-by-block-transitive actions and transitive actions on the boundary of a tree for $k = 2,3$.  For the theory of groups acting on trees, the standard reference is \cite[Chapter I]{Serre:trees}.

\begin{defn}\label{defn:tree}
Let $T$ be a tree; write $VT$ for the set of vertices of $T$.  Then there is a natural partition of $VT$ into two parts $V_0T$ and $V_1T$\index{V@$VT, V_tT$}, such that if $v \in V_tT$ then all neighbours of $v$ belong to $V_{1-t}T$.  Say a vertex $v$ is of \defbold{type $t$}\index{type (of vertex)} if $v \in V_tT$.  Throughout, if we refer to the type of a vertex, or some other object indexed by vertex type, using an integer (including implicitly, for instance with subscripts), that integer is to be understood modulo $2$.  We say $T$ is \defbold{semiregular}\index{semiregular (tree)} if the \defbold{degree}\index{degree} of $v$, that is, the number of neighbours of $v$, is constant on $v \in V_0T$ and also on $v \in V_1T$; in that case we will write $d_t$\index{d@$d_t$} for the number of neighbours of $v$ for all $v \in VT_t$.  (Both $d_0 = d_1$ and $d_0 \neq d_1$ are allowed.)  If $d_0=d_1$ the tree is \defbold{regular}.  To avoid degenerate cases, we will assume $T$ is \defbold{thick}, that is, every vertex has at least three neighbours.  The \defbold{geometric realization} $[T]$ of $T$ is the geodesic space formed by taking a point for each vertex $x \in VT$, and adjoining a \defbold{geometric edge}, that is, a unit line segment with endpoints $x$ and $y$, for each unordered pair $\{x,y\}$ of neighbouring vertices.  Note that automorphisms of $T$ naturally extend to isometries of $[T]$; we equip $VT$ with the subspace metric in $[T]$.

An \defbold{arc}\index{arc} of $T$ is an ordered pair $(x,y)$ of neighbouring vertices of $T$, while an \defbold{(undirected) edge} is an unordered pair $\{x,y\}$ of neighbouring vertices.  (We will sometimes identify an undirected edge with its associated geometric edge, where it is convenient to do so.)   Given an arc $a = (x,y)$, then $\ol{a}$ denotes the \defbold{reverse} of $a$, namely $\ol{a} = (y,x)$.

In pictures of trees, we adopt the following convention: 

Vertices of interest will be marked as black squares or hollow circles.  The permutation group of interest in the picture is the pointwise stabilizer of the black squares, in its action on the set of hollow circles.

A \defbold{(vertex) path}\index{path} in $T$ is a sequence $P = (x_i)_{i \in I}$ of distinct vertices such that $x_{i+1}$ is adjacent to $x_i$ for all $i \ge 0$, where $I$ is an interval in $\Zb$; the \defbold{length} of the path is $|I|-1$.  We say the path is a \defbold{ray}\index{ray} if the indexing set is the non-negative integers.  Two rays $(x_i)_{i \ge 0}$ and $(y_i)_{i \ge 0}$ are \defbold{equivalent} if there exist $n,t \in \Zb$ such that $y_{i+t} = x_i$ for all $i \ge n$.  An \defbold{end}\index{end} of $T$ is an equivalence class of rays; the set of ends then forms the \defbold{boundary}\index{boundary} $\partial T$\index{d@$\partial T$} of $T$.  (The boundary is also usually equipped with a topology, but this is not relevant for the present discussion.)  Note that for each pair $(v,\xi) \in VT \times \partial T$, there is a unique ray $(v_{\xi,0},v_{\xi,1},\dots)$ representing $\xi$ such that $v_{\xi,0} = v$.

Given a vertex $v \in VT$ we write
\[
S(v,n) = \{w \in VT \mid d(v,w)=n\}.
\]
Given $x,y \in VT$, we write $S_x(y,k)$ for the set of vertices $z$ at distance $k$ from $y$, such that the path from $x$ to $z$ passes through $y$.  To put this another way, if we picture the tree as being rooted at $x$, then $S_x(y,k)$ denotes the set of vertices that are descended from $y$ and are at distance $k$ from $y$.  For $n \ge 1$ we can partition $S(v,n)$ as follows:
\[
S(v,n) = \bigsqcup_{w \in S(v,1)}S_v(w,n-1).
\]
This defines an $\Aut(T)(v)$-invariant equivalence relation $\sim_{v,n}$\index{s@$\sim_{v,n}$} on $S(v,n)$: we write $x \sim_{v,n} y$ if there exists $w \in S(v,1)$ such that $\{x,y\} \subseteq S_v(w,n-1)$.  Note that given $x,y \in S(v,n)$, then $(x,y)$ is a $\sim_{v,n}$-distant pair (in other words, $x \not\sim_{v,n} y$) if and only if $d(x,y)=2n$; see Figure~\ref{fig:svn}, where the blocks are indicated by dashed rectangles and the path between a distant pair is highlighted.

In particular, we see that as $\Aut(T)(v)$-sets, $S(v,n)$ is an extension of $S(v,1)$.

\begin{figure}
\caption{The equivalence relation $\sim_{v,3}$}
\label{fig:svn}
\begin{center}
\begin{tikzpicture}[scale=0.7, every loop/.style={}, square/.style={regular polygon,regular polygon sides=4}]

\tikzstyle{every node}=[circle,
                        inner sep=0pt, minimum width=5pt]

\draw[line width=2pt,color=black] (-5.5,-6) -- (-5,-4) -- (-4,-2) -- (0,0) -- (0,-2) -- (1,-4) -- (0.5,-6);

\draw
{
(0,0.5) node{$v$}
(0,0) node[square, draw=black, fill=black, minimum width=10pt]{}
(-4,-2) node[fill=black]{} to (0,0)
(0,-2) node[fill=black]{} to (0,0)
(4,-2) node[fill=black]{} to (0,0)
(-5,-4) node[fill=black]{} to (-4,-2)
(-3,-4) node[fill=black]{} to (-4,-2)
(-1,-4) node[fill=black]{} to (0,-2)
(1,-4) node[fill=black]{} to (0,-2)
(3,-4) node[fill=black]{} to (4,-2)
(5,-4) node[fill=black]{} to (4,-2)

(-5.5,-6) node[draw=black, fill=white, minimum width=8pt]{} to (-5,-4)
(-4.5,-6) node[draw=black, fill=white, minimum width=8pt]{} to (-5,-4)
(-3.5,-6) node[draw=black, fill=white, minimum width=8pt]{} to (-3,-4)
(-2.5,-6) node[draw=black, fill=white, minimum width=8pt]{} to (-3,-4)

(-1.5,-6) node[draw=black, fill=white, minimum width=8pt]{} to (-1,-4)
(-0.5,-6) node[draw=black, fill=white, minimum width=8pt]{} to (-1,-4)
(0.5,-6) node[draw=black, fill=white, minimum width=8pt]{} to (1,-4)
(1.5,-6) node[draw=black, fill=white, minimum width=8pt]{} to (1,-4)

(2.5,-6) node[draw=black, fill=white, minimum width=8pt]{} to (3,-4)
(3.5,-6) node[draw=black, fill=white, minimum width=8pt]{} to (3,-4)
(4.5,-6) node[draw=black, fill=white, minimum width=8pt]{} to (5,-4)
(5.5,-6) node[draw=black, fill=white, minimum width=8pt]{} to (5,-4)
};

\draw[color=black, dashed] (-5.9,-6.4) rectangle (-2.1,-5.6);
\draw[color=black, dashed] (-1.9,-6.4) rectangle (1.9,-5.6);
\draw[color=black, dashed] (2.1,-6.4) rectangle (5.9,-5.6);
\end{tikzpicture}
\end{center}
\end{figure}
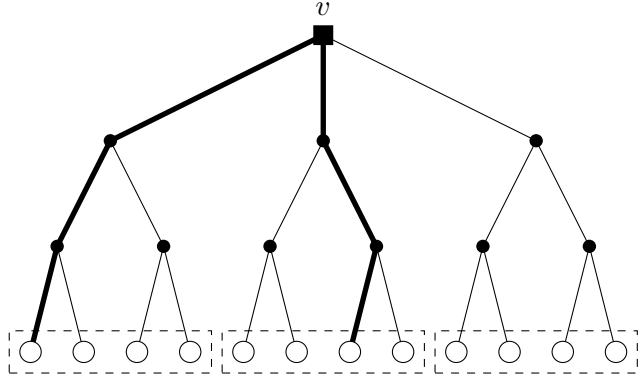

Analogously, we define an equivalence relation $\sim_v$\index{s@$\sim_{v}$} on $\partial T$ by writing $\xi \sim_v \xi'$ if $v_{\xi,1} = v_{\xi',1}$.

Given $G \le \Aut(T)$, the \defbold{type-preserving subgroup} $G^+$ of $G$ is the setwise stabilizer of $V_0T$ in $G$; we say $G$ is \defbold{type-preserving}\index{type-preserving} if $G = G^+$. 
\end{defn}

Given an automorphism $g$ of $\Aut(T)$, define the \defbold{translation length}\index{translation length} $l(g) := \min_{p \in [T]}\{d(p,gp)\}$.  A standard observation is that if $l(g) > 0$, then $\{p \in [T] \mid d(p,gp) = l(g)\}$ is a line in the tree (called the \defbold{axis}\index{axis} of $g$; we call $g$ a \defbold{translation} in this case), whereas if $l(g)=0$, then $g$ fixes a vertex or the midpoint of an edge.  The next lemma gives a way of detecting when $g$ is translation using pairs $(a,ga)$, for $a$ an arc of $T$.

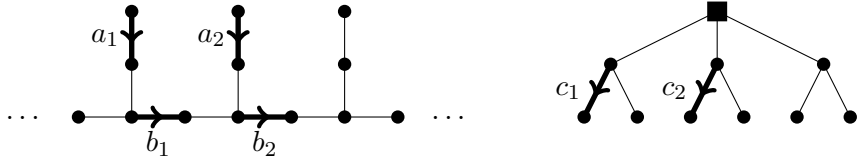
\begin{figure}
\caption{Elliptic versus hyperbolic pairs of arcs}
\label{fig:translation}
\begin{center}
\begin{tikzpicture}[scale=0.7, every loop/.style={}, square/.style={regular polygon,regular polygon sides=4}]

\tikzstyle{every node}=[circle,
                        inner sep=0pt, minimum width=5pt]
                        
\tikzset{->-/.style={decoration={
  markings,
  mark=at position .6 with {\arrow{>}}},postaction={decorate}}}

\draw (0,0) node[fill=black]{} -- (1,0)  node[fill=black]{} -- (2,0)  node[fill=black]{} -- (3,0)  node[fill=black]{} -- (4,0) node[fill=black]{} -- (5,0) node[fill=black]{}  -- (6,0) node[fill=black]{};

\draw{
(-1,0) node{$\dots$}
(1,2) node[fill=black]{} to (1,1) node[fill=black]{} to (1,0)
(3,2) node[fill=black]{} to (3,1) node[fill=black]{} to (3,0)
(5,2) node[fill=black]{} to (5,1) node[fill=black]{} to (5,0)
(7,0) node{$\dots$}
};

\draw[line width=2pt,->-] (1,2) -- (1,1);
\draw[line width=2pt,->-] (3,2) -- (3,1);
\draw[line width=2pt,->-] (1,0) -- (2,0);
\draw[line width=2pt,->-] (3,0) -- (4,0);
\draw{
(0.5,1.5) node{$a_1$}
(2.5,1.5) node{$a_2$}
(1.5,-0.5) node{$b_1$}
(3.5,-0.5) node{$b_2$}
};

\draw{
(12,2) node[square, draw=black, fill=black, minimum width=10pt]{}

(10,1) node[fill=black]{} to (12,2)
(12,1) node[fill=black]{} to (12,2)
(14,1) node[fill=black]{} to (12,2)

(9.5,0) node[fill=black]{} to (10,1)
(10.5,0) node[fill=black]{} to (10,1)
(11.5,0) node[fill=black]{} to (12,1)
(12.5,0) node[fill=black]{} to (12,1)
(13.5,0) node[fill=black]{} to (14,1)
(14.5,0) node[fill=black]{} to (14,1)
};   

\draw[line width=2pt,->-] (10,1) -- (9.5,0);
\draw[line width=2pt,->-] (12,1) -- (11.5,0);
\draw{
(9.2,0.5) node{$c_1$}
(11.2,0.5) node{$c_2$}
};

\end{tikzpicture}
\end{center}
\end{figure}

\begin{lem}\label{lem:axis_detection}
Let $T$ be a tree and let $g \in \Aut(T)$.  Given an arc $a = (x,y)$, we write $a^- =x$ and $a^+=y$.  Given arcs $a,b$ of $T$, we say that $(a,b)$ is \defbold{hyperbolic}\index{hyperbolic (pair of arcs)} if $\{a^-,a^+\} \neq \{b^-,b^+\}$ and $d(a^-,b^-) = d(a^+,b^+)$.  Otherwise, we say that $(a,b)$ is \defbold{elliptic}\index{elliptic (pair of arcs)}.  Then given an arc $a$ of $T$, the pair $(a,ga)$ is hyperbolic if $g$ is a translation and $a$ belongs to the axis of $g$; otherwise, $(a,ga)$ is elliptic.
\end{lem}

\begin{proof}
The lemma is a straightforward observation considering how the arcs are arranged relative to a fixed point of $g$, respectively, relative to the axis of $g$.  For example, in Figure~\ref{fig:translation}, the pair $(a_1,a_2)$ is elliptic, the pair $(b_1,b_2)$ is hyperbolic, and the pair $(c_1,c_2)$ is elliptic.
\end{proof}

Given a set $X$ of $k \ge 2$ ends, there is a unique smallest subtree $T_X$ containing representatives for each $\xi \in X$.  If $|X| = 2$ then $T_X$ is a line and for $|X|=3$ then $T_X$ has a unique vertex $v(X)$ of degree $3$ and all others of degree $2$.  Observe that, for a given vertex $v \in VT$ and triple of distinct ends $(\xi_1,\xi_2,\xi_3)$, then one has $v(\xi_1,\xi_2,\xi_3) = v$ if and only if $(\xi_1,\xi_2,\xi_3)$ is a $\sim_v$-distant triple.  For example, in Figure~\ref{fig:triple}, part of the $3$-regular tree is shown with the subtree $T_{\{\xi_1,\xi_2,\xi_3\}}$ marked with larger vertices and thicker edges.  Viewing the tree from the vertices $x,y,z \in VT$, we see that $(\xi_1,\xi_2,\xi_3)$ is a distant triple with respect to $\sim_x$, whereas we have $\xi_2 \sim_y \xi_3$ (but $\xi_1 \not\sim_y \xi_2$) and all three ends are $\sim_z$-equivalent.

\begin{figure}
\caption{Structure associated to a triple of ends}
\label{fig:triple}
\begin{center}
\begin{tikzpicture}[scale=0.7, every loop/.style={}, square/.style={regular polygon,regular polygon sides=4}]

\tikzstyle{every node}=[circle,
                        inner sep=0pt, minimum width=5pt]
          
\draw[line width=2pt,color=black] \foreach \x in {1,2,3}
{
(0,0) node[fill=black, minimum width=8pt]{} to ++ (-90+120*\x:1) node[fill=black, minimum width=8pt]{} to ++ (-90+120*\x:1) node[fill=black, minimum width=8pt]{} to ++ (-90+120*\x:1) node[fill=black, minimum width=8pt]{} ++(-90+120*\x:1.5) node[rotate=-90+120*\x]{$\ldots$} ++(-90+120*\x:1) node{$\xi_{\x}$}
};

\foreach \x in {1,2,3}{
\foreach \y in {1,2,3}{
\foreach \z in {-1,1}{
\draw (0,0) ++ (-90+120*\x:\y) to ++ (-180+120*\x:1) node[fill=black]{} to ++ (-180+120*\x+15*\z:1) node[fill=black]{};
}
}
};

\draw{
(0,0.5) node{$x$} ++ (30:1) node{$y$}
(0,0) ++ (30:1) ++ (-60:1) ++ (-75:1) ++ (0,-0.5) node{$z$}
};
\end{tikzpicture}
\end{center}
\end{figure}
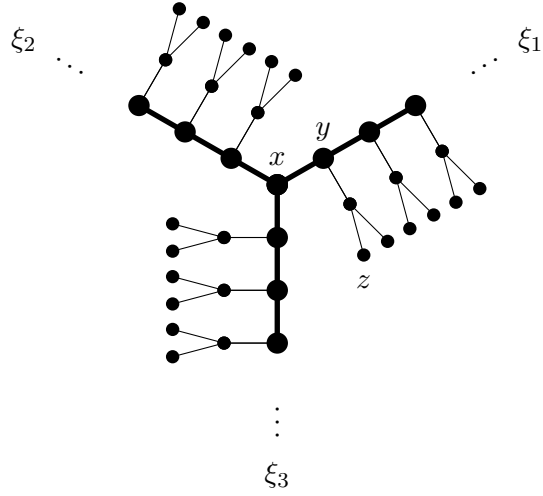

We now consider the action of a subgroup $G \le \Aut(T)$.  We say that $G$ is \defbold{$k$-type-distance-transitive} for $k \ge 0$ if given $x,y,x',y' \in VT$ such that $x'$ is the same type as $x$ and $d(x,y) = d(x',y') = k$, then there exists $g \in G$ such that $gx = x'$ and $gy = y'$.  In particular, $G$ is $0$-type-distance-transitive if and only if any two vertices of the same type are in the same $G$-orbit.  Since every path of length $k'$ is contained in a path of length $k$ for all $k \ge k'$, we see that if $G$ is $k$-type-distance-transitive then it is also $k'$-type-distance-transitive for $k' < k$.   We say that $G$ is \defbold{type-distance-transitive}\index{type-distance-transitive} if $G$ is $k$-type-distance-transitive for all $k$.  Note that $\Aut(T)$ itself is type-distance-transitive if and only if $T$ is regular or semiregular.  Write $\mc{H}_T$ for the class of closed boundary-$2$-transitive subgroups of $\Aut(T)$.

We note the following version of a well-known fact about groups acting on trees (actually an analogous statement is true for any connected bipartite graph), which shows one connection between the conditions we want to consider on pairs of vertices and transitivity of the local action.

\begin{lem}\label{lem:edge-walk}
Let $T$ be a tree and let $G \le \Aut(T)$.  Then the following are equivalent:
\begin{enumerate}[(i)]
\item $G^+$ is transitive on undirected edges of $T$;
\item For all $v \in VT$ and arcs $a,a'$ originating at $v$, there exists $g \in G(v)$ such that $ga = a'$;
\item $G$ is $1$-type-distance-transitive.
\end{enumerate}
\end{lem}

\begin{proof}
Suppose $G^+$ is transitive on undirected edges; let $v \in VT$ and let $a$ and $a'$ be arcs originating at vertices $v$ and $v'$ respectively, where $v$ and $v'$ are of the same type.  Then $\ol{a'}$ or $a'$ must be in the same orbit as $a$; however because $G^+$ is type-preserving, it cannot send $a$ to $\ol{a'}$.  Thus there is $g \in G$ such that $ga = a'$.  Thus (i) implies (ii).

Now suppose (ii) holds; note that for each $v \in VT$ the stabilizer $G(v)$ is type-preserving, so $G(v) = G^+(v)$.  Let $e$ be an undirected edge; let $T'$ be the subgraph generated by the $G^+$-orbit $G^+e$.  Clearly $T'$ has no isolated vertex.  Moreover, all undirected edges incident with a vertex $v \in VT$ lie in the same $G^+$-orbit; thus if $T'$ contains a vertex, it contains all edges incident with that vertex.  Since $T$ is connected we conclude that $T' = T$, showing that $G^+$ is transitive on undirected edges of $T$.  Thus (i) and (ii) are equivalent.

Given the definitions, we see that $G$ is $1$-type-distance-transitive if and only if $G^+$ is $1$-type-distance-transitive.  It is then clear that (i) and (iii) are equivalent.
\end{proof}

Under mild nondegeneracy assumptions, an action that is transitive on the boundary, in particular any $G \in \mc{H}_T$, is type-distance-transitive.  The conclusion is similar to part of \cite[Lemma~3.1.1]{BurgerMozes}, however that result assumes the group is closed and the tree is locally finite, and makes an appeal to the Baire category theorem.  For the next two propositions, we give purely combinatorial arguments, retaining the notation of Lemma~\ref{lem:axis_detection}.  We also give some other related implications that can be proved at this level of generality.

\begin{prop}\label{prop:boundary-trans}
Let $T$ be a thick tree and let $G \le \Aut(T)$.  Suppose that $G$ acts transitively on $\partial T$ and that $G$ does not fix any vertex or preserve any undirected edge; write $m$ for the minimum translation length occurring in $G$.  Then the following holds.
\begin{enumerate}[(i)]
\item Exactly one of the following holds:
\begin{enumerate}[(a)]
\item $m=1$ and $G$ is arc-transitive;
\item $m=2$ and the orbits of $G$ on $VT$ are exactly $V_0T$ and $V_1T$.
\end{enumerate}
\item Every finite path in $T$ is contained in the axis of some element of $G$ of translation length $m$.
\item $G$ is type-distance-transitive.
\end{enumerate}
\end{prop}

\begin{proof}
Since $G$ does not fix any vertex, end or undirected edge, by the standard classification of groups acting on trees (see \cite{Serre:trees}), $G$ contains a translation, so $m$ is well-defined.  By \cite[Lemma~2.1(iii)]{MollerVonk}, the union of the axes of translation of $G$ is a subtree $T'$ of $T$.  Since $G$ acts transitively on $\partial T$ we see that $\partial T' = \partial T$ and hence $T' = T$.  Thus given an arc $a$ of $T$, then $a$ lies along the axis $L$ of some translation $g \in G$; since $G$ is transitive on $\partial T$, we see that $L$ and $hL_0$ have a common end for some $h \in G$, and hence $g^na$ belongs to $hL_0$ for some $n \in \Zb$.  Thus every arc lies on some axis of translation length $m$.

Let $g_0 \in G$ have translation length $m$ and let $(x_i)_{i \in \Zb}$ be the vertices of the axis of $g_0$, such that $g_0x_i = x_{i+m}$.  Then we see that $VT = \bigcup^m_{i=1} Gx_i$ and that every arc of $T$ can be represented as $(gx_i,gx_{i+1})$ for some $g \in G$ and $0 \le i < m$.  Since $T$ is thick, for each vertex $x_i$ there is some neighbour $y_i$ other than $x_{i-1}$ and $x_{i+1}$, and then we have $(x_i,y_i)$ in the same $G$-orbit as $(x_j,x_{j+1})$ for some $j < i$, with the result that $((x_j,x_{j+1}),(x_i,y_i))$ is a hyperbolic pair of arcs, and hence belongs to some axis of translation by Lemma~\ref{lem:axis_detection}.  Taking $j$ maximal such that these conditions are satisfied, then $i-j \le m$, but also, given that $m$ is the minimum translation length of $G$, we see that $i-j \ge m$.  So in fact $i-j=m$, and we deduce that $y_i$ is in the same $G(x_i)$-orbit as $x_{i+1}$.  A similar argument shows that $y_i$ is in the same $G(x_i)$-orbit as $x_{i-1}$.  Thus $G(x_i)$ acts transitively on its neighbours for all $i \in \Zb$; since $VT = \bigcup^m_{i=1} Gx_i$ we deduce that $G$ is transitive on undirected edges.  Considering hyperbolic pairs of arcs and using Lemma~\ref{lem:axis_detection}, it is now easy to see that (i) holds.  We then observe that if $\{a,b\}$ is a pair of arcs whose initial vertices are of the same type, then $b = ga$ for some $g \in G$.

\begin{figure}
\caption{Embedding a finite path in an axis of minimal translation length}
\label{fig:tdt}
\begin{center}
\begin{tikzpicture}[scale=1, every loop/.style={}, square/.style={regular polygon,regular polygon sides=4}]

\tikzstyle{every node}=[circle,
                        inner sep=0pt, minimum width=5pt]
                        
\tikzset{->-/.style={decoration={
  markings,
  mark=at position .6 with {\arrow{>}}},postaction={decorate}}}

\draw[->-] (0,1) -- (0,0);
\draw[->-] (3,0) -- (4,0);
\draw[line width=2pt] (0,0) -- (3,0);

\draw[->] (0.5,2) to (0.5,0.5) to (1.5,0.5);
\draw[->] (6.5,-2) to (6.5,-0.5) to (7.5,-0.5);
\draw[->] (0.5,-2) to (0.5,-0.5) to (1.5,-0.5);

\draw{
(0,-2.5) node[rotate=90]{$\ldots$}
(7.5,0) node{$\ldots$}
(8,0) node{$\xi$}
(0,2.5) node[rotate=90]{$\ldots$}
(6,-2.5) node[rotate=90]{$\ldots$}

(-0.2,0.5) node{$a$}
(3.5,0.25) node{$b$}
(-0.25,-0.25) node{$x$}
(3,-0.4) node{$y$}

(0.9,1.5) node{$g$}
(6.9,-1.5) node{$h$}
(1.5,-1.5) node{$g^{-n}hg^n$}
};

\draw{
(0,2) node[draw=black, fill=white]{} -- (0,1) node[draw=black, fill=black]{} -- (0,0)
(6,-2) node[draw=black, fill=white]{} -- (6,-1) node[draw=black, fill=black]{} -- (6,0)
};
\draw (0,-2) node[draw=black, fill=white]{} -- (0,-1) node[draw=black, fill=black]{} -- (0,0) node[draw=black, fill=white, minimum width=8pt]{} --(1,0)  node[draw=black, fill=black, minimum width=8pt]{} -- (2,0)  node[draw=black, fill=white, minimum width=8pt]{} -- (3,0)  node[draw=black, fill=black, minimum width=8pt]{} -- (4,0)  node[draw=black, fill=white]{} -- (5,0)  node[draw=black, fill=black]{}  -- (6,0)  node[draw=black, fill=white]{} -- (7,0)  node[draw=black, fill=black]{};

\end{tikzpicture}
\end{center}
\end{figure}
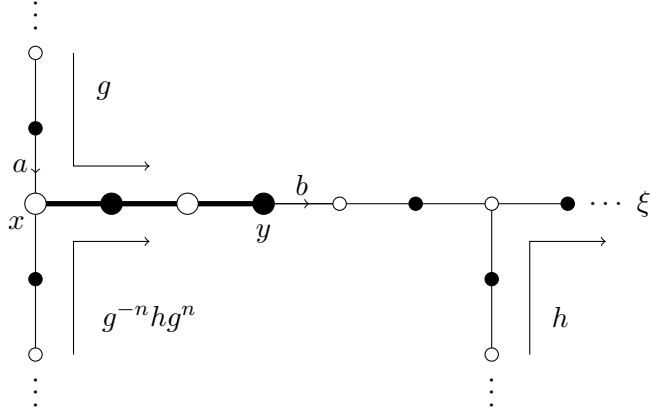

Consider now a pair of vertices $x,y \in VT$ of opposite types and let $P$ be the path from $x$ to $y$; see Figure~\ref{fig:tdt}, where the types of vertices are indicated as in Figure~\ref{fig:semiregular} and $P$ is marked with larger vertices and thicker edges.  Then there are arcs $a=(a_1,x)$ and $b = (y,b_2)$, such that $d(a_1,y) > d(x,y)$ and $d(x,b_2) > d(x,y)$.  We then see that $a_1$ and $y$ are of the same type, so $b = ga$ for some $g \in G$, and that $(a,b)$ is a hyperbolic pair of arcs.  By Lemma~\ref{lem:axis_detection}, the arcs $a$ and $ga$ both belong to the axis of $g$, so $x$ and $y$ lie on the axis of $g$; it then follows that $P$ lies on the axis of $g$.  Let $\xi$ be the attracting end of $g$ and let $h$ be a translation of length $m$; since $G$ acts transitively on $\partial T$, we can choose $h$ to have attracting end $\xi$.  Then for sufficiently large $n \ge 0$ we see that $g^nP$ is contained in the axis of $h$, and hence $P$ is contained in the axis of $g^{-n}hg^{n}$.  Thus all odd length paths are contained in an axis of translation length $m$; since a path of length $2n$ is contained in a path of length $2n+1$, in fact every finite path in $T$ is contained in some axis of translation length $m$, proving (ii).

It remains to show that $G$ is type-distance-transitive.  Let $x,y,x',y' \in VT$ such that $x'$ is the same type as $x$, $y'$ is the same type as $y$ and $d(x,y) = d(x',y')$.  Let $P$ be the oriented path from $x$ to $y$ and $P'$ the oriented path from $x'$ to $y'$.  Then $P$ lies on the axis of some translation $g$ of length $2$ and $P'$ lies on the axis of some translation $g'$ of length $2$.  After replacing $g$ or $g'$ with inverses as necessary, we may assume that $P$ is oriented towards the attracting end of $g$, and similarly for $g'$.  After conjugating in $G$ we may assume that $g$ and $g'$ have the same attracting end.  Then after applying a sufficiently large power of $n$, we find that $g^nP$ is on the axis of $g'$.  We now observe that the action of $g'$ on its axis is such that any two segments of the same type, length and orientation are in the same $\grp{g'}$-orbit.  Thus there exists $m$ such that $(g')^{-m}g^nP = P'$, proving (iii).
\end{proof}

\begin{prop}\label{prop:type-distance-trans}
Let $T$ be a thick tree and let $G \le \Aut(T)$.
\begin{enumerate}[(i)]
\item Suppose that for all $v \in VT$, the action of $G(v)$ on $\partial T$ is $2$-by-block-transitive relative to $\sim_v$.  Then $G$ is $2$-transitive on $\partial T$.
\item Suppose that $G$ is $2$-transitive on $\partial T$.  Then $G$ is type-distance-transitive and for every ray $\rho = (v_1,v_2,\dots)$ in $T$, the pointwise stabilizer of $\rho$ in $G$ acts transitively on $S_{v_2}(v_1,1)$.
\item Suppose that $G$ is $1$-type-distance-transitive, and suppose that for each type $t$ of vertex and some $k$, there exists a path $P = (v_1,v_2,\dots,v_k)$ in $T$ of length $k$ with $v_1$ of type $t$, such that the pointwise stabilizer of $P$ in $G$ acts transitively on $S_{v_2}(v_1,1)$.  Then $G$ is $k$-type-distance-transitive.
\item Suppose that $G$ is $0$-type-distance-transitive.  Then $G$ is $2n$-type-distance-transitive for some $n \ge 1$ if and only if, for all $v \in VT$, the action of $G(v)$ on $S(v,n)$ is $2$-by-block-transitive relative to the equivalence relation $\sim_{v,n}$.
\end{enumerate}
\end{prop}

\begin{proof}
(i)
Consider three distinct ends $\xi_1,\xi_2,\xi_3 \in \partial T$, and let $w = v(\xi_1,\xi_2,\xi_3)$.  Then we see that $\xi_1,\xi_2,\xi_3$ are in distinct $\sim_w$-classes.  Since the action of $G(w)$ is $2$-by-block-transitive, there is $g \in G(v)$ such that $g\xi_1 = \xi_1$ and $g\xi_2 = \xi_3$.  Thus $G$ is $2$-transitive on $\partial T$.

(ii)
Given $v \in VT$, it is clear that $G(v)$ preserves the equivalence relation $\sim_{v}$, so $G(v)$ is not $2$-transitive on $\partial T$, and hence $G$ does not fix an vertex; by a similar argument, $G$ does not preserve any edge.  Thus Proposition~\ref{prop:boundary-trans} applies; in particular, $G$ is type-distance-transitive and has some minimum translation length $m \in \{1,2\}$.  Let $H$ be the pointwise stabilizer of the given ray $\rho$.  Given $x \in S_{v_2}(v_1,1)$, we can extend $\rho$ to a bi-infinite line $L_x = (\dots,w_{-1},w_0,w_1,w_2,\dots)$ such that $w_0 = x$ and $w_i = v_i$ for $i \ge 1$.  Now let $x,y \in S_{v_2}(v_1,1)$; form the lines $L_x$ and $L_y$, and let $\xi \in \partial T$ be the end represented by $\rho$.  Since $G$ acts $2$-transitively on $\partial T$, there exists $g \in G(\xi)$ such that $gL_y = L_x$.  If $m=1$ we can write $gy = w_j$ for some $j \in \Zb$, which implies $gv_i = w_{i+j}$ for all $i \ge 1$.  If instead $m=2$, then by Proposition~\ref{prop:boundary-trans}(i), $G$ is type-preserving, so in fact $gy = w_{2j}$ for some $j \in \Zb$, and then $gv_i = w_{i+2j}$ for all $i \ge 1$.  In either case, $L_x$ is the axis of some translation $h \in G$ of translation length $m$, where say $hw_i = w_{i+m}$ for all $i \in \Zb$.  It then follows that $h^{-j}gy = w_0 = x$ and $h^{-j}gv_i = v_i$ for all $i \ge 1$; thus $h^{-j}g$ is an element of $H$ that sends $y$ to $x$, showing that $H$ is transitive on $S_{v_2}(v_1,1)$.

(iii)
Note first that if the given condition is true for the given $k$, then it is also true for $k'$ satisfying $2 \le k' < k$: we can simply truncate the path $P$.  We can therefore show that $G$ is $k$-type-distance-transitive by induction on $k$.  We may take $k \ge 2$ and assume that $G$ is $(k-1)$-type-distance-transitive.  Fix $t \in \{0,1\}$; we take the path $P = (v_1,\dots,v_k)$ as in the statement and choose some $v_0 \in S_{v_2}(v_1,1)$.  Now let $x,y \in VT$ with $d(x,y)=k$ and $x$ of type $1-t$; write the path from $x$ to $y$ as $(w_0,w_1,w_2,\dots,w_k)$ where $w_0=x$ and $w_k=y$.  Since $G$ is $(k-1)$-type-distance-transitive, there exists $g \in G$ such that $gw_i = v_i$ for $i \in \{1,k\}$; then in fact $gw_i = v_i$ for all $i \in \{1,\dots,k\}$.  In particular, $gx$ is some neighbour of $gw_1$ other than $v_2$, so $gx \in S_{v_2}(v_1,1)$.  By hypothesis, there is then an element $h$ of the pointwise stabilizer of $P$ such that $hgx = v_{t,0}$, and we also have $hgy = gy = v_{t,k}$, so we have mapped the pair $(x,y)$ to the pair $(v_{t,0},v_{t,k})$; note that the latter pair depends only on $t$ and $k$.  Thus $G$ is $k$-type-distance-transitive, and the conclusion follows by induction.

(iv)
Given $v \in VT$, it is clear that $G(v)$ preserves the equivalence relation $\sim_{v,n}$.  Suppose $G$ is $2n$-type-distance-transitive.  Then given two $\sim_{v,n}$-distant pairs $(x,y)$ and $(x',y')$ in $S(v,n)$, then $d(x,y) = d(x',y')=2n$ and all vertices have the same type; thus there is $g \in G$ such that $gx = x'$ and $gy = y'$.  We now observe that $v$ is the midpoint of the unique path from $x$ to $y$ and also the midpoint of the unique path from $x'$ to $y'$, so in fact $g \in G(v)$.  Thus the action of $G(v)$ on $S(v,n)$ is $2$-by-block-transitive.  Conversely, suppose $G(v)$ is $2$-by-block-transitive on $S(v,n)$ for all $v \in VT$, and let $x,y,x',y' \in VT$ such that $d(x,y) = d(x',y') = 2n$.  Let $v$ and $v'$ be the midpoint of the path from $x$ to $y$, respectively from $x'$ to $y'$.  Since $G$ is $0$-type-distance-transitive, there exists $g \in G$ such that $gv = v'$.  Then we see that $(gx,gy)$ and $(x',y')$ are both $\sim_{v',n}$-distant pairs in $S(v',n)$, so there exists $h \in G(v')$ such that $hgx = x'$ and $hgy = y'$.  Thus $G$ is $2n$-type-distance-transitive.
\end{proof}

If we assume the tree is locally finite and that $G$ is closed, we get a series of equivalent properties of the action.

\begin{cor}\label{cor:distance_transitivity}
Let $T$ be a thick locally finite tree and let $G$ be a closed subgroup of $\Aut(T)$.  Then the following are equivalent:
\begin{enumerate}[(i)]
\item $G$ acts transitively on $\partial T$ and does not fix any vertex or preserve any undirected edge;
\item For all $n \ge 1$ and $v \in VT$, the action of $G(v)$ on $S(v,n)$ is $2$-by-block-transitive relative to the equivalence relation $\sim_{v,n}$;
\item For all $v \in VT$, the action of $G(v)$ on $\partial T$ is $2$-by-block-transitive relative to the equivalence relation $\sim_{v}$;
\item $G \in \mc{H}_T$;
\item $G$ is $1$-type-distance-transitive, and for each type $t$ of vertex and each $k \ge 2$, there exists a path $P = (v_1,v_2,\dots,v_k)$ in $T$ of length $k$ with $v_1$ of type $t$, such that the pointwise stabilizer of $P$ in $G$ acts transitively on $S_{v_2}(v_1,1)$;
\item $G$ is type-distance-transitive.
\end{enumerate}
\end{cor}

\begin{proof}
The parts of Proposition~\ref{prop:type-distance-trans} give us the implications 
\[
\text{(iii)} \Rightarrow \text{(iv)} \Rightarrow \text{(v)} \Rightarrow \text{(vi)} \Rightarrow \text{(ii)}.
\]
Suppose (ii) holds; note that $G(v)$ is compact for all $v \in VT$.  Fix $v \in VT$; given $\xi \in \partial T$, write $(v_{\xi,0},v_{\xi,1},\dots)$ for the ray representing $\xi$ such that $v_{\xi,0} = v$.  Let $(\xi,\chi)$ and $(\xi',\chi')$ be two $\sim_v$-distant pairs of ends.  Then for each $n \ge 1$, we find that $(v_{\xi,n},v_{\chi,n})$ and $(v_{\xi',n},v_{\chi',n})$ are both $\sim_{v,n}$-distant pairs in $S(v,n)$, so by (ii) there is $g_n \in G(v)$ such that $g_nv_{\xi,n} = v_{\xi',n}$ and $g_nv_{\chi,n} = v_{\chi',n}$.  Taking $g \in G(v)$ to be the limit of any convergent subsequence of $(g_n)$, we find that $g\xi = \xi'$ and $g\chi = \chi'$, showing that $G(v)$ is $2$-by-block-transitive relative to $\sim_v$.  Thus (ii) $\Rightarrow$ (iii), so the statements (ii)--(vi) are all equivalent.

To finish the proof, the implication (iv) $\Rightarrow$ (i) is clear, and we have (i) $\Rightarrow$ (vi) by Proposition~\ref{prop:boundary-trans}.  Thus all six statements are equivalent.
\end{proof}

Let us now consider higher degrees of transitivity on $\partial T$.  We see that the trees $T_X$ are an obstacle to boundary-$4$-transitivity.

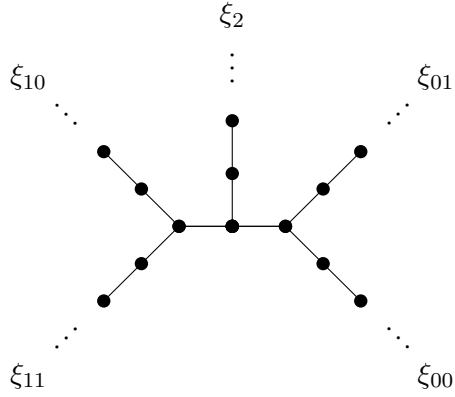
\begin{figure}
\caption{Five ends with quadruples belonging to different $\Aut(T)$-orbits}
\label{fig:quadruple}
\begin{center}
\begin{tikzpicture}[scale=0.7, every loop/.style={}, square/.style={regular polygon,regular polygon sides=4}]

\tikzstyle{every node}=[circle,
                        inner sep=0pt, minimum width=5pt]
                        
\foreach \x in {0,1}{
\foreach \y in {0,1}{
\draw 
{
(0,0) node[fill=black]{} to ++ (180*\x:1) node[fill=black]{} to ++ (180*\x-45+90*\y:1) node[fill=black]{} to ++ (180*\x-45+90*\y:1) node[fill=black]{} ++ (180*\x-45+90*\y:1) node[rotate=180*\x-45+90*\y]{$\ldots$} ++ (180*\x-45+90*\y:1) node{$\xi_{\x\y}$}
};
}
};

\draw 
{
(0,0) node[fill=black]{} to (0,1) node[fill=black]{} to (0,2) node[fill=black]{}
(0,3) node[rotate=90]{$\ldots$}
(0,4) node{$\xi_2$}
};

\end{tikzpicture}
\end{center}
\end{figure}

\begin{lem}\label{lem:4-trans}
Let $T$ be a thick tree, let $G \le \Aut(T)$ and let $k \ge 4$.  Then $G$ has infinitely many orbits of $k$-tuples of ends of $T$.
\end{lem}

\begin{proof}
Let $X = \{\xi_1,\xi_2,\dots,\xi_k\}$ be a set of $k$ distinct ends of $T$; then $T_X$ has finitely many vertices of degree $\ge 3$, and we can assign an invariant $d^*(X) \in \Nb$, which is the diameter in $T_X$ of the set of vertices of $T_X$ of degree $\ge 3$.  It is then clear that $d^*(X) = d^*(gX)$ for all $g \in G$, and we see that $d^*(X)$ can take unbounded values in $\Nb$ as $X$ ranges over all possible sets of $k$ ends; thus $G$ has infinitely many orbits on $k$-tuples of ends.  For example in Figure~\ref{fig:quadruple}, if $X = \{\xi_{ij} \mid i,j \in \{0,1\}\}$ and $Y = \{\xi_{00},\xi_{01},\xi_2,\xi_{10}\}$, then $d^*(X)=2$ but $d^*(Y)=1$, so there is no element of $\Aut(T)$ that maps $X$ to $Y$.
\end{proof}

We are thus left with the case of a group $G \le \Aut(T)$ acting $3$-transitively on the boundary of $T$.    Here is an analogue of Proposition~\ref{prop:type-distance-trans} and Corollary~\ref{cor:distance_transitivity} for boundary-$3$-transitive actions.

\begin{prop}\label{prop:3-trans}
Let $T$ be a thick tree and let $G \le \Aut(T)$.  Then the following are equivalent:
\begin{enumerate}[(i)]
\item $G$ acts $3$-transitively on $\partial T$;
\item $G$ acts vertex-transitively on $T$, and for each vertex $v \in VT$, then $G(v)$ is $3$-by-block-transitive on $\partial T$ relative to the equivalence relation $\sim_v$.
\end{enumerate}
\end{prop}

\begin{proof}
Suppose that $G$ acts $3$-transitively on $\partial T$.  We see that every vertex of $T$ is realized as $v(\xi_1,\xi_2,\xi_3)$ for some triple $(\xi_1,\xi_2,\xi_3)$ of distinct ends, so $G$ is vertex-transitive.  Fix $v \in VT$; since $v$ has at least $3$ neighbours, there are at least $3$ blocks of $\sim_v$.  Let $(\xi_1,\xi_2,\xi_3)$ and $(\xi'_1,\xi'_2,\xi'_3)$ be $\sim_v$-distant triples.  Then there is $g \in G$ such that $g\xi_i = \xi'_i$ for $1 \le i \le 3$.  Moreover, we see that
\[
v(\xi_1,\xi_2,\xi_3) = v = v(\xi'_1,\xi'_2,\xi'_3),
\]
so $gv = v'$.  Thus $G(v)$ is transitive on $\sim_v$-distant triples in $\partial T$.

Conversely, suppose that $G$ is vertex-transitive and that $G(v)$ is transitive on $\sim_v$-distant triples for all $v \in VT$.  Then given triples $(\xi_1,\xi_2,\xi_3)$ and $(\xi'_1,\xi'_2,\xi'_3)$ of distinct ends, with $v(\xi'_1,\xi'_2,\xi'_3) = v$ say, there is $g \in G$ such that
\[
v(g\xi_1,g\xi_2,g\xi_3) = gv(\xi_1,\xi_2,\xi_3) = v,
\]
and then $h \in G(v)$ such that $hg\xi_i = \xi'_i$ for $1 \le i \le 3$.  Thus $G$ is $3$-transitive on $\partial T$.
\end{proof}

For the remainder of the article, we will assume that $T$ is a thick locally finite semiregular tree, meaning that it is specified up to isomorphism by the vertex degrees $d_0,d_1 \ge 3$.  Let $\Omega_t$\index{Omega@$\Omega_t$} be a set of size $d_t$ for $t \in \{0,1\}$; depending on the group under consideration, it will often be convenient to regard $\Omega_t$ as the set underlying some additional structure, such as the points of a projective space.  Given $G \in \mc{H}_T$, since all vertices in $V_tT$ belong to the same $G$-orbit, there is a well-defined (up to permutational isomorphism) \defbold{local action of $G$ of type $t$}\index{local action}, which is the subgroup of $\Sym(\Omega_t)$ induced by the action of the point stabilizer $G(v)$ on $S(v,1)$, for any $v \in V_tT$.  Write $\mc{H}_{F_0,F_1}$ for the class of $G \in \mc{H}_T$ such that the local action of $G$ of type $t$ is the subgroup $F_t$\index{F@$F_t$} of $\Sym(d_t)$; note that we may assume $F_0$ and $F_1$ are $2$-transitive, since otherwise $\mc{H}_{F_0,F_1}$ will be empty.  We will write $F_t(\omega)$ for a point stabilizer of $F_t$ and regard it as a permutation group acting on $\Omega_t \setminus \{\omega\}$; in particular, taking a vertex $v$ of type $t$ neighbour $w$ of $v$, then $G(v,w)$ acts as $F_t(\omega)$ on $S_w(v,1)$.

Given a group $G \le \Aut(T)$, an \defbold{end local action} of $G$ is the action of $G(x,\xi)$ on $S_y(x,1)$, where $x$ is a vertex of $T$ and $(x,y)$ is the arc originating at $x$ that points towards the end $\xi$ of $T$.  In general, a closed subgroup $G \le \Aut(T)$ can have many end local actions at a vertex, depending on the choice of end.  However the situation is more homogeneous if $G \in \mc{H}_T$.

Let $G \le \Aut(T)$ with local action $F_t$ at each vertex of type $t$.  Let $x$ be a vertex of type $t$ and consider the action of $G(x,y)$ on $S_y(x,1)$, where $y \neq x$; write $k = d(x,y)$.  If $G$ is $k$-type-distance-transitive, then up to permutational isomorphisms of $F_t(\omega)$, the action of $G(x,y)$ on $S_y(x,1)$ is a subgroup of $F_t(\omega)$ that depends only on $k$ and the type $t$ of $x$; in this case we write $\Lambda^k_t(G)$\index{L@$\Lambda^k_t$} for this permutation group.   Note in particular that $\Lambda^1_t(G) = F_t(\omega)$.  (On the other hand, for $k \ge 2$, if $G$ fails to be $k$-type-distance-transitive, there is no guarantee that $\Lambda^k_t(G)$ is well-defined.)  The following is a variant of Proposition~\ref{prop:type-distance-trans}(iii).

\begin{cor}\label{cor:type-distance-trans}
Let $T$ be a thick tree and let $G \le \Aut(T)$.  Suppose for some $k \ge 1$ that $G$ is $k$-type-distance-transitive.  Then $G$ is $(k+1)$-type-distance-transitive if and only if $\Lambda^k_t(G)$ is transitive for all $t \in \{0,1\}$.
\end{cor}

\begin{proof}
If $\Lambda^k_t(G)$ is transitive for all $t \in \{0,1\}$, the argument of Proposition~\ref{prop:type-distance-trans}(iii) shows that $G$ is $(k+1)$-type-distance-transitive.  Conversely, suppose $G$ is $(k+1)$-type-distance-transitive.  We can take any path $P = (v_1,v_2,\dots,v_k)$ in $T$ of length $k$, and for any $x,y \in S_{v_2}(v_1,1)$, there exists $g \in G$ such that $gx = gy$ and $gv_k = v_k$.  We see that in fact every such element $g$ fixes $P$ pointwise, so in fact $G(P)$ acts transitively on $S_{v_2}(v_1,1)$.  Thus $\Lambda^k_t(G)$ is transitive.
\end{proof}

Now suppose $G \in \mc{H}_{F_0,F_1}$.  Then given Corollary~\ref{cor:distance_transitivity}, we can define $\Lambda^k_t(G)$ for all $k \ge 1$ and $t \in \{0,1\}$.  In turn, we have a normal subgroup $\Delta^k_t(G)$\index{D@$\Delta^k_t(G)$} of $\Lambda^k_t(G)$, which is the action of $G(x,y) \cap \bigcap_{z \in VT, d(y,z)=1}G(z)$ on $S_y(x,1)$; the groups $\Delta^k_t(G)$ will become important in Section~\ref{sec:main_section}.  It is clear that $\Lambda^{k+1}_t(G) \le \Lambda^k_t(G)$, so we can obtain the action of $G(x,\xi)$ on $S_y(x,1)$ for $x \in V_tT$ and $\xi \in \partial T_{(x,y)}$ as
\[
\Lambda_t(G) = \Lambda^\infty_t(G) := \bigcap_{k \ge 1} \Lambda^k_t(G);
\]
again, up to permutational isomorphisms of $F_t(\omega)$, the group $\Lambda_t(G)$ is independent of the choice of the pair $(x,\xi)$.\index{L@$\Lambda_t(G)$}  In Figure~\ref{fig:lambda} we illustrate these finite permutation groups using the convention from Definition~\ref{defn:tree}: part of the $(6,4)$-semiregular tree is shown (only vertices of interest displayed), with $x$ of type $0$, and we obtain an instance of $\Lambda^3_0(G)$, $\Delta^3_0(G)$ or $\Lambda_0(G)$ acting on the hollow circles by fixing the square vertices shown.

\begin{figure}\index{L@$\Lambda_t(G), \; \Lambda^k_t(G)$}\index{D@$\Delta^k_t(G)$}
\caption{The finite permutation groups $\Lambda^k_t(G)$, $\Delta^k_t(G)$ and $\Lambda_t(G)$}
\label{fig:lambda}
\begin{center}
\begin{tikzpicture}[scale=0.7, every loop/.style={}, square/.style={regular polygon,regular polygon sides=4}]

\tikzstyle{every node}=[circle,
                        inner sep=0pt, minimum width=5pt]
          
\draw \foreach \x in {0,1,2}
{
(\x,0) node[square, draw=black, fill=black, minimum width=10pt]{} to (\x+1,0) node[square, draw=black, fill=black, minimum width=10pt]{}
};
\draw \foreach \x in {-2,-1,0,1,2}
{
(3,0) to ++(30*\x:1.5) node[draw=black, fill=white, minimum width=8pt]{}
};
\draw{
(0,0.5) node{$y$}
(3,0.5) node{$x$}
(3.5,-2.5) node{$\Lambda^3_0(G)$}
};

\draw \foreach \x in {0,1,2}
{
(8+\x,0) node[square, draw=black, fill=black, minimum width=10pt]{} to (8+\x+1,0) node[square, draw=black, fill=black, minimum width=10pt]{}
};
\draw \foreach \x in {-2,-1,0,1,2}
{
(11,0) to ++(30*\x:1.5) node[draw=black, fill=white, minimum width=8pt]{}
};
\draw \foreach \x in {-1,0,1}
{
(8,0) to ++ (180+45*\x:1) node[square, draw=black, fill=black, minimum width=10pt]{}
};
\draw{
(8,0.5) node{$y$}
(11,0.5) node{$x$}
(11.5,-2.5) node{$\Delta^3_0(G)$}
};

\draw \foreach \x in {0,1,2,3}
{
(16+\x,0) node[square, draw=black, fill=black, minimum width=10pt]{} to (16+\x+1,0) node[square, draw=black, fill=black, minimum width=10pt]{}
};
\draw \foreach \x in {-2,-1,0,1,2}
{
(20,0) to ++(30*\x:1.5) node[draw=black, fill=white, minimum width=8pt]{}
};
\draw{
(14.5,0) node{$\xi$}
(15,0) node{$\ldots$}
(20,0.5) node{$x$}
(20.5,-2.5) node{$\Lambda_0(G)$}
};
\end{tikzpicture}
\end{center}
\end{figure}
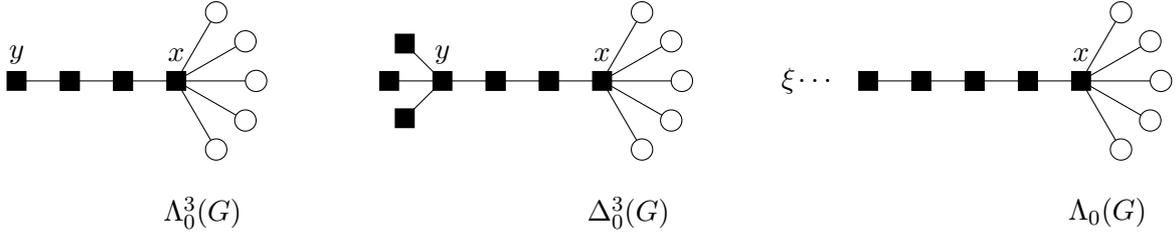

The following is an immediate consequence of Corollary~\ref{cor:distance_transitivity} (more specifically, the fact that (iv) implies (ii)).

\begin{cor}\label{cor:tree_2bbtrans}
Let $T$ be a thick locally finite tree, let $G \in \mc{H}_{F_0,F_1}$, let $t \in \{0,1\}$ and let $k \in \Nb \cup \{\infty\}$.  Then $F_t$ has block-faithful $2$-by-block-transitive action on $F_t/\Lambda^k_t(G)$, with block stabilizer $F_t(\omega)$.
\end{cor}

\section{Preliminaries on finite groups}

In this section and in Example~\ref{ex:exceptional} later, some calculations on individual finite groups were performed using the computer algebra package GAP (\cite{GAP}); we omit the details of these routine computations.  The main textbook used by the author for facts about finite permutation groups was \cite{DixonMortimer}.

\subsection{The soluble residual}

\begin{defn}
Given a (topological) group $G$ and $g,h \in G$, the \defbold{commutator}\index{G@$[g,h], [G,H]$} of $g$ and $h$ is the element $[g,h] := ghgh\inv$.  By convention, if $H$ and $K$ are subgroups of $G$, we define the \defbold{commutator group} $[H,K]:= \grp{[h,k] \mid h \in H, k \in K}$; note that elements of $[H,K]$ need not be of the form $[h,k]$ for $h \in H$ and $k \in K$.  In particular, the commutator $[G,G]$ is the \defbold{derived group}\index{derived group} of $G$; we say $G$ is \defbold{perfect} if $[G,G]=G$ and \defbold{topologically perfect}\index{p@(topologically) perfect} if $[G,G]$ is dense in $G$.  On the other hand, taking the derived series $G_0 = G$, $G_{i+1} = [G_i,G_i]$ for $i \ge 1$, we say $G$ is \defbold{soluble} if $G_n = \triv$ for some $n$ (the least $n$ for which this occurs is the \defbold{derived length} of $G$).

\index{prosoluble residual}\index{soluble residual}\index{soluble radical}\index{O@$O^\mc{C}(G),O_\mc{C}(G)$}\index{O@$O^\infty(G),O_\infty(G)$}Let $\mc{C}$ be a class of finite groups that is closed under subgroups, quotients and extensions.  Given a profinite group $G$ we write $O^\mc{C}(G)$ for the \defbold{pro-$\mc{C}$ residual}\index{pro-$\mc{C}$ residual} of $G$, that is, the intersection of all kernels of continuous homomorphisms of $G$ to groups in $\mc{C}$.  If $G$ is finite, then $G/O^\mc{C}(G)$ is the largest soluble quotient of $G$, whereas $O^\mc{C}(G)$ has no nontrivial quotient in $\mc{C}$; in this case $O^\mc{C}(G)$ is known as the \defbold{$\mc{C}$-residual}\index{$\mc{C}$-residual}; we also define the \defbold{$\mc{C}$-radical}\index{$\mc{C}$-radical} $O_\mc{C}(G)$, which is the largest normal subgroup of $G$ that belongs to $\mc{C}$.  Mostly we will be interested in the case that $\mc{C}$ is the class of finite soluble groups: in that case we define $O^\infty(G) := O^\mc{C}(G)$ (the \defbold{(pro-)soluble residual}) and $O_{\infty}(G) := O_{\mc{C}}(G)$ (the \defbold{soluble radical}).

A finite group $G$ has a composition series, that is, a subnormal series in which every factor is simple; the number of composition factors of a given isomorphism type does not depend on the choice of series.  Write $\kappa(G)$\index{k@$\kappa(G)$} for the number of nonabelian factors of $G$ appearing in a composition series; note that $\kappa(G)=0$ if and only if $G$ is soluble.  A \defbold{minimal normal subgroup}\index{minimal normal subgroup} of a finite group $G$ is a normal subgroup that is minimal among nontrivial normal subgroups of $G$; the \defbold{socle}\index{socle} is the group generated by all minimal normal subgroups of $G$.
\end{defn}

Note that if $\phi$ is a homomorphism with domain $G$ and $H,K \le G$, we have $\phi([h,k]) = [\phi(h),\phi(k)]$ for all $h,k \in G$, and hence $\phi([H,K]) = [\phi(H),\phi(K)]$.  We now make some observations about the relationship between composition factors and subnormal subgroups.

\begin{lem}\label{lem:minimally_covered}
Let $\phi: A \rightarrow B$ be a homomorphism of finite groups and let $P$ be a perfect subgroup of $B$ such that $\phi(A) \ge P$.  Then there is a unique smallest normal subgroup $N$ of $A$ such that $\phi(N) \ge P$; moreover $N$ is perfect.
\end{lem}

\begin{proof}
Let $\mc{N}$ be the set of normal subgroups $N$ of $H$ such that $\phi(A) \ge P$.  Given $N_1,N_2 \in \mc{N}$ and $g,h \in P$, then $\phi(g_1)=g$ and $\phi(h_2)=h$ for some $g_1 \in N_1$ and $h_2 \in N_2$.    Since $N_1$ and $N_2$ are normal in $H$, we then have $[g_1,h_2] \in [N_1,N_2] \le N_1 \cap N_2$, and clearly 
\[
\phi([g_1,h_2]) = [\phi(g_1),\phi(g_2)] = [g,h].
\]  Since $P$ is generated by commutators it follows that $\phi([N_1,N_2]) \ge P$, so $[N_1,N_2] \in \mc{N}$; in particular, any two elements of $\mc{N}$ have a common lower bound.  Since $A$ is finite, we deduce that $\mc{N}$ has a smallest element $N$ and that $N = [N,N]$.
\end{proof}

\begin{lem}\label{lem:minimally_covered:simple}
Let $G$ be a finite group with $\kappa(G)=1$, let $O^\infty(G) \le H \le G$ and let $N$ be a subnormal subgroup of $H$.  Then $N \ge O^\infty(G)$ or $N \le O_{\infty}(G)$.
\end{lem}

\begin{proof}
We see that $\kappa(H)=1$ and $O^{\infty}(G) = O^\infty(H)$, so we may assume $G = H$.

First consider the case that $N$ is normal in $G$ and $N$ does not contain $O^\infty(G)$.  Since $O^\infty(G)$ is perfect, we see that $O^\infty(G)N/N$ is nontrivial and perfect, so it accounts for at least one nonabelian composition factor of $G/N$.  Since $\kappa(G)=1$ it follows that $\kappa(N)=0$, that is, $N$ is soluble.  Since $N$ is normal in $G$ we deduce that $N \le O_{\infty}(G)$.

In the general case we proceed by induction on the subnormal depth of $N$.  By the induction hypothesis $N \unlhd N^{*}$ where $N^{*} \ge O^\infty(G)$ or $N^{*} \le O_{\infty}(G)$.  If $N^{*} \le O_{\infty}(G)$ then clearly also $N \le O_{\infty}(G)$; if $N^{*} \ge O^\infty(G)$, then $O^\infty(N^*) = O^\infty(G)$ and the conclusion follows by the previous paragraph.
\end{proof}

The next lemma is a generalization of \cite[Lemma~2.5]{Radu}.

\begin{lem}\label{lem:subdirect}
Let $B$ be a finite group such $\kappa(B)=1$ and form the direct product $\prod^n_{i=1}B_i$ of finitely many copies $B_i$ of $B$, with projection $\pi_i$ onto the $i$-th factor; form the semidirect product $\prod^n_{i=1}B_i \rtimes \Sym(n)$, with $\Sym(n)$ in its natural action permuting the copies of $B_i$, and let $\phi$ be the quotient map from $\prod^n_{i=1}B_i \rtimes \Sym(n)$ to $\Sym(n)$.  Write $O_i = O^\infty(B_i)$.  Let $A \le \prod^n_{i=1}B_i \rtimes \Sym(n)$ be such that $\phi(A)$ is transitive and let $H$ be a subgroup of $\prod^n_{i=1}B_i$ normalized by $A$ such that $\pi_i(H) \ge O_i$ for $1 \le i \le n$.

Then there is a $\phi(A)$-invariant equivalence relation $\sim$ on $\{1,\dots,n\}$ and perfect normal subgroups $M_1,\dots,M_n$ of $H$, with the following properties:
\begin{enumerate}[(i)]
\item Given a normal subgroup $N$ of $H$, then $\pi_i(N) \ge O_i$ if $N \ge M_i$ and $\pi_i(N) \le O_\infty(B_i)$ if $N \not\ge M_i$;
\item $M_i=M_j$ if and only if $i \sim j$;
\item $\pi_i(M_j) = O_i$ if $i \sim j$ and $\pi_i(M_j) = \triv$ if $i \not\sim j$;
\item For $1 \le i \le n$ we have $\kappa(M_i)=1$.
\end{enumerate}
\end{lem}

\begin{proof}
We define $M_i$ to be the unique smallest normal subgroup $N$ of $H$ such that $\pi_i(N) \ge O_i$, which exists and is perfect by Lemma~\ref{lem:minimally_covered}.  The minimality of $M_i$ ensures that $\pi_i(M_i) = O_i$.

To complete the proof of (i), consider $N \unlhd H$ such that $N \not\ge M_i$.  Then by the uniqueness of $M_i$, we have $\pi_i(N) \not\ge O_i$.  Since $\pi_i(H) \ge O_i$ and $N$ is normal in $H$, we see from Lemma~\ref{lem:minimally_covered:simple} that $\pi_i(N) \le O_\infty(B_i)$, proving (i).

Let $K_i = \pi\inv_i(O_\infty(B_i)) \cap H$ and define a relation $\sim$ on $\{1,\dots,n\}$ by setting $i \sim j$ if $O^\infty(K_i) = O^{\infty}(K_j)$; clearly $\sim$ is then an $A$-invariant equivalence relation.  We argue that $i \sim j$ if and only if $\pi_i(K_j)$ is soluble.  If $O^\infty(K_i) = O^{\infty}(K_j)$, then $\pi_i(O^{\infty}(K_j))$ is soluble, so $\pi_i(K_j)$ is also soluble.  Conversely if $\pi_i(K_j)$ is soluble then $O^{\infty}(K_j) \le K_i$, and then since $O^{\infty}(K_j)$ is perfect we have $O^{\infty}(K_j) \le O^{\infty}(K_i)$; but then $O^{\infty}(K_i)$ and $O^{\infty}(K_j)$ are finite groups conjugate under the action of $A$, so they have the same order and hence $O^{\infty}(K_i) = O^{\infty}(K_j)$.

Consider now $i,j$ such that $i \not\sim j$.  Then $\pi_i(K_j)$ is insoluble, so $K_j \ge M_i$ by (i), hence $\pi_j(M_i)$ is soluble; since $M_i$ is perfect, it follows that $\pi_j(M_i) = \triv$.  In particular, $M_i \neq M_j$.

Next, consider $i,j$ such that $i \sim j$.  Suppose that $\pi_i(M_j) \not\ge O_i$.  Then $\pi_i(N) \le O_\infty(O_i)$ by (i), so $M_j \le K_i$.  But then, since $M_j$ is perfect,
\[
M_j \le O^\infty(K_i) = O^{\infty}(K_j) \le K_j,
\]
so $\pi_j(M_j)$ is soluble, a contradiction.  Thus in fact $\pi_i(M_j) \ge O_i$, so $M_j \ge M_i$; by symmetry also $M_i \ge M_j$.  This completes the proof of (ii) and (iii).

We now know that $\pi_j(M_i \cap K_i)$ is soluble for all $j$: if $j \sim i$ then $\pi_j(K_i)$ is soluble, whereas if $j \not\sim i$ then $\pi_j(M_i)$ is soluble.  Thus $M_i \cap K_i$ is soluble, and hence $\kappa(M_i) = \kappa(\pi_i(M_i)) = 1$, proving (iv).
\end{proof}

\subsection{Finite $2$-transitive permutation groups}\label{sec:2trans}

Since projective linear groups and their relatives will appear in several places, we establish some terminology.

\begin{defn}\label{defn:linear}
Let $p$ be a prime, let $q = p^e$ for some natural number $e$, and let $V = \Fb^n_q$ be a finite-dimensional vector space with basis $\{v_1,\dots,v_n\}$ over the field $\Fb_q$ of $q$ elements.  The \defbold{general linear group} $\GL(V) = \GL_n(q)$\index{GL@$\GL$} consists of all bijective linear maps from $V$ to itself; there are normal subgroups $\SL(V)$\index{SL@$\SL$}, the \defbold{special linear group}, consisting of all such linear maps of determinant $1$, and $Z(V)$\index{Z@$Z(V)$}, the group of nonzero scalar maps from $V$ to $V$.  The \defbold{general semilinear group} $\GaL(V) = \GaL_n(q)$\index{GL@$\GaL$} is the group of permutations of $V$ generated by $\GL(V)$ and $f$, where $f$ acts on $V$ by sending $\sum^n_{i=1}\lambda_iv_i$ to $\sum^n_{i=1}\lambda^p_iv_i$.  In fact $\GaL_n(q)$ can be written as a semidirect product $\GaL_n(q) = \GL_n(q) \rtimes \grp{f}$, where $f$ has order $e$.

Given $G \le \GaL(V)$, the action of $G$ on $V$ induces an action of $G$ on the \defbold{projective space} $P(V) = P_n(q)$\index{P@$P(V)$}, which is the set of $1$-dimensional subspaces of $V$; we write $\mathrm{P}G$\index{PGL@$\PGaL, \; \PGL, \; \PSL$} for the group of permutations induced by $G$ on $P(V)$, and regard $\mathrm{P}G$ as the \defbold{projective} version of $G$ (so $\PGaL_n(q)$ is the projective general semilinear group, and so on).  It is a standard fact, which we will use without further comment, that the kernel of the action of $\GaL(V)$ on $P(V)$ is exactly $Z(V)$.

Given a subgroup $G$ of $\GaL(V)$ or $\PGaL(V)$, we define $e_G$\index{e@$e_G$} to be the index $|G:G \cap \GL(V)|$, respectively $|G:G \cap \PGL(V)|$.

A \defbold{finite permutation group of affine type}\index{affine type} is a subgroup of $\Sym(A)$ of the form $G = A \rtimes G_0$, where $A$ is a finite elementary abelian group acting on itself by multiplication, while $G_0$ acts faithfully on $A$ by conjugation.  Note that when $G$ is written in this form then $G_0$ is the stabilizer of the trivial element of $A$.  We interpret $A$ as the additive group of a finite vector space $V$; our convention is to choose $V$ to have as small a dimension as possible, subject to the existence of an isomorphism $(V,+) \cong A$ such that $G_0$ corresponds to a subgroup of $\GaL(V)$.  If there is no ambiguity, we can also write $G = V \rtimes G_0$ and $G_0 \le \GaL(V)$, and regard $G$ as a subgroup of $\Sym(V)$ where $V$ acts on itself by addition and $G_0$ acts as a subgroup of the standard action of $\GaL(V)$ on $V$.
\end{defn}

\begin{rem}\label{rem:linear}
Let $\omega \in P(V)$ and $L = \GaL_n(q)(\omega)$.  We regard $P(V/\omega)$ as an $L$-factor space of $P(V) \setminus \{\omega\}$ via the map
\[
P(V) \setminus \{\omega\} \rightarrow P(V/\omega);\; \omega' \mapsto \grp{\omega, \omega'}_{\Fb_q}/\omega.
 \]
 \end{rem}

\begin{table}[h!]
\begin{small}
\begin{center}

\begin{tabular}{ c | c | c | c | c}

Degree & $s$ & $S$ & $F_\mathrm{max}$ & $F_\mathrm{max}(\omega)$ \\ \hline
$d \ge 5$ & $d-2$ or $d$ & $\Alt(d)$ & $\Sym(d)$ & $\Sym(d-1)$ \\ 
$\frac{q^{n+1}-1}{q-1},n \ge 2$ & $2$ & $\PSL_{n+1}(q)$ & $\PGaL_{n+1}(q)$ & $(\Fb_q,+)^n \rtimes \GaL_n(q)$ \\
$2^{2n+1} -2^{n}$ & $2$ & $\mathrm{Sp}_{2n+2}(2)$ & $\mathrm{Sp}_{2n+2}(2)$ & $\Omega^-_{2n+2}(2).2$ \\
$2^{2n+1} +2^{n}$ & $2$ & $\mathrm{Sp}_{2n+2}(2)$ & $\mathrm{Sp}_{2n+2}(2)$ & $\Omega^+_{2n+2}(2).2$ \\ \hline
$11$ & $2$ & $\PSL_2(11)$ & $\PSL_2(11)$ & $\Alt(5)$ \\
$11$ & $4$ & $\mathrm{M}_{11}$ & $\mathrm{M}_{11}$ & $\Alt(6).2$ \\ 
$12$ & $3$ & $\mathrm{M}_{11}$ & $\mathrm{M}_{11}$ & $\PSL_2(11)$ \\ 
$12$ & $5$ & $\mathrm{M}_{12}$ & $\mathrm{M}_{12}$ & $\mathrm{M}_{11}$ \\ 
$15$ & $2$ & $\Alt(7)$ & $\Alt(7)$ & $\PSL_3(2)$ \\ 
$22$ & $3$ & $\mathrm{M}_{22}$ & $\mathrm{M}_{22}.2$ & $\PSL_3(4).2$ \\ 
$23$ & $4$ & $\mathrm{M}_{23}$ & $\mathrm{M}_{23}$ & $\mathrm{M}_{22}$ \\ 
$24$ & $5$ & $\mathrm{M}_{24}$ & $\mathrm{M}_{24}$ & $\mathrm{M}_{23}$ \\
$176$ & $2$ & $\mathrm{HS}$ & $\mathrm{HS}$ & $\mathrm{PSU}_3(5).2$ \\ 
$276$ & $2$ & $\mathrm{Co}_3$ & $\mathrm{Co}_3$ & $\mathrm{McL}.2$\\
\end{tabular}
\captionof{table}{Finite almost simple $2$-transitive permutation groups with insoluble point stabilizer\label{table:2trans_almost_simple}}
\end{center}
\end{small}
\end{table}

Now let $\Omega$ be a finite set and suppose $F \le \Sym(\Omega)$ is $2$-transitive.  The possibilities for $F$ are all known, in part as a consequence of the classification of finite simple groups; an overview can be found for example in \cite[\S7.7]{DixonMortimer}.  Every such group $F$ is either of affine type or has simple socle; independently, $F(\omega)$ could be soluble or insoluble.  For the purposes of this article, we will mainly be focused on the case that $F$ has simple socle and $F(\omega)$ is insoluble.  In Table~\ref{table:2trans_almost_simple}, the possible socles $S$ are listed (which is also $2$-transitive in all cases), including the largest group $F_\mathrm{max}$ with the given socle, the degree $s$ of transitivity of a group $F$ such that $S \le F \le F_\mathrm{max}$ and the structure of a point stabilizer of $F_\mathrm{max}$.  With socle $\PSL_{n+1}(q)$, the cases $(n,q) \in \{(2,2),(2,3)\}$ are exceptions; in these cases $F$ is still $2$-transitive, but it has soluble point stabilizer.

An \defbold{inner automorphism} of a group $G$ is an automorphism $c_g: x \mapsto gxg\inv$ for some $g \in G$.  The inner automorphisms form a normal subgroup $\mathrm{Inn}(G)$ of the automorphism group $\Aut(G)$; the quotient $\Out(G):=\Aut(G)/\Inn(G)$\index{Out@$\Out(G)$} is known as the \defbold{outer automorphism group}\index{outer automorphism group}.

The next lemma gives some facts we will need about $2$-transitive affine type groups.

\begin{lem}[{See for example \cite[Appendix~1]{LiebeckAffine}}]\label{lem:insoluble_affine}
Let $G = V \rtimes G_0$ be a finite $2$-transitive group of affine type with point stabilizer $G_0$, and suppose $G$ is insoluble.  We consider $V$ as a vector space of dimension $a$ over the field $q = p^e$, where $a$ is minimal such that $G_0 \le \GaL(V)$.  Let 
\[
Z = Z(V) \cap G_0, N = O^\infty(G_0) \text{ and } R = O_\infty(G_0).
\]
Then one of the following holds:
\begin{enumerate}[(a)]
\item $N = \SL_a(q)$;
\item $N = \mathrm{Sp}_a(q)$, $a \ge 4$ even, $(a,q) \neq (4,2)$, or $N = \Alt(6)$ of index $2$ in $\mathrm{Sp}_4(2)$;
\item $N = \mathrm{G}_2(q)$, $a=6$, $p=2$, $e > 1$, or $N = \mathrm{PSU}_3(3)$ of index $2$ in $\mathrm{G}_2(2)$;
\item $G_0 = N \in \{\Alt(7),\SL_2(13)\}$;
\item $N = \SL_2(5)$, $a=2$, $q \in \{9,11,19,29,59\}$;
\item $q=3,a=4$ and $N$ has the structure $2^{1+4}.\Alt(5)$, with $G_0 \le 2^{1+4}.\Sym(5)$ and $R = 2^{1+4}$.
\end{enumerate}
Moreover, in all cases $Z = \CC_{G_0}(N)$ and $NR/R$ is simple; except in case (f) we have $Z = R$.
\end{lem}

\begin{proof}
Since $G$ is insoluble but $V$ is abelian, the group $G_0$ must be insoluble.  Starting from Hering's Theorem as stated in \cite[Appendix~1]{LiebeckAffine} and eliminating the cases where $G_0$ is soluble, we see that $G_0$ has a perfect normal subgroup $N$ as given in one of the cases (a)--(f).  We claim that in all cases the action of $N$ on $P(V)$ is absolutely irreducible: that is, if we extend scalars in $V$ to the algebraic closure, $\ol{V} := V \otimes_{\Fb_q} \overline{\Fb_p}$, then there is no proper nonzero $\overline{\Fb_p}$-subspace $W$ of of $\ol{V}$ such that $P(W)$ is $N$-invariant.  This claim follows in cases (a) and (b) from \cite[Corollary~2.10.7]{KleidmanLiebeck} and in the remaining cases from the argument given in \cite[Appendix~1]{LiebeckAffine}.

It follows that the image of $N$ in $\PGaL(V)$ has trivial centralizer (see \cite[Lemmas~2.10.1 and~4.0.5]{KleidmanLiebeck}), so $Z = \CC_{G_0}(N)$.  In cases (a)--(e), we see that the image of $N$ in $\PGaL(V)$ is simple; it follows that the image of $G_0$ has trivial soluble radical, so $R = Z$.  In all cases including case (f), one sees that quotient $NR/R$ is simple.  In particular, the quotient $G_0R/NR$ embeds in $\Out(NR/R)$, which is known to be soluble in all cases, so $G_0/NR$ is soluble and hence $G_0/N$ is soluble; since $N$ is perfect we deduce that $N = O^\infty(G_0)$.
\end{proof}

The following fact will be useful for later arguments; in particular, it will let us apply the lemmas from the previous subsection.

\begin{lem}\label{lem:kappa_one}
Let $F$ be a $2$-transitive subgroup of $\Sym(\Omega)$ with $\Omega$ finite, let $\omega \in \Omega$ and let $O = O^\infty(F(\omega))$.  Then
\[
\kappa(F(\omega)) \le 1.
\]
Consequently, the group $S = O/O_\infty(O)$ is either trivial or a nonabelian simple group, and $\Out(S)$ is soluble.
\end{lem}

\begin{proof}
We may assume $F(\omega)$ is insoluble.  If $F$ is almost simple, we see that $\kappa(F(\omega)) = 1$ by inspecting Table~\ref{table:2trans_almost_simple}.

If instead $F$ is of affine type, then $F$ is as described in Lemma~\ref{lem:insoluble_affine}.  In particular, $G_0 = F(\omega)$ has a normal series
\[
\triv \unlhd R \unlhd NR \unlhd G_0
\]
such that $NR/R$ is a nonabelian simple group whereas $R$ and $G/NR$ are soluble, so $\kappa(G_0)=1$.

In particular, if $\kappa(F(\omega))=0$ then $S$ and $\Out(S)$ are trivial.  Otherwise, $S$ belongs to a known list of finite nonabelian simple groups and $\Out(S)$ is soluble.
\end{proof}

Given a finite $2$-transitive permutation group $F$ acting on a set $\Omega$, we write $\delta^*(F)$\index{d@$\delta^*(F)$} for the number of orbits of $O^\infty(F(\omega))$ acting on $\Omega \setminus \{\omega\}$.  For all finite $2$-transitive permutation groups, the following lemma gives the value of $\delta^*(F)$ and also whether or not the soluble radical of the point stabilizer is transitive.  The key exceptional case is that $F$ is as in case (e) of Lemma~\ref{lem:insoluble_affine}, which we call \defbold{exceptional $\SL_2(5)$ affine type}\index{exceptional $\SL_2(5)$ affine type}.

\begin{lem}\label{lem:soluble_residual_trans}
Let $F$ be a $2$-transitive subgroup of $\Sym(\Omega)$ and let $\omega \in \Omega$.
\begin{enumerate}[(i)]
\item If $F(\omega)$ is soluble, then $\delta^*(F) = |\Omega|-1$.
\item If $F$ is of exceptional $\SL_2(5)$ affine type, with socle of order $q^2$, then $\delta^*(F) = 2,1,3,7,29$ for $q =9,11,19,29,59$ respectively.
\item If $F(\omega)$ is insoluble but $F$ is not of exceptional $\SL_2(5)$ affine type, then $\delta^*(F)=1$.
\end{enumerate}
Moreover, $O_\infty(F(\omega))$ acts transitively on $\Omega \setminus \{\omega\}$ if and only if $F(\omega)$ is soluble.
\end{lem}

\begin{proof}
Let $R = O_\infty(F(\omega))$ and $N = O^\infty(F(\omega))$.  If $F(\omega)$ is soluble then $N$ is trivial, so (i) is clear, and also $F(\omega) = R$, so $R$ is transitive.  Thus we may assume $F(\omega)$ is insoluble.

If $\Omega$ is a projective space $P_n(q)$ with $\PSL_{n+1}(q) \le F \le \PGaL_{n+1}(q)$, then $n \ge 2$ and $(n,q) \not\in \{(2,2),(2,3)\}$, and we can extend $F$ by the scalar matrices to obtain $\widetilde{F}$ where $Z(V) \le \widetilde{F} \le \GaL_{n+1}(q)$, and similarly the subgroups $R$ and $N$ lift to subgroups $\widetilde{R}$ and $\widetilde{N}$ of $\GaL_{n+1}(q)$.  The stabilizer $\widetilde{F}(\omega)$ of the line $\omega$ then acts on the factor space $P(V/\omega)$ of $\Omega \setminus \{\omega\}$ (recall Remark~\ref{rem:linear}), where the action contains the simple group $\PSL(V/\omega)$; thus $\widetilde{F}(\omega)$ modulo the kernel of this action is almost simple.  We deduce that $\widetilde{R}$ acts trivially on $P(V/\omega)$, so $R$ does not act transitively on $\Omega \setminus \{\omega\}$.  Meanwhile, $\widetilde{N} = Z(V)W \rtimes \SL_n(q)$ where $W$ is the kernel of the action of $\GL(V)(\omega)$ on $(V/\omega) \oplus \omega$.  We can write $P(V) \setminus \{\omega\}$ as the union
\[
P(V) \setminus \{\omega\} = \bigcup_{\beta/\omega \in P(V/\omega)}\left(P(\beta) \setminus \{\omega\}\right).
\]
The copy of $\SL_n(q)$ is enough for $\widetilde{N}$ to act transitively on $P(V/\omega)$, and then given a $1$-dimensional subspace $\beta/\omega$ of $V/\omega$, then $W(\beta)$ acts transitively on $P(\beta) \setminus \{\omega\}$.  Thus $N$ acts transitively on $P(V) = \Omega \setminus \{\omega\}$.

If $F$ is almost simple, but not as in the previous paragraph, then we can inspect the possible point stabilizers in Table~\ref{table:2trans_almost_simple}.  In all cases we see that $R = \triv$, so certainly $R$ is not transitive.  In most cases, $|F(\omega):N|$ is coprime to $|\Omega \setminus \{\omega\}|$, so clearly $N$ acts transitively on $\Omega \setminus \{\omega\}$.  The exception is $F= \mathrm{M}_{11}$ acting on $11$ points and $N = \Alt(6)$, and in that case one can easily check directly that $N$ acts transitively on the remaining $10$ points.

We have now proved all the conclusions in the case that $F$ is almost simple.  Thus from now on we may assume that $F = V \rtimes G_0$ as in Lemma~\ref{lem:insoluble_affine}, and we can identify $\Omega \setminus \{\omega\}$ with the set of nontrivial elements of $V$, so that $F(\omega) = G_0$.

For the action of $N$: In cases (a)--(d) the group $V \rtimes N$ belongs to one of the standard families of $2$-transitive affine groups, so $N$ acts transitively on $\Omega \setminus \{\omega\}$; one can also check that $N = 2^{1+4}.\Alt(5)$ acts transitively on $\Omega \setminus \{\omega\}$ in case (f).  In case (e), the case of exceptional $\SL_2(5)$ affine type, then $N = \SL_2(5)$ acting on $\Fb^2_q$ for $q \in \{9,11,19,29,59\}$.  In these cases, $NZ(V)/Z(V)$ is one of the subgroups of $\PSL_2(q)$ isomorphic to $\Alt(5)$; these are all conjugate in $\PGL_2(q)$, see for example \cite[XII.259]{Dickson}.  The number of orbits of $N$ on $\Fb^2_q \setminus \{0\}$ is $\delta^* = 2,1,3,7,29$ for $q=9,11,19,29,59$ respectively, as can be checked by calculation.  This completes the proof of (ii) and (iii).

For the action of $R$: In cases (a)--(e) then $R$ acts by scalars on $V$, so it acts trivially on $P(V)$; since $P(V)$ is not a singleton and naturally occurs as a factor space of $\Omega \setminus \{\omega\}$, it follows that $R$ acts intransitively on $\Omega \setminus \{\omega\}$.  In case (f), we have $|R| = 32$ but $|\Omega \setminus \{\omega\}|=80$, so again $R$ cannot act transitively.
\end{proof}

\subsection{Block-faithful $2$-by-block-transitive actions}

We need to pay special attention to the finite $2$-transitive permutation groups where the action can be properly extended to a block-faithful $2$-by-block-transitive action.  Here we recall one of the main theorems from \cite{Reidkblock}.

\begin{thm}[{See \cite[Theorem~1.3 and Corollary~1.4]{Reidkblock}}]\label{thm:2bbtrans}
Let $G$ be a finite group with a faithful $2$-transitive action on the set $\Omega_0$, extending to a $2$-by-block-transitive action of $G$ on the set $\Omega = \Omega_0 \times B$, with block size $|B| \ge 2$; let $\omega \in \Omega$.  Then $G$ has a nonabelian simple socle $S$ and the stabilizer $G([\omega])$ of the block $[\omega]$ containing $\omega$ is the unique largest proper subgroup of $G$ that contains $G(\omega)$; write $W$ for the socle of $G([\omega])$.  If $S$ is of Lie type and naturally represented as a group of (projective) matrices of dimension $n+1$ over $\Fb_q$, then we can regard $G$ as a subgroup of $\PGaL_{n+1}(q)$.  Up to isomorphism of permutation groups, exactly one of the following is satisfied.
\begin{enumerate}[(a)]
\item $\PSL_{n+1}(q) \le G \le \PGaL_{n+1}(q)$, with the standard action on $\Omega_0 = P_n(q)$ and with $n \ge 2$, $q > 2$.  Let $\mu$ be a generator of $\Fb^*_q$ and let $D$ be the subgroup of $\Fb^*_q/\grp{\mu^{n+1}}$ generated by determinants of matrices representing elements of $G \cap \PGL_{n+1}(q)$.  In this case $G(\omega)$ contains the soluble residual of $G([\omega])$; in the case $(n,q)=(2,3)$, then $G(\omega)$ is of the form $W \rtimes \SL_2(3)$.  In addition, $e_{G(\omega)} = e_G$, and the block size $|B|$ divides $q-1$ and is coprime to $|D|$.
\item $\PSL_3(q) \le G \le \PGaL_3(q)$ and $G(\omega)$ is contained in a group of the form
\[
L^{\GaL_1} = G \cap (W \rtimes \GaL_1(q^2)), \text{ such that } |L^{\GaL_1}:G(\omega)| \le 2;
\]
in this case $|B| = |L^{\GaL_1}:G(\omega)|q(q-1)/2$.
\item $S$ is of rank $1$ simple Lie type and the action of $G$ on $\Omega_0$ is the standard $2$-transitive action.  In this case, $|B|$ divides $e_G$ and also divides the order of the multiplicative group of the field.
\item The action is one of eighteen exceptional $2$-by-block-transitive actions, listed in \cite[Table~1]{Reidkblock}.
\end{enumerate}
Moreover, in all cases $W \le G(\omega)$.
\end{thm}

Not all groups as in cases (b) and (c) admit proper block-faithful $2$-by-block-transitive actions; the precise conditions are given in \cite[Proposition~3.27]{Reidkblock} for case (b) and \cite[Proposition~3.21]{Reidkblock} for case (c).

Since these will come up later as special cases, in Table~\ref{table:small_q} we highlight the $2$-by-block-transitive actions with nontrivial blocks of $\mathrm{M}_{11}$ and with socle $\PSL_3(q)$ for $2 \le q \le 5$, as given in \cite[Tables~1 and 3]{Reidkblock}.  A blank entry should be read as a repeat of the entry above it.  In the last column it is indicated which case of Theorem~\ref{thm:2bbtrans} the action falls under.

\begin{table}[h!]
\begin{center}

\begin{tabular}{ c | c | c | c | c}

$G$ & $G(\omega)$ & $|\Omega_0|$ & $|B|$ & case \\ \hline
$\mathrm{M}_{11}$ & $\Alt(6)$ & $11$ & $2$ & (d) \\ \hline 
$\PGaL_3(2)$ & $W \rtimes C_3$ & $7$ & $2$ & (b) \\ \hline 
$\PGaL_3(3)$ & $W \rtimes \SL_2(3)$ & $13$ & $2$ & (a)  \\
 & $W \rtimes \GaL_1(9)$ & & $3$ & (b)  \\
 & $W \rtimes C_8$ & & $6$ &  \\ 
 & $W \rtimes Q_8$ & & $6$ & \\ \hline
$\PGL_3(4)$ & $W \rtimes (C_{15} \rtimes C_2)$ & $21$ & $6$  & (b) \\
 & $W \rtimes C_{15}$ & & $12$  & \\ \hline
$\PGaL_3(4)$ & $W \rtimes \GaL_1(16)$ & $21$ & $6$ & (b) \\ \hline
$\PGaL_3(5)$ & $W \rtimes \SL_2(5).C_2$ & $31$ & $2$ & (a) \\
 & $W \rtimes \SL_2(5)$ & & $4$ &  \\
 & $W \rtimes \GaL_1(25)$ & & $10$ & (b) \\
 & $W \rtimes C_{24}$ & & $20$ & \\
 & $W \rtimes (C_3 \rtimes C_8)$ & & $20$ & \\
 &  $W \rtimes (\SL_2(3) \rtimes C_4)$ & & $5$ & (d) \\
  &  $W \rtimes (\SL_2(3) \rtimes C_2)$ & & $10$ &  \\
  &  $W \rtimes \SL_2(3)$ & & $20$ & \\
\end{tabular}
\captionof{table}{Some $2$-by-block-transitive actions of special interest\label{table:small_q}}
\end{center}
\end{table}

Combining Corollary~\ref{cor:tree_2bbtrans} and Theorem~\ref{thm:2bbtrans} leads to the following restrictions on $\Lambda_t(G)$ for $G \in \mc{H}_T$, which we will build upon to prove the main results of this article.

\begin{lem}\label{lem:lambda_restriction}
Let $T$ be a thick locally finite tree, let $G \in \mc{H}_{F_0,F_1}$ and let $t \in \{0,1\}$.  Let $\Lambda = \Lambda_t(G)$.  Then one of the following holds, where unless otherwise stated, if $\PSL_{n+1}(q) \le F_t \le \PGaL_{n+1}(q)$ for $n \ge 1$ we have $\Omega_t = P_n(q)$:
\begin{enumerate}[(i)]
\item $\PSL_{n+1}(q) \le F_t \le \PGaL_{n+1}(q)$, with $n \ge 2$ and $q > 2$, and the action of $F_t$ on $F_t/\Lambda$ is as described in case (a) of Theorem~\ref{thm:2bbtrans};
\item $\PSL_3(q) \le F_t \le \PGaL_3(q)$ and the action of $F_t$ on $F_t/\Lambda$ is as described in case (b) of Theorem~\ref{thm:2bbtrans};
\item $F_t$ has socle of rank $1$ Lie type and the action of $F_t$ on $F_t/\Lambda$ is as described in case (c) of Theorem~\ref{thm:2bbtrans};
\item $(F_t,\Lambda,d_t) = (\mathrm{M}_{11},\Alt(6), 11)$;
\item $(F_t,\Lambda,d_t) = (\PSL_5(2),C^4_2 \rtimes \Alt(7), 31)$;
\item $\PSL_3(q) \le F_t \le \PGaL_3(q)$, the action of $F_t$ on $F_t/\Lambda$ is one of the remaining exceptional actions included in case (d) of Theorem~\ref{thm:2bbtrans}, and $O^\infty(F_t(\omega)) \nleq \Lambda$;
\item $\Lambda = F_t(\omega)$.
\end{enumerate}
Moreover, in all cases:
\begin{enumerate}[(a)]
\item $\Lambda$ contains the socle of $F_t(\omega)$;
\item $O_\infty(\Lambda)$ acts transitively on $\Omega_t \setminus \{\omega\}$ if and only if $\Lambda$ is soluble.
\end{enumerate}
\end{lem}

\begin{proof}
Let $W$ be the socle of $F_t(\omega)$.  As indicated in \cite[Table~1]{Reidkblock}, case (d) of Theorem~\ref{thm:2bbtrans} accounts for one exceptional $2$-by-block-transitive action each of $\mathrm{M}_{11}$ and $\PSL_5(2)$ (with point stabilizer $\Alt(6)$ and $C^4_2 \rtimes \Alt(7)$ respectively).  For the remaining $2$-by-block-transitive actions in \cite[Table~1]{Reidkblock}, the group satisfies $\PSL_3(q) \le G \le \PGaL_3(q)$ for some $q$ and the action is an extension of the action on projective space, such that the the point stabilizer does not contain the soluble residual of the block stabilizer.  The reduction to the cases listed and the fact that $\Lambda \ge W$ are now clear from Corollary~\ref{cor:tree_2bbtrans} and Theorem~\ref{thm:2bbtrans}.

It remains to prove that $O_\infty(\Lambda)$ acts transitively on $\Omega_t \setminus \{\omega\}$ if and only if $\Lambda$ is soluble.  We can eliminate cases (iv) and (v) by inspection.  It is clear that $\Lambda$ itself acts transitively on $\Omega_t \setminus \{\omega\}$, so if $\Lambda$ is soluble there is nothing to prove.  Thus we may assume $\Lambda$, hence also $F_t(\omega)$, is insoluble, which eliminates cases (ii) and (iii).

Write $O_t = O^\infty(F_t(\omega))$.  If $O_t \le \Lambda$, then by Lemma~\ref{lem:minimally_covered:simple} we have $O_\infty(\Lambda) \le O_{\infty}(F_t(\omega))$, and then $O_\infty(\Lambda)$ acts intransitively on $\Omega_t \setminus \{\omega\}$ by  Lemma~\ref{lem:soluble_residual_trans}.  Thus we may assume $O_t \not\le \Lambda$, which eliminates cases (i) and (vii).

The only remaining case is (vi); let $V = \Fb^3_q$ and regard $\omega$ as a line in $V$.  By \cite[Proposition~3.12]{Reidkblock}, after extending by the group of scalars to obtain $\widetilde{F} \le \GaL_3(q)$, then the lift $\widetilde{\Lambda}$ of $\Lambda$ takes the form $\widetilde{\Lambda} = W' \rtimes A$, where $W'$ acts trivially and $A$ acts transitively on nonzero vectors in $V/\omega$.  The structure of $A$ is thus subject to Lemma~\ref{lem:insoluble_affine}, and in fact falls under Lemma~\ref{lem:insoluble_affine}(e). In particular, $O_\infty(\widetilde{\Lambda})$ acts trivially on $P(V/\omega)$.  Consequently, $O_\infty(\Lambda)$ acts trivially on $P(V/\omega)$ and hence acts intransitively on $\Omega_t \setminus \{\omega\}$.
\end{proof}

Subsequent arguments will eliminate or restrict some cases of Lemma~\ref{lem:lambda_restriction}; see Remark~\ref{rem:lambda_restriction}.

\section{Rigid stabilizers and end local actions of boundary-$2$-transitive groups}\label{sec:main_section}

\subsection{Steps towards the local action of a rigid stabilizer of a half-tree}

In this section we will prove the main theorems of the article.  For the bulk of this first subsection (up to and including Proposition~\ref{prop:soluble_dichotomy}), we will not need to use details of the classification of finite $2$-transitive permutation groups, except that our approach makes use of the fact that point stabilizers have at most one nonabelian composition factor (Lemma~\ref{lem:kappa_one}).    We will state the next series of results in a form that is independent of the classification of finite $2$-transitive permutation groups.

For the rest of this section, we will write $G_k(x)$ for the pointwise stabilizer of $S(x,k)$, and $G_k(x_1,\dots,x_n) = \bigcap^n_{i=1}G_k(x_i)$\index{Gk@$G_k(x), \; G_k(x_1,\dots,x_n)$}.

Let us establish some more notation by way of a proposition.

\begin{prop}\label{prop:Hk_subnormal}\index{Theta@$\Theta_t, \; \Theta^{r,i}_t$}
Let $G \in \mc{H}_{F_0,F_1}$, acting on the tree $T$.  Let $x \in V_0T$.  For $v \in VT \setminus \{x\}$ and $i \ge 0$, picture the tree as rooted at $v$, and write $B^v_i$ for the set of vertices at distance $d(v,x)+i$ from $v$ that are not descendants of $x$.  Then set
\[
K^v_i = G(x) \cap \bigcap_{y \in B^v_i}G(y).
\]
Let $O_0 =  O^\infty(F_0(\omega))$ and write $\Theta^{r,i}_0(G)$ for the image of $K = K^v_i$ given by its action on $S_v(x,1)$, regarded as a subgroup of $F_0(\omega)$, where $r=d(x,v)$; see Figure~\ref{fig:theta}.  Then up to permutational automorphisms of $F_0(\omega)$, and for a given value of $d(v,x)$, then $\Theta^{r,i}_0(G)$ does not depend on the choice of $(v,x)$.  Moreover, the following conditions hold:
\begin{enumerate}[(i)]
\item We have $\Theta^{r,i}_0(G) \unlhd \Lambda^r_0(G)$ and $\Theta^{1,i}_0(G) \unlhd F_0(\omega)$.
\item Up to a permutational automorphism of $F_0(\omega)$, we can arrange that $\Theta^{r+1,i}_0(G) \unlhd \Theta^{r,i}_0(G)$.  In particular, $\Theta^{r,i}_0(G)$ is subnormal in $F_0(\omega)$ for all $r \ge 1$ and $i \ge 0$.
\item If $\Theta^{r,0}_0(G) \ge O_0$ for all $r \ge 1$, then $\Theta^{r,i}_0(G) \ge O_0$ for all $r \ge 1$ and $i \ge 0$.
\end{enumerate}
\end{prop}

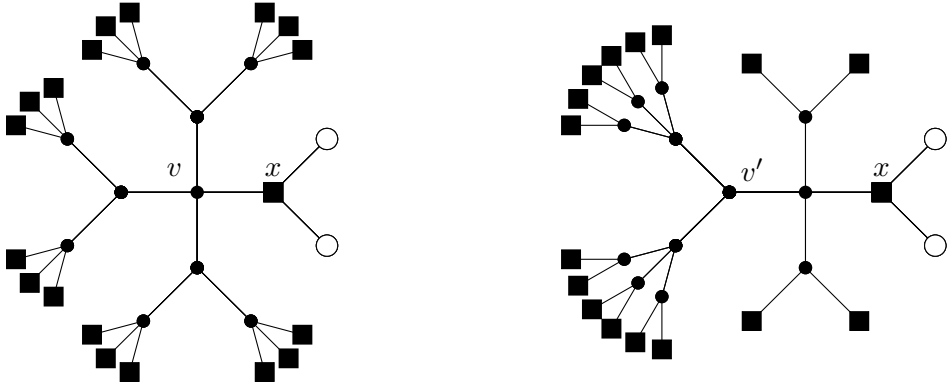
\begin{figure}
\caption{The groups $\Theta^{1,2}_0(G)$ (left) and $\Theta^{2,1}_0(G)$ (right)}
\label{fig:theta}
\begin{center}
\begin{tikzpicture}[scale=1, every loop/.style={}, square/.style={regular polygon,regular polygon sides=4}]

\tikzstyle{every node}=[circle,
                        inner sep=0pt, minimum width=5pt]

\foreach \x in {1,2,3} {
\foreach \y in {-1,1} {
\foreach \z in {-1,0,1} {
\draw (0,0) to ++ (90*\x:1) node[fill=black]{} to ++ (90*\x+45*\y:1) node[fill=black]{} to ++ (90*\x+45*\y+30*\z:0.7) node[square,draw=black, fill=black, minimum width=10pt]{};
}
}
};
                        
\foreach \x in {0} {
\foreach \y in {-1,1} {
\foreach \z in {-1,0,1} {
\draw (0,0) to ++ (90*\x:1) node[square,draw=black, fill=black, minimum width=10pt]{} to ++ (90*\x+45*\y:1) node[draw=black, fill=white, minimum width=8pt]{};
}
}
};

\draw{
(0,0) node[fill=black]{}
(-0.3,0.3) node{$v$}
(1,0.3) node{$x$}
};

\foreach \x in {1,3} {
\foreach \y in {-1,1} {
\draw (8,0) to ++ (90*\x:1) node[fill=black]{} to ++ (90*\x+45*\y:1) node[square,draw=black, fill=black, minimum width=10pt]{};
}
};

\foreach \x in {2} {
\foreach \y in {-1,1} {
\foreach \z in {-1,0,1} {
\foreach \w in {-1,1} {
\draw (8,0) to ++ (90*\x:1) node[fill=black]{} to ++ (90*\x+45*\y:1) node[fill=black]{} to ++ (90*\x+45*\y+30*\z:0.7) node[fill=black]{} to ++ (90*\x+45*\y+30*\z+15*\w:0.7) node[square,draw=black, fill=black, minimum width=10pt]{};
}
}
}
};
                        
\foreach \x in {0} {
\foreach \y in {-1,1} {
\foreach \z in {-1,0,1} {
\draw (8,0) to ++ (90*\x:1) node[square,draw=black, fill=black, minimum width=10pt]{} to ++ (90*\x+45*\y:1) node[draw=black, fill=white, minimum width=8pt]{};
}
}
};

\draw{
(8,0) node[fill=black]{}
(7.3,0.3) node{$v'$}
(9,0.3) node{$x$}
};
     
\end{tikzpicture}
\end{center}
\end{figure}

\begin{proof}
The independence of $\Theta^{r,i}_0(G)$ up to permutational automorphisms of $F_0(\omega)$ follows by the fact that $G$ acts transitively on vertices of a given type and that for each vertex $v$, $G(v)$ acts transitively on $S(v,r)$.

Note that for fixed $(v,x)$, the set $B^v_i$ is invariant under the action of $G(v,x)$, which acts as $\Lambda^r_0(G)$ on $S_v(x,1)$.  As a result, we see that $\Theta^{r,i}_0(G) \unlhd \Lambda^r_0(G)$ for all $i \ge 0$.  Since $\Lambda^1_0(G) = F_0(\omega)$, we deduce that $\Theta^{1,i}_0(G) \unlhd F_0(\omega)$, proving (i).

Let $\overline{B}^v_i$ be the smallest subtree containing $B^v_i \cup \{x\}$; we see that in fact $K^v_i$ fixes $\overline{B}^v_i$ pointwise.  For fixed $i$, the fact that $\Theta^{r+1,i}_0(G) \unlhd \Theta^{r,i}_0(G)$ follows by observing that if we take a neighbour $w$ of $v$ such that $d(w,x)=d(v,x)+1$, then $w$ is fixed by $K^v_i$, so $K^v_i$ stabilizes $B^w_i$ setwise, and moreover $\overline{B}^w_i$ also contains $B^v_i$.  The fact that $\Theta^{r,i}_0(G)$ is subnormal in $F_0(\omega)$ now follows using part (i) and induction on $r$, completing the proof of (ii).

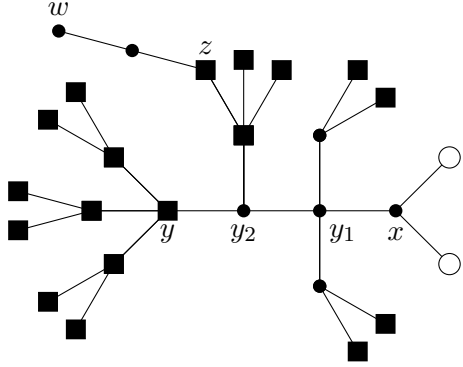
\begin{figure}
\caption{The set $\bigcup^k_{j=1}X_j$ (black squares) and a vertex $w \in B^y_{i+1}$ for $k=3,i=1$}
\label{fig:theta_extend}
\begin{center}
\begin{tikzpicture}[scale=1, every loop/.style={}, square/.style={regular polygon,regular polygon sides=4}]

\tikzstyle{every node}=[circle,
                        inner sep=0pt, minimum width=5pt]
                        
\draw (0,0) node[square,draw=black, fill=black, minimum width=10pt]{} to (1,0) node[fill=black]{} to (2,0) node[fill=black]{} to (3,0) node[fill=black]{};

\foreach \x in {-1,1} {
\draw (3,0) to ++ (45*\x:1) node[draw=black, fill=white, minimum width=8pt]{};
};

\foreach \x in {-1,1} {
\foreach \y in {1,2} {
\draw (2,0) to ++ (90*\x:1) node[fill=black]{} to ++ (30*\x*\y:1) node[square,draw=black, fill=black, minimum width=10pt]{};
}
};

\foreach \y in {-1,0,1} {
\draw (1,0) to ++ (90:1) node[square,draw=black, fill=black, minimum width=10pt]{} to ++ (90+30*\y:1) node[square,draw=black, fill=black, minimum width=10pt]{};
};
                     
\foreach \x in {-1,0,1} {
\foreach \y in {-1,1} {
\draw (0,0) to ++ (180+45*\x:1) node[square,draw=black, fill=black, minimum width=10pt]{} to ++ (180+45*\x+15*\y:1) node[square,draw=black, fill=black, minimum width=10pt]{};
}
};

\draw (1,0) to ++ (90:1) to ++ (120:1) to ++ (165:1) node[fill=black]{} to ++ (165:1) node[fill=black]{} ++ (90:0.3) node{$w$};
\draw (1,0) to ++ (90:1) ++ (120:1) ++ (90:0.3) node{$z$};

\draw{
(3,-0.3) node{$x$}
(0,-0.3) node{$y$}
(1,-0.3) node{$y_2$}
(2.3,-0.3) node{$y_1$}
};     
\end{tikzpicture}
\end{center}
\end{figure}

For (iii), we proceed by induction on $i$.  Assume $\Theta^{r,i}_0(G) \ge O_0$ for all $r \ge 1$ and some $i \ge 0$, and let $k \ge 1$.  Fix a vertex $x$ of type $0$ and a vertex $y$ such that $d(x,y)=k$.  Suppose that the vertices on the path from $x$ to $y$ are, in order, $x=y_0,y_1,\dots,y_k=y$.  Notice that for all $v \in VT \setminus \{x\}$, the group $K^v_{i}$ has local action at $x$ containing $\Theta^{r,i}_0(G)$ for $r = d(v,x)$, and hence containing $O_0$.

For $1 \le j \le k$, let $X_j = S_x(y_j,i+1)$; an example of a set of the form $\bigcup^k_{j=1}X_j$ is given in Figure~\ref{fig:theta_extend}.  Then given $z \in X_j$ we have $d(z,x) = i+j+1$ and $K^z_{i}$ has local action containing $O_0$ at $x$.  Given $z,z' \in X_j$ we have $d(z,z') \le 2i+2 \le (i+j+1)+i$, so $K^z_{i}$ fixes $z'$ and \textit{vice versa}.  In particular, for each $z \in X_j$, then $K^z_{i}$ is normal in the group $M_j = \grp{K^z_i \mid z \in X_j}$.  Let $L_j = \bigcap_{z \in X_j}K^z_i$.  Then by Lemma~\ref{lem:minimally_covered}, we see that $L_j$ has local action at $x$ containing $O_0$.  We see moreover that the subgroups $L_j$ are all normal in $G(x,y)$, so by Lemma~\ref{lem:minimally_covered} again, the group $L = \bigcap L_j$ has local action containing $O_0$ at $x$.  We now aim to show that $L$ fixes $B^y_{i+1}$ pointwise: note that $B^y_{i+1}$ consists of all vertices $w$ such that $d(w,y)=k+i+1$ and the path from $w$ to $y$ does not pass through $x$.  Given $w \in B^y_{i+1}$, then the path from $w$ to $y$ meets the path from $x$ to $y$ at some vertex $y_j$ for $1 \le j \le k$ ($j=2$ in the example shown in Figure~\ref{fig:theta_extend}),
so that
\[
k+i+1 = d(w,y_j)+d(y_j,y)=d(w,y_j)+k-j; \text{ hence } d(w,y_j) = i+j+1.
\]
Taking a vertex $z$ on the path from $y_j$ to $w$ at distance $i+1$ from $y_j$, we see that $z \in X_j$ and $d(w,z)=j \le d(z,x)+i$.  Thus $w \in \overline{B}^z_{i}$, so $w$ is fixed by $L$.  Thus $L$ fixes $B^y_{i+1}$ pointwise, from which we deduce that $\Theta^{k,i+1}_0(G) \ge O_0$.  Since $k \ge 1$ was arbitrary, in fact $\Theta^{k,i+1}_0(G) \ge O_0$ for all $k \ge 1$ and the inductive step is complete.
\end{proof}

We define $\Theta^{r,i}_1(G)$ similarly to $\Theta^{r,i}_0(G)$.  Given that $\Theta^{r,i}_t(G)$ is weakly decreasing as $r$ increases (for fixed $i$) or $i$ increases (for fixed $r$), together with the fact that all the groups here are subgroups of a fixed finite group $F_t(\omega)$, we find that $\Theta^{r,i}_t(G)$ is eventually constant as $r \rightarrow \infty$ or $i \rightarrow \infty$.  We define limiting subgroups $\Theta^{\infty,i}_t(G)$, $\Theta^{r,\infty}_t(G)$ and $\Theta^{\infty,\infty}_t(G)$\index{Theta@$\Theta_t, \; \Theta^{r,i}_t$}, which are again defined up to permutational automorphisms of $F_t(\omega)$.

Where the group $G \in \mc{H}_T$ is obvious from context, from now on we will simply write $\Theta^{r,i}_t$\index{Theta@$\Theta_t, \; \Theta^{r,i}_t$} for $\Theta^{r,i}_t(G)$, and similarly for $\Delta^k_t$\index{D@$\Delta^k_t(G)$} and $\Lambda^k_t$\index{L@$\Lambda_t(G), \; \Lambda^k_t(G)$}.

Note that in Proposition~\ref{prop:Hk_subnormal}, for any fixed vertex $v \in VT \setminus \{x\}$ then $K = \bigcap_{i \ge 0} K^v_i$ is the pointwise stabilizer of the half-tree $T_{(x,y)}$, where $y$ is the first vertex after $x$ in the path from $x$ to $v$, or equivalently $K$ is the rigid stabilizer of the half-tree $T_{(y,x)}$.  Thus $\Theta^{r,\infty}_0$ is the local action $\Theta_0$ of the rigid stabilizer of a half-tree at the root of the half-tree; since $\Theta^{r,\infty}_0$ does not depend on $r$, we also have $\Theta^{\infty,\infty}_0 = \Theta_0$.  Similar observations apply to $\Theta_1$.  Proposition~\ref{prop:Hk_subnormal}(iii) can then be restated more succinctly as follows:

\begin{cor}\label{cor:Hk_subnormal}
Let $G \in \mc{H}_{F_0,F_1}$.  If $\Theta^{\infty,0}_0 \ge O^\infty(F_0(\omega))$, then $\Theta_0 \ge O^\infty(F_0(\omega))$.
\end{cor}

We recall that $\Delta^k_t(G)$\index{D@$\Delta^k_t(G)$} is the action of $G(x) \cap G_1(y)$ on $S_y(x,1)$, where $x$ and $y$ are vertices such that $d(x,y) = k  \ge 1$ and $x$ has type $t$.  Note that $\Delta^k_t(G) \unlhd \Lambda^k_t(G)$ and $\Delta^k_t(G) \le \Lambda^{k+1}_t(G)$.

The next lemma will be used in a special case in this subsection and then in a more significant way in Section~\ref{sec:line}, where we investigate the structure of the groups $\Lambda^k_t$ in more detail.  A \defbold{subdirect product}\index{subdirect product} of two groups $A$ and $B$ is a subgroup $C$ of $A \times B$ such that 
\[
(A \times \triv)C = C(\triv \times B) = A \times B.
\]

\begin{lem}\label{lem:Goursat}
Let $G \in \mc{H}_T$ and let $x,y \in VT$ where $x$ is of type $t$ and $d(x,y) = k > 0$.  Let $N_{y,x}$ and $N_{x,y}$ be the kernels of the action of $G(x,y)$ on $S_y(x,1)$ and $S_x(y,1)$ respectively and let $N = N_{x,y}N_{y,x}$.  If $k$ is odd, we regard the quotient $G(x,y)/N$ as a copy of a group $A_k = A_{k,0} = A_{k,1}$ that depends only on $G$ and $k$; if $k$ is even, we regard the quotient $G(x,y)/N$ as a copy of a group $A_{k,t}$ that depends on $(G,k,t)$.

Then we have surjective homomorphisms $\psi_{k,t}: \Lambda^k_t \rightarrow A_{k,t}$ with kernel $\Delta^k_t$, such that $G(x,y)$ acts on $S_y(x,1) \times S_x(y,1)$ as a subdirect product
\[
M_{x,y} = \{(a,b) \in \Lambda^k_t \times \Lambda^k_{t+k} \mid \psi_{k,t}(a) = \psi_{k,t+k}(b)\}.
\] 
\end{lem}

\begin{proof}
The fact that up to isomorphism, the groups $A_{k,t}$ are independent of the choice of vertices $(x,y)$ follows from the fact that $G$ is type-distance-transitive.

The action of $G(x,y)$ on $S_y(x,1) \times S_x(y,1)$ yields a homomorphism $\pi$ from $G(x,y)$ to $\Lambda^k_t \times \Lambda^k_{t+k}$ with kernel $N_{x,y} \cap N_{y,x}$; write $M_{x,y} = \pi(G(x,y))$.  The description of $M_{x,y}$ as a subdirect product of $\Lambda^k_t$ and $\Lambda^k_{t+k}$ is an instance of Goursat's lemma; we can also see the argument as follows.    If we take the action of $G(x,y)$ on $S_y(x,1)$, we have a surjective homomorphism $\pi_{x,y}$ from $G(x,y)$ to $\Lambda^k_t$ with kernel $N_{y,x}$.  From the definition of $\Delta^k_t$, we see that the image of $N_{x,y}$ in $\Lambda^k_t$, and hence also of $N$, is exactly $\Delta^k_t$.  The same description applies to $\Lambda^k_{t+k}$ with the roles of $x$ and $y$ reversed.  Hence
\[
\frac{\Lambda^k_t}{\Delta^k_t} \cong \frac{G(x,y)}{N} \cong A_{k,t} \cong \frac{\Lambda^k_{t+k}}{\Delta^k_{t+k}}. \qedhere
\]
\end{proof}

At this point, we have enough of a general setup to construct examples; these are in their own section (Section~\ref{sec:examples}) so as not to interrupt the flow of progress towards the main theorems, but the reader may find it informative to go back and forth between the present section and the examples section as motivation for some of the steps in the proof.

We now proceed by a series of lemmas on the groups $\Theta^{k,i}_t$, culminating in a set of restrictions on the situation where one has $\Theta_t \ngeq O_t$.

\begin{lem}\label{lem:action_to_rigid_action}
Let $G \in \mc{H}_{F_0,F_1}$, with $\kappa(F_0(\omega)) = 1$.  Write $O_0 := O^\infty(F_0(\omega))$ and $R_0 = O_\infty(O_0)$, and suppose that $\Theta^{k,0}_0 \ge O_0$ for some $k$.  We also make one of the following assumptions:
\begin{enumerate}[(a)]
\item We have $k =1$; $F_1$ is not isomorphic to any subquotient of $\Out(O_0/R_0)$; and $O_\infty(F_0(\omega))$ is intransitive.
\item We have $k \ge 2$ and $\Theta^{k-1,1}_0 \ge O_0$.
\end{enumerate}
Then $\Theta^{k,1}_0 \ge O_0$.
\end{lem}

\begin{proof}
Let $x$ be a vertex of type $k$ and consider the action of $G_k(x)$ on $S(x,k+1)$.  By hypothesis, $\Theta^{k,0}_0 \ge O_0$, which ensures that the image of $G_k(x)$ in the action on $S_x(y,1)$ contains $O_0$ for each $y \in S(x,k)$.  We can thus apply Lemma~\ref{lem:subdirect} where $A = G(x)/G_{k+1}(x)$ embedded in $\prod_{y \in S(x,k)}B_y \rtimes \Sym(S(x,k))$, each $B_y$ is a copy of $F_0(\omega)$, and $H = G_k(x)/G_{k+1}(x)$.  We obtain a $G(x)$-invariant equivalence relation $\sim$ on $S(x,k)$, and for each $y \in S(x,k)$ a perfect normal subgroup $M_y$ of $H$ with $\kappa(M_y)=1$, such that $M_y$ acts as $O_0$ on $S_x(y,1)$, where $M_{y} = M_{y'}$ if and only if $y \sim y'$, and where $M_y$ acts trivially on $S_x(y',1)$ if $y' \not\sim y$.  To show that $\Theta^{k,1}_0 \ge O_0$, it is now enough to show that $\sim$ is the trivial equivalence relation.  Note that for all $k' \ge 1$, since $G(x)$ has $2$-by-block-transitive action on $S(x,k')$, every nonuniversal $G(x)$-invariant equivalence relation $\sim'$ on $S(x,k')$ must satisfy $\sim' \subseteq \sim_{x,k'}$.  In particular, if $\sim$ is nontrivial then there exist $y',y'' \in S(x,k)$ such that $y' \sim y''$ and $0 < d(y',y'') < 2k$.

In case (a), we have $k=1$.  Since $G(x)$ acts primitively on $S(x,1)$, the only possibilities are that $\sim$ is trivial or universal, so we may assume for a contradiction that $\sim$ is the universal relation on $S(x,1)$.  Then $M := M_y$ is constant as a function of $y \in S(x,1)$, and hence $M$ is a subgroup of $H$ normalized by $A$.  Let $C$ be the centralizer of $M/O_\infty(M)$ in $A$.  Then $A/MC$ is isomorphic to a subgroup of $\Out(O_0/R_0)$, so $A/HC$ is isomorphic to a subquotient of $\Out(O_0/R_0)$; thus by our hypothesis, $C \nleq H$.  Since $G(x)/G_1(x) \cong F_1$ is $2$-transitive, it follows that $C$ acts transitively on $S(x,1)$.  Now the $C$-orbit relation on $S(x,2)$ is $G(x)$-invariant and not contained in $\sim_{x,2}$, so it is universal, that is, $C$ is transitive on $S(x,2)$.  In particular, given $y \in S(x,1)$ then $C_y := C \cap G(x,y)/G_2(x)$ acts transitively on $S_x(y,1)$.  However, we see that $C_y$ acts as a normal subgroup of $F_0(\omega)$ that centralizes $O_0/R_0$.  Since $\kappa(F_0(\omega))=1$, we deduce that $C_y$ acts as a soluble normal subgroup of $F_0(\omega)$.  Thus $F_0(\omega)$ has a soluble transitive normal subgroup, contradicting our hypothesis.  This contradiction completes the proof of case (a).

Now suppose we are in case (b), so $k \ge 2$.  Suppose for a contradiction that $\sim$ is not the trivial equivalence relation.  There are then $y',y'' \in S(x,k)$ such that $y' \sim y''$ and $0 < d(y',y'') < 2k$, so we have $y',y'' \in S_x(z,k-1)$ for some neighbour $z$ of $x$.  We have an action of $G_{k-1}(z)$ on $S_x(z,k)$.  Let $y \in S_x(z,k-1)$.  Since $\Theta^{k-1,1}_0 \ge O_0$, we see that there is $N_y \unlhd G_{k-1}(z)$ acting as $O_0$ on $S_z(y,1) = S_x(y,1)$ and fixing pointwise the rest of $S(z,k)$.  Thus $G_{k-1}(z)$ acts on $S_x(z,k)$ as a group containing the soluble residual $O = \prod_{y \in S_x(z,k-1)}O_0$ of $\prod_{y \in S_x(z,k-1)}F_0(\omega)$.  Now $G_k(x)$ is normal in $G_{k-1}(z)$, so if we take the image $H'$ of $G_k(x)$ in $\prod_{y \in S_x(z,k-1)}F_0(\omega)$, then $H' \cap O$ is normal in $O$; we then see that in fact $H' \ge O$ by Lemma~\ref{lem:minimally_covered:simple}.  Recalling that $\kappa(M_y)=1$, it follows that for each $y \in S_x(z,k-1)$, the image of $M_y$ in $O$ is just the direct factor indexed by $y$.  But then $M_{y'} \neq M_{y''}$, so $y' \not\sim y''$, a contradiction.  This contradiction completes the proof of case (b).
\end{proof}

\begin{lem}\label{lem:type_swap}
Let $G \in \mc{H}_{F_0,F_1}$.  For each $t \in \{0,1\}$ write $O_t := O^\infty(F_t(\omega))$.  Let $t \in \{0,1\}$ and $k \ge 1$, and assume that $\Theta^{k,1}_{1-t} \ge O_{1-t}$ and, if $k \ge 2$, that $\Theta^{k-1,1}_t \ge O_t$.  Then $\Theta^{k,0}_t \ge O_t$.
\end{lem}

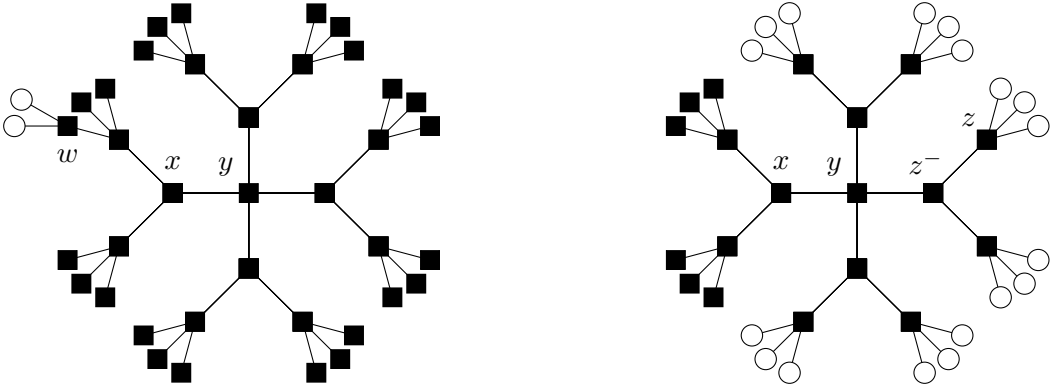
\begin{figure}
\caption{The action of $G_3(y)$ on $S_y(w,1)$ (left) and of $G_2(x,y)$ on $S_x(y,3)$ (right)}
\label{fig:typeswap}
\begin{center}
\begin{tikzpicture}[scale=1, every loop/.style={}, square/.style={regular polygon,regular polygon sides=4}]

\tikzstyle{every node}=[circle,
                        inner sep=0pt, minimum width=5pt]
                        
\foreach \x in {0,1,2,3} {
\foreach \y in {-1,1} {
\foreach \z in {-1,0,1} {
\draw (0,0) to ++ (90*\x:1) node[square, draw=black, fill=black, minimum width=10pt]{} to ++ (90*\x+45*\y:1) node[square, draw=black, fill=black, minimum width=10pt]{} to ++ (90*\x+45*\y+30*\z:0.7) node[square, draw=black, fill=black, minimum width=10pt]{};
}
}
};

\foreach \x in {2} {
\foreach \y in {-1} {
\foreach \z in {1} {
\foreach \w in {-1,1} {
\draw (0,0) ++ (90*\x:1) ++ (90*\x+45*\y:1) ++ (90*\x+45*\y+30*\z:0.7) node[fill=black]{} to ++ (90*\x+45*\y+30*\z+15*\w:0.7) node[draw=black, fill=white, minimum width=8pt]{};
}
}
}
};

\draw{
(0,0) node[square, draw=black, fill=black, minimum width=10pt]{}
(-0.3,0.35) node{$y$}
(-1,0.4) node{$x$} ++(135:1) ++ (165:0.7) ++ (0,-0.8) node{$w$}
};

\foreach \x in {0,1,3} {
\foreach \y in {-1,1} {
\foreach \z in {-1,0,1} {
\draw (8,0) to ++ (90*\x:1) node[square, draw=black, fill=black, minimum width=10pt]{} to ++ (90*\x+45*\y:1) node[square, draw=black, fill=black, minimum width=10pt]{} to ++ (90*\x+45*\y+30*\z:0.7) node[draw=black, fill=white, minimum width=8pt]{};
}
}
};

\foreach \x in {2} {
\foreach \y in {-1,1} {
\foreach \z in {-1,0,1} {
\draw (8,0) to ++ (90*\x:1) node[square, draw=black, fill=black, minimum width=10pt]{} to ++ (90*\x+45*\y:1) node[square, draw=black, fill=black, minimum width=10pt]{} to ++ (90*\x+45*\y+30*\z:0.7) node[square, draw=black, fill=black, minimum width=10pt]{};
}
}
};

\draw{
(8,0) node[square, draw=black, fill=black, minimum width=10pt]{}
(7.7,0.35) node{$y$}
(7,0.4) node{$x$}
(8.9,0.4) node{$z^-$} ++(45:0.8) node{$z$}
};

\end{tikzpicture}
\end{center}
\end{figure}

\begin{proof}
Let us take $t=1$; the argument for $t=0$ is similar.  Let $x$ and $y$ be adjacent vertices of $T$ of types $k$ and $k+1$ modulo $2$ respectively.

Suppose $k=1$; note that $\Theta^{1,0}_i = \Delta^1_i$ and $F_i(\omega) = \Lambda^1_i$.  We have $\Theta^{1,0}_0 \ge O_0$ by hypothesis, so $F_0(\omega)/\Theta^{1,0}_0$ is soluble.  By Lemma~\ref{lem:Goursat} it follows that $F_1(\omega)/\Theta^{1,0}_1$ is soluble, so $\Theta^{1,0}_1 \ge O_1$.

From now on we suppose $k \ge 2$.  We can interpret $\Theta^{k,0}_1$ as the action of $G_k(y)$ on $S_y(w,1)$ for any $w \in S(y,k)$; in fact we will take $w \in S_y(x,k-1)$, which ensures that $S_y(w,1) = S_x(w,1)$ (see Figure~\ref{fig:typeswap}).  For $z \in VT \setminus \{x\}$, let $z^-$ be the penultimate vertex on the path from $x$ to $z$ and write $G_{x,z} := G(z,z^-)/G_1(z)$.  Note that $G_{x,z}$ is a copy of $F_i(\omega)$, where $i$ is the type of $z$.

As can be seen in Figure~\ref{fig:typeswap} (right), the action of $G_{k-1}(x,y)$ on $S_x(y,k)$ is naturally described by a homomorphism 
\[
\varphi: G_{k-1}(x,y) \rightarrow \prod_{z \in S_x(y,k-1)}G_{x,z}.
\]

Note that $\varphi$ is equivariant with respect to the natural actions of $G(x,y)$ on the domain and codomain of $\varphi$, and the kernel of $\varphi$ is $G_k(y)$.  Since $\Theta^{k,1}_0 \ge O_0$, we see that $G_k(x)/G_{k+1}(x)$ contains $(O_0)^{|S(x,k)|}$; in particular, $\varphi(G_k(x))$ contains the soluble residual of $\prod_{z \in S_x(y,k-1)}G_{x,z}$.

Since $\Theta^{k-1,1}_1 \ge O_1$, we see that $G_{k-1}(x)/G_k(x)$ contains $(O_1)^{|S(x,k-1)|}$; in particular, we have a subgroup $A$ of $G_{k-1}(x)$ that acts as $O_1$ on $S_{x}(w,1)$ but fixes pointwise the set $S(x,k) \setminus S_{x}(w,1)$.  Let $N$ be the normal closure of $A$ in $G(x,y)$ and let $N^*$ be the soluble residual of $N$; then $N \le G_{k-1}(x,y)$ and the local action of $N$ at $w$ contains $O_1$, so the local action of $N^*$ at $w$ also contains $O_1$.  At the same time, $\varphi(N^*)$ is contained in the soluble residual of $\prod_{z \in S_x(y,k-1)}G_{x,z}$, so $\varphi(N^*) \le \varphi(G_k(x))$.

For each $a \in O_1$ we now take $g_a \in N^*$ acting as $a$ on $S_{x}(w,1)$, and $g'_a \in G_k(x)$ such that $\varphi(g'_a) = \varphi(g_a)$.  Then $h_a = (g'_a)\inv g_a$ is an element of $\ker\varphi = G_k(y)$; since $g'_a \in S_k(x)$ and $S_{y}(w,1) \subseteq S(x,k)$, the actions of $g_a$ and $h_a$ on $S_{y}(w,1)$ are the same.  Given the freedom of choice of $a$, we see that the action of $G_k(y)$ on $S_y(w,1)$ contains $O_1$, in other words, $\Theta^{k,0}_1 \ge O_1$.
\end{proof}

\begin{lem}\label{lem:soluble_bridge}
Let $G \in \mc{H}_T$ and let $\mc{C}$ be an extension-closed variety of finite groups.  Let $k \ge 1$ and suppose $x \in V_0T$ and $y \in V_1T$ are adjacent vertices such that $G_{k-1}(x,y)/G_k(x)$ and $G_{k-1}(x,y)/G_k(y)$ are both $\mc{C}$-groups.  Then the compact open subgroup $G_{k-1}(x,y)$ is pro-$\mc{C}$; in particular $G$ is locally pro-$\mc{C}$.
\end{lem}

\begin{proof}
The hypothesis ensures the following equalities of pro-$\mc{C}$ residuals:
\[
O^{\mc{C}}(G_k(x)) = O^{\mc{C}}(G_{k-1}(x,y)) = O^{\mc{C}}(G_k(y)).
\]
In particular, we see that $O^{\mc{C}}(G_{k-1}(x,y))$ is normalized by both $G(x)$ and $G(y)$, and hence is normal in $H = \grp{G(x),G(y)}$.  Since both local actions are $2$-transitive, we see that $H$ acts minimally on $T$; in particular, $H$ has no nontrivial compact normal subgroup.  It follows that $O^{\mc{C}}(G_{k-1}(x,y))$ is trivial, in other words, $G_{k-1}(x,y)$ is pro-$\mc{C}$.
\end{proof}

We now obtain the promised restrictions on $G \in \mc{H}_{F_0,F_1}$ in the case that $\Theta_t \ngeq O_t$ for some $t \in \{0,1\}$.

\begin{prop}\label{prop:soluble_dichotomy}
Let $G \in \mc{H}_{F_0,F_1}$, with $F_t$ acting on $\Omega_t$, and write $O_t := O^\infty(F_t(\omega))$.  Suppose that $\kappa(F_0),\kappa(F_1) \le 1$ and that $\Out(O_t/O_\infty(O_t))$ is soluble for $t \in \{0,1\}$.  Then exactly one of the following holds:
\begin{enumerate}[(i)]
\item $\Theta_t \ge O_t$ for all $t \in \{0,1\}$ and $O_t \neq \triv$ for some $t \in \{0,1\}$.
\item $G$ is locally pro-$\mc{C}$, where $\mc{C}$ is the smallest extension-closed variety of finite groups containing $O_\infty(F_0(\omega))$ and $O_\infty(F_1(\omega))$.  Moreover, for $t \in \{0,1\}$, we have $\Theta_t \le O_\infty(F_t(\omega))$, and the group $O_{\mc{C}}(\Lambda_t)$ acts transitively on $\Omega_t \setminus \{\omega\}$.
\end{enumerate}
\end{prop}

\begin{proof}
Note that all groups in $\mc{C}$ are soluble.  If (i) holds, it is clear that there is no $k \ge 0$ or $v \in VT$ for which $G_k(v)$ is prosoluble, and hence $G$ is not locally prosoluble.  Thus (i) and (ii) are mutually exclusive and we may assume (ii) is false.

Suppose for the moment that $G$ is not locally pro-$\mc{C}$.  Then Lemma~\ref{lem:soluble_bridge} implies that there is some $t \in \{0,1\}$ such that $\Theta^{\infty,0}_t$ is not a $\mc{C}$-group; say this happens for $t=1$.  By the definition of $\mc{C}$, we deduce that
\[
\Theta^{\infty,0}_1 \nleq O_{\infty}(F_1(\omega)) \neq F_1(\omega);
\]
by Proposition~\ref{prop:Hk_subnormal}, $\Theta^{\infty,0}_1$ is subnormal in $F_1(\omega)$, so we deduce via Lemma~\ref{lem:minimally_covered:simple} that
\[
\Theta^{\infty,0}_1 \ge O_1 \neq \triv,
\]
and hence $\Theta_1 \ge O_1$ by Corollary~\ref{cor:Hk_subnormal}.  If $F_0(\omega)$ is soluble, we automatically have $\Theta_0 \ge O_0$, so suppose $F_0(\omega)$ is insoluble.  Lemma~\ref{lem:type_swap} ensures that $\Theta^{1,0}_0 \ge O_0$.  Proceeding by induction on $k$, we find that
\[
\Theta^{k,0}_0 \ge O_0 \; \Rightarrow \; \Theta^{k,1}_0 \ge O_0 \; \Rightarrow \; \Theta^{k+1,0}_0 \ge O_0,
\]
for all $k \ge 1$, with the first implication by Lemma~\ref{lem:action_to_rigid_action} (in the case $k=1$ we use the fact that $\Out(O_0/O_\infty(O_0))$ is soluble), and the second by Lemma~\ref{lem:type_swap}.  Thus $\Theta^{\infty,0}_0 \ge O_0$, and we deduce $\Theta_0 \ge O_0$ via Corollary~\ref{cor:Hk_subnormal}.  Thus (i) holds.

From now on we may assume $G$ is locally pro-$\mc{C}$, in other words, there is some $x \in V_0T$ and $k \ge 1$ such that $G_k(x)$ is pro-$\mc{C}$.  We immediately deduce that $\Theta_t$ is soluble for $t \in \{0,1\}$; since $\Theta_t$ is clearly normal in $F_t(\omega)$, we deduce that
\[
\forall t \in \{0,1\}: \; \Theta_t \le O_\infty(F_t(\omega)).
\]
We can also choose $k$ large enough that $\Lambda^l_t = \Lambda_t$ for $t \in \{0,1\}$ and all $l \ge k$.

The remaining possibility to rule out is that we have some $t \in \{0,1\}$ such that $O_{\mc{C}}(\Lambda_t)$ does not act transitively.  Let $y \in S(x,2l+t)$ such that $2l \ge k$ and consider the action $\Delta$ of $G_k(x) \cap G(y)$ on $S_x(y,1)$.  Then, since $\Delta$ is a normal $\mc{C}$-subgroup of $\Lambda_t$, we see that $\Delta$ acts intransitively.  Repeating this argument for all $l \ge k/2$, we find that $G_k(x)$ has infinitely many orbits on the boundary; since $G_k(x)$ has finite index in $G(x)$, it follows that $G(x)$ has infinitely many orbits on the boundary, contradicting Corollary~\ref{cor:distance_transitivity}(iii).  This completes the proof that the negation of (ii) implies (i), and hence the proof of the dichotomy.
\end{proof}

The proof of Proposition~\ref{prop:soluble_dichotomy} did not require any detailed knowledge of finite $2$-transitive permutation groups.  However, if we now invoke Lemmas~\ref{lem:kappa_one} and~\ref{lem:lambda_restriction}, we see there is a strong restriction on Proposition~\ref{prop:soluble_dichotomy}(ii).

\begin{cor}\label{cor:soluble_dichotomy}
Let $G \in \mc{H}_{F_0,F_1}$, with $F_t$ acting on $\Omega_t$, and write $O_t := O^\infty(F_t(\omega))$.  Then exactly one of the following holds:
\begin{enumerate}[(i)]
\item $\Theta_t \ge O_t$ for all $t \in \{0,1\}$.
\item For all $t \in \{0,1\}$, we have $\Theta_t \le O_{\infty}(F_t(\omega))$ and $\Lambda_t$ is soluble; but at least one of $F_0(\omega)$ and $F_1(\omega)$ is insoluble.  Moreover, for all $t \in \{0,1\}$ such that $F_t(\omega)$ is insoluble, then $\PSL_3(q) \le F_t \le \PGaL_3(q)$ with $q > 3$, and $F_t(\omega)$ is a point stabilizer of the usual action of $F_t$ on $P_2(q)$.
\end{enumerate}
\end{cor}

\begin{proof}
Cases (i) and (ii) are clearly mutually exclusive.  If $F_0(\omega)$ and $F_1(\omega)$ are both soluble then clearly (i) holds, so we may assume $O_0 \neq \triv$ say.  In that case, Proposition~\ref{prop:soluble_dichotomy}(i) implies case (i) of the corollary, so we may assume Proposition~\ref{prop:soluble_dichotomy}(ii) holds, implying that $\Theta_t \le O_{\infty}(F_t(\omega))$ and that $O_{\infty}(\Lambda_t)$ acts transitively on $\Omega_t \setminus \{\omega\}$ for both $t \in \{0,1\}$.  By Lemma~\ref{lem:lambda_restriction}(b), $\Lambda_0$ and $\Lambda_1$ are both soluble, so $\Lambda_0$ does not contain $O_0$.  Taking $t$ such that $O_t \neq \triv$, then all cases except (ii) and (vi) of Lemma~\ref{lem:lambda_restriction} are ruled out.  In the remaining cases we have $\PSL_3(q) \le F_t \le \PGaL_3(q)$, and $F_t(\omega)$ is a point stabilizer of the usual action of $F_t$ on $P_2(q)$; since $O_t$ is nontrivial we have $q > 3$.
\end{proof}

Corollary~\ref{cor:soluble_dichotomy}(i) is the conclusion of Theorem~\ref{intro:locally_prosoluble}(i), so the theorem will be proved as soon as we reduce Corollary~\ref{cor:soluble_dichotomy}(ii) to the special case given in Theorem~\ref{intro:locally_prosoluble}(ii).

\subsection{The local actions of the pointwise stabilizer of a line segment}\label{sec:line}

Let $x$ and $y$ be vertices in the tree $T$ at distance $k$ apart, such that $x$ is of type $t$.  Recall (as illustrated in Figure~\ref{fig:lambda}) that given $G \in \mc{H}_T$, the group $\Lambda^k_t:=\Lambda^k_t(G)$ is the action of $G(x,y)$ on $S_y(x,1)$, while $\Delta^k_t:=\Delta^k_t(G)$ is the action of $G(x) \cap G_1(y)$ on $S_y(x,1)$.  Let $\delta^k_t$\index{d@$\delta^k_t$} be the number of orbits of $\Delta^k_t$ on $S_y(x,1)$.

We find a useful relationship between the groups $\Lambda^k_t$ and $\Delta^k_t$.

\begin{lem}\label{lem:line_index}
Let $G \in \mc{H}_T$ and let $k \ge 1$.  Then
\[
|\Lambda^k_t: \Lambda^{k+1}_t| = \delta^k_{t+k}.
\]
Moreover, $|\Lambda^k_t: \Lambda^{k+1}_t|$ divides $d_{t+k}-1$.
\end{lem}

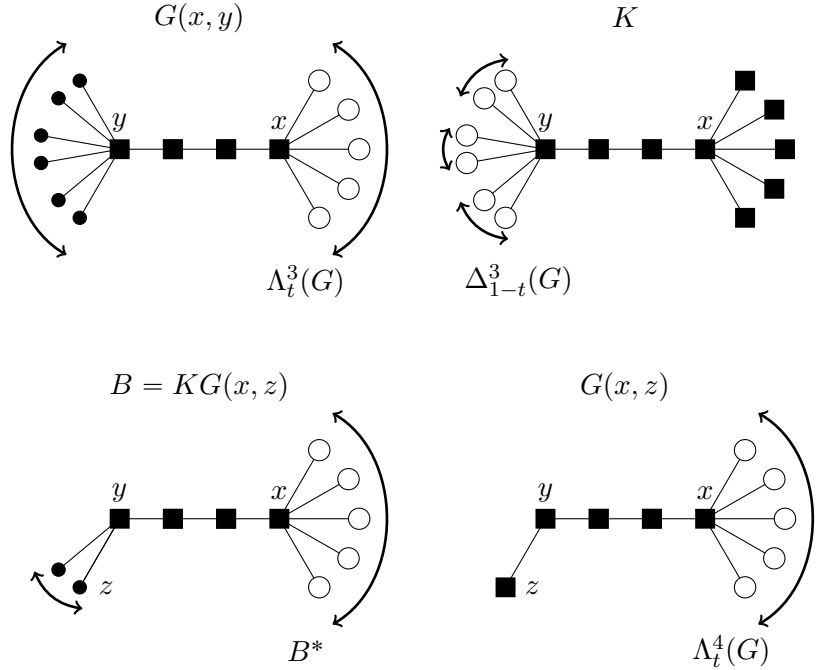
\begin{figure}[h]
\caption{From $\Lambda^k_t(G)$ to $\Lambda^k_{t+1}(G)$ via $\Delta^k_{t+k}(G)$}
\label{fig:delta}
\begin{center}
\begin{tikzpicture}[scale=0.7, every loop/.style={}, square/.style={regular polygon,regular polygon sides=4}]

\tikzstyle{every node}=[circle,
                        inner sep=0pt, minimum width=5pt]
          
\draw \foreach \x in {0,1,2}
{
(\x,0) node[square, draw=black, fill=black, minimum width=10pt]{} to (\x+1,0) node[square, draw=black, fill=black, minimum width=10pt]{}
};
\draw \foreach \x in {-2,-1,0,1,2}
{
(3,0) to ++(30*\x:1.5) node[draw=black, fill=white, minimum width=8pt]{}
};
\draw \foreach \x in {-3,-2,-0.5,0.5,2,3}
{
(0,0) to ++ (180+20*\x:1.5) node[draw=black, fill=black]{}
};
\draw{
(-1,2) edge[line width=1pt, <->, bend right=60] (-1,-2)
(4,2) edge[line width=1pt, <->, bend left=60] (4,-2)
(0,0.5) node{$y$}
(3,0.5) node{$x$}
(1.5,2.5) node{$G(x,y)$}
(3.5,-2.5) node{$\Lambda^3_t(G)$}
};

\draw \foreach \x in {0,1,2}
{
(8+\x,0) node[square, draw=black, fill=black, minimum width=10pt]{} to (8+\x+1,0) node[square, draw=black, fill=black, minimum width=10pt]{}
};
\draw \foreach \x in {-2,-1,0,1,2}
{
(11,0) to ++(30*\x:1.5) node[square, draw=black, fill=black, minimum width=10pt]{} 
};
\draw \foreach \x in {-3,-2,-0.5,0.5,2,3}
{
(8,0) to ++ (180+20*\x:1.5) node[draw=black, fill=white, minimum width=8pt]{}
};
\draw{
(7.3,1.7) edge[line width=1pt, <->, bend right=30] (6.4,1)
(7.3,-1.7) edge[line width=1pt, <->, bend left=30] (6.4,-1)
(6.2,-0.4) edge[line width=1pt, <->, bend left=30] (6.2,0.4)
(8,0.5) node{$y$}
(11,0.5) node{$x$}
(9.5,2.5) node{$K$}
(7.5,-2.5) node{$\Delta^3_{1-t}(G)$}
};

\draw (0,-7) to ++ (240:1.5) ++ (0.5,0) node{$z$};

\draw \foreach \x in {0,1,2}
{
(0+\x,-7) node[square, draw=black, fill=black, minimum width=10pt]{} to (0+\x+1,-7) node[square, draw=black, fill=black, minimum width=10pt]{}
};
\draw \foreach \x in {-2,-1,0,1,2}
{
(3,-7) to ++(30*\x:1.5) node[draw=black, fill=white, minimum width=8pt]{}
};
\draw \foreach \x in {2,3}
{
(0,-7) to ++ (180+20*\x:1.5) node[draw=black, fill=black]{}
};
\draw{
(-0.7,-8.7) edge[line width=1pt, <->, bend left=30] (-1.6,-8)
(4,-5) edge[line width=1pt, <->, bend left=60] (4,-9)
(0,-6.5) node{$y$}
(3,-6.5) node{$x$}
(1.5,-4.5) node{$B = KG(x,z)$}
(3.5,-9.5) node{$B^*$}
};

\draw (8,-7) to ++ (240:1.5) node[square, draw=black, fill=black, minimum width=10pt]{} ++ (0.5,0) node{$z$};

\draw \foreach \x in {0,1,2}
{
(8+\x,-7) node[square, draw=black, fill=black, minimum width=10pt]{} to (8+\x+1,-7) node[square, draw=black, fill=black, minimum width=10pt]{}
};
\draw \foreach \x in {-2,-1,0,1,2}
{
(11,-7) to ++(30*\x:1.5) node[draw=black, fill=white, minimum width=8pt]{}
};
\draw{
(12,-5) edge[line width=1pt, <->, bend left=60] (12,-9)
(8,-6.5) node{$y$}
(11,-6.5) node{$x$}
(9.5,-4.5) node{$G(x,z)$}
(11.5,-9.5) node{$\Lambda^4_t(G)$}
};

\end{tikzpicture}
\end{center}
\end{figure}

\begin{proof}
The key groups involved in the proof are illustrated in Figure~\ref{fig:delta}.  Take two vertices $x$ and $y$ with $d(x,y)=k$ and $x \in V_tT$, and take $z \in S_x(y,1)$; we compare the actions of $G(x,y)$ and $G(x,z)$ on $S_y(x,1)$.  Note that $G(x,y)$ acts transitively on $S_x(y,1)$.  The kernel $K = G_1(x) \cap G(y)$ of the action of $G(x,y)$ on $S_y(x,1)$ is a normal subgroup of $G(x,y)$, so its orbits on $S_x(y,1)$ are blocks of the action of $G(x,y)$; considering the type of $y$, there are $\delta = \delta^k_{t+k}$ orbits in this action, and also $\delta$ divides $d_{t+k}-1$.  Let $B$ be the setwise stabilizer in $G(x,y)$ of the $K$-orbit of $z$ and let $B^*$ be the local action of $B$ at $x$.  Then $G(x,z),K \le B \le G(x,y)$, and moreover the orbits $Kz$ and $Bz$ coincide, so $B = KG(x,z)$.  In particular, since the local action of $K$ at $x$ is trivial, we see that $B^* = \Lambda^{k+1}_t$.  Since $G(x,y)$ permutes the blocks of $S_x(y,1)$ transitively, we have $|G(x,y):B|=\delta$, and then since $B \ge K$, we have $|\Lambda^k_t:B^*| = \delta$.  We conclude that $\delta = |\Lambda^k_t: \Lambda^{k+1}_t|$.
\end{proof}

When $\triv \neq O^\infty(F_0(\omega)) \le \Lambda^k_0$, we can narrow down the possible values of $\delta^k_0$ further, finding that except in some special situations, we have $\delta^k_0 =1$, so that $\Lambda^{k+1}_k = \Lambda^k_k$.  The main work in the proof will be to deal with the situation of Corollary~\ref{cor:soluble_dichotomy}(ii), which will isolate the case of exceptional locally prosoluble type.

\begin{lem}\label{lem:insoluble_line_index}
Let $G \in \mc{H}_{F_0,F_1}$; for $t \in \{0,1\}$, write $O_t = O^\infty(F_t(\omega))$.  Suppose that $\tau \in \{0,1\}$ and $k \ge 1$ are such that $\Lambda^k_\tau \ge O_\tau$.  Then exactly one of the following is satisfied.
\begin{enumerate}[(a)]
\item We have $O_\tau = \triv$.
\item We have $\Delta^k_\tau \ge O_\tau \neq \triv$ and $\delta^k_\tau= 1$, so $\Lambda^{k+1}_{\tau+k} = \Lambda^k_{\tau+k}$.
\item We have $\Delta^k_\tau \ge O_\tau \neq \triv$ but $\delta^k_\tau > 1$.  In this case, $F_\tau$ is of exceptional $\SL_2(5)$ affine type, whereas $F_{\tau+k}$ is almost simple.  In particular, $d_\tau = q^2$ for $q \in \{9,19,29,59\}$; we have $O^\infty(F_\tau(\omega)) = \SL_2(5)$; $\delta^k_\tau = \delta^*(F_\tau) \in \{2,3,7,29\}$; and $\Lambda^k_\tau = \Lambda_\tau = F_\tau(\omega)$.
\item We have $\Delta^k_\tau$ soluble but $O_\tau \neq \triv$, so $\Delta^k_\tau$ does not contain $O_\tau$.  In this case, $G$ is of exceptional locally prosoluble type, with $F_t = \PGaL_3(q_t)$ and $\{q_0,q_1\} = \{4,5\}$.  For both $t \in \{0,1\}$, it is also the case that $\Lambda^k_t = F_t(\omega)$ whereas $\Delta^k_t$ is the soluble radical of $F_t(\omega)$, with the result that $\delta^k_t = q_t+1$, and $\Lambda^{k+1}_t$ is soluble.
\end{enumerate}
Moreover, if $F_\tau \cong F_{\tau+k}$ (for example, if $k$ is even), then cases (c) and (d) cannot occur.
\end{lem}

\begin{proof}
Note that the four cases are mutually exclusive: case (a) is the only case in which $O_\tau = \triv$; cases (b) and (c) are distinguished from each other by the value of $\delta^k_\tau$; case (d) is the only case in which $\Delta^k_\tau$ does not contain $O_\tau$.  If $O_\tau = \triv$ there is nothing more to prove, so assume $O_\tau \neq \triv$.  Given Lemma~\ref{lem:line_index}, if $O_\tau \neq \triv$ and $\delta^k_\tau = 1$ we are in case (b), so we assume from now on that $\delta^k_\tau > 1$.

Without loss of generality let us take $\tau=0$.  For $t \in \{0,1\}$, let $R^k_t$ be the soluble radical of $\Lambda^k_t$.  We have already observed that $\Delta^k_0$ is normal in $\Lambda^k_0$.  Since $\Lambda^k_0$ acts transitively on $\Omega \setminus \{\omega\}$, all orbits of $\Delta^k_0$ are of equal size.  Note also that $\Delta^k_t$ contains $\Theta_t$.

Suppose $\Delta^k_0 \ge O_0$.  Then $O_0$ is nontrivial and intransitive, so we fall into case (ii) of Lemma~\ref{lem:soluble_residual_trans}; thus $\delta^k_0$ divides $\delta^*(F_0)$, with $\delta^*(F_0) \in \{2,3,7,29\}$.  Since $\delta^*(F_0)$ is prime, in fact $\delta^k_0=\delta^*(F_0)$.  Since $F_0$ is of affine type, we see by Lemma~\ref{lem:lambda_restriction} that $\Lambda^k_0 = \Lambda_0 = F_0(\omega)$.  By Lemma~\ref{lem:line_index} we deduce that $\Lambda^{k+1}_k$ is properly contained in $F_k(\omega)$; by Lemma~\ref{lem:lambda_restriction} it follows that $F_k$ is almost simple, so $F_0 \not\cong F_k$ and $k$ is odd. Thus we are in case (c) of the present lemma.

We have now dealt with the situation in which $\Delta^k_0 \ge O_0$, so from now on we suppose that $\Delta^k_0$ does not contain $O_0$.  By hypothesis $\Lambda^k_0 \ge O_0$, from which it follows that $O_0$ is the soluble residual of $\Lambda^k_0$; we then see that $\kappa(\Lambda^k_0) = \kappa(F_k(\omega)) = 1$.  Since $\Delta^k_0$ is normal in $\Lambda^k_0$, it follows via Lemma~\ref{lem:minimally_covered:simple} that $\Delta^k_0 \le R^k_t$, and hence $\Theta_0 \le R^k_t$.

By Corollary~\ref{cor:soluble_dichotomy} we have $\PSL_3(q_0) \unlhd F_0 \le \PGaL_3(q_0)$ with $F_0$ acting on $\Omega_0 = P_2(q_0)$, for some prime power $q_0 > 3$; moreover, $F_1$ is of the same form (using a possibly different prime power $q_1>3$) or $F_1(\omega)$ is soluble.

The projective line $P_1(q_0)$ is an $F_0(\omega)$-equivariant quotient of $\Omega_0 \setminus \{\omega\}$ in the manner described in Remark~\ref{rem:linear}.  Since $O_0 \le \Lambda^k_0$, the action of $\Lambda^k_0$ on $P_1(q_0)$ yields a homomorphism $\pi_0: \Lambda^k_0 \rightarrow \PGaL_2(q_0)$ such that the image contains the simple normal subgroup $\PSL_2(q_0)$.  The kernel of $\pi_0$ is the soluble radical $R^k_0$ of $\Lambda^k_0$, since the centralizer of $\PSL_2(q_0)$ in $\PGaL_2(q_0)$ is trivial.  In particular, $\pi_0(\Delta^k_0) = \triv$.

Let $A_{k,t}$ and $\psi_{k,t}$ be as in Lemma~\ref{lem:Goursat}; note that $A_{k,0} = A_{k,k}$.  Then we have surjective homomorphisms $\psi_{k,0}:\Lambda^k_0 \rightarrow A_{k,0}$ and $\psi_{k,k}:\Lambda^k_k \rightarrow A_{k,0}$ with kernel $\Delta^k_0$ and $\Delta^k_k$ respectively.  Let $B_{k,0} = A_{k,0}/O_{\infty}(A_{k,0})$ and let $\alpha$ be the quotient map from $A_{k,0}$ to $B_{k,0}$; then $B_{k,0}$ is an almost simple group with socle $\PSL_2(q_0)$.  Since $B_{k,0}$ occurs as a quotient of $\Lambda^k_k$, it follows that $\Lambda^k_k$ is insoluble, and hence $F_k(\omega)$ is insoluble; thus $\PSL_3(q_k) \unlhd F_k \le \PGaL_3(q_k)$ for some $q_k > 3$.  We can then define $\pi_k: \Lambda^k_k \rightarrow \PGaL_2(q_k)$ similarly to $\pi_0$.

To finish the proof, we must show that we are in case (d), with $F_0 \not\cong F_k$; we accomplish this with a series of claims.

\emph{Claim 1: The parameter $k$ is odd.  If $\Lambda^k_k \ge O_k$, then $\{q_0,q_1\} = \{4,5\}$ and $\PSL_2(q_0) \cong \PSL_2(q_1) \cong \Alt(5)$.}

Note that if $k$ is even, then $\Lambda^k_0 = \Lambda^k_k$ and $F_0 = F_k$; thus in this case $\Lambda^k_k \ge O_k$ and $q_0=q_k$.

Suppose that $\Lambda^k_k \ge O_k$ and $q_0=q_k$, with the aim of deriving a contradiction.  We see that $\alpha\psi_{k,0}$ and $\alpha\psi_{k,k}$ have kernel $R^k_0$ and $R^k_k$ respectively; there are thus injective homomorphisms $\beta_0,\beta_k:B_{k,0} \rightarrow \PGaL_2(q_0)$ such that
\[
\beta_0\alpha\psi_{k,0} = \pi_0 \text{ and } \beta_k\alpha\psi_{k,k} = \pi_k.
\]
The maps $\beta_0$ and $\beta_k$ differ only by an automorphism of $\PGaL_2(q_0)$; indeed, since $\PGaL_2(q_0)$ has trivial outer automorphism group (see for instance \cite[Theorem~3.2]{Wilson}), there is $\lambda \in \PGaL_2(q_0)$ such that $\beta_k(b) = \lambda \beta_0(b) \lambda\inv$ for all $b \in B_{k,0}$.
Consider the action of $G(x,y)$ on $S_y(x,1) \times S_x(y,1)$.  Via the map described in Remark~\ref{rem:linear}, we obtain a factor $G(x,y)$-space $P_1(q_0) \times P_1(q_k)$ of $S_y(x,1) \times S_x(y,1)$.  The group of permutations of $P_1(q_0) \times P_1(q_k)$ induced by $G(x,y)$ is then the following subgroup of $\PGaL_2(q_0)^2 = \PGaL_2(q_0) \times \PGaL_2(q_k)$:
\[
M = \{(h,\lambda h \lambda\inv) \mid h \in H\}, \text{ where } H = \pi_0(\Lambda^k_0).
\]
In particular, $M$ is contained in a $\PGaL_2(q_0)^2$-conjugate of $D = \{(h,h) \mid h \in \PGaL_2(q_0)\}$.  However, we see that the action of any conjugate of $D$ is such that a point stabilizer in one factor fixes a point in the other factor, so $D$ acts intransitively on $P_1(q_0) \times P_1(q_k)$, and hence $G(x,y)$ is intransitive on $S_y(x,1) \times S_x(y,1)$, which implies that $G$ is not type-distance-transitive.  This contradicts Corollary~\ref{cor:distance_transitivity}.

From the contradiction we deduce that $k$ is odd; suppose that $\Lambda^k_k \ge O_k$.  Then $q_0 \neq q_k$ and $\kappa(\Lambda^k_0) = \kappa(\Lambda^k_k)=1$; moreover, the unique nonabelian composition factor of $\Lambda^k_t$ for $t \in \{0,k\}$ is $\PSL_2(q_k)$.  Given the surjective maps $\psi_{k,0}$ and $\psi_{k,k}$, we then have $\PSL_2(q_0) \cong \PSL_2(q_k)$ as abstract groups; thus $\{q_0,q_k\} = \{4,5\}$ and $\PSL_2(q_0) \cong \PSL_2(q_1) \cong \Alt(5)$, see \cite{Artin}.  This completes the proof of Claim 1.

We now know that $\PSL_3(q_t) \unlhd F_t \le \PGaL_3(q_t)$ and $q_t > 3$ for all $t \in \{0,1\}$, and we can write $q_1 = q_k$.

\emph{Claim 2: Given $k' \ge 1$ and $t \in \{0,1\}$ such that $\Delta^{k'}_t \ge O_t$, then $\delta^{k'}_t= 1$.}

As mentioned before, we have already reduced the situation in which $\Delta^{k'}_t \ge O_t$ to that described in the three cases (a), (b), (c).  On the other hand, since $\PSL_3(q_t) \unlhd F_t \le \PGaL_3(q_t)$ and $q_t > 3$ for all $t \in \{0,1\}$, we see that neither local action is soluble and neither local action is of affine type.  Thus if $\Delta^{k'}_t \ge O_t$ then we are in the situation of case (b), meaning that $\delta^{k'}_t= 1$ as claimed.

\emph{Claim 3: Suppose $k$ is minimal such that there exists $t \in \{0,1\}$ such that $\Delta^k_t$ does not contain $O_t$.  Then case (d) holds.}

By Claim 2 and Lemma~\ref{lem:line_index}, we see that
\[
\forall t \in \{0,1\}: \; F_t(\omega) = \Lambda^1_t = \Lambda^2_t = \dots = \Lambda^k_t.
\]
In particular, $\Lambda^k_k \ge O_k$.  By Claim 1 we have $\{q_0,q_1\} = \{4,5\}$; without loss of generality, $q_0=5$ and $q_1=4$.  We then have $F_0 = \PSL_3(5) = \PGaL_3(5)$, and hence $B_{k,0} \cong \Sym(5)$, acting as $\PGaL_2(q_t)$ on $P_1(q_t)$.  It follows that $F_1(\omega)/R^k_k$ is isomorphic to $\Sym(5)$ (as opposed to $\Alt(5)$), which in turn implies that $F_1 \nleq \PGL_3(4)$.

We see that $\Lambda^k_t = F_t(\omega)$ has a unique minimal normal subgroup $W_t$, which is elementary abelian of order $q^2_t$. Considering the structure of $\Lambda^k_t$, one sees that $B_{k,0}$ is maximal among common quotients of $\Lambda^k_0$ and $\Lambda^k_1$, so to satisfy Lemma~\ref{lem:Goursat} we must have $\Delta^k_t = R^k_t$.  We then see that orbits of $\Delta^k_t$ correspond to points in $P_1(q_t)$, so $\delta^k_t = q_t+1$.  Then by Lemma~\ref{lem:line_index}, we have $|\Lambda^k_t : \Lambda^{k+1}_t| = q_{1-t}+1$; more precisely, $\Lambda^{k+1}_t$ is the preimage of a point stabilizer of the action of $B_{k,0}$ on $P_1(q_{1-t})$.

By Corollary~\ref{cor:tree_2bbtrans}, the action of $F_t$ on the cosets of $\Lambda^{k+1}_t$ is $2$-by-block-transitive, with block size $q_{1-t}+1$.  The possibilities are controlled by Table~\ref{table:small_q}; in particular, given that $F_1 \nleq \PGL_3(4)$, we must have $F_1 = \PGaL_3(4)$.  Obtaining $\Lambda^{k+1}_t$ as in the previous paragraph, one sees that $\Lambda^{k+1}_0$ is of the form $W_0 \rtimes (\SL_2(3) \rtimes C_4)$ and $\Lambda^{k+1}_1$ is of the form $W_1 \rtimes \GaL_1(16)$.  Moreover, these subgroups are indeed point stabilizers of $2$-by-block-transitive actions as indicated in Table~\ref{table:small_q}. In particular, note that the groups $\Lambda^{k'}_0$ and $\Lambda^{k'}_1$ are soluble for all $k' > k$ including $k'=\infty$.  From Proposition~\ref{prop:soluble_dichotomy}, it follows that $G$ is locally prosoluble, so we deduce that $G$ is of exceptional locally prosoluble type.  This completes the proof of Claim 3.

\emph{Claim 4: The parameter $k$ is minimal such that there exists $t \in \{0,1\}$ such that $\Delta^k_t$ does not contain $O_t$.}

For this claim we consider the structure of $\Lambda^{k'}_t$ for all possible choices of $k'$ and $t$.  As in the proof of Claim 3, we have
\[
\forall t \in \{0,1\}: \; F_t(\omega) = \Lambda^1_t = \Lambda^2_t = \dots = \Lambda^{k_0}_t,
\]
where $k_0$ is minimal such that there exists $t \in \{0,1\}$ such that $\Delta^k_t$ does not contain $O_t$.  Then by Claim 3, we see that $\Lambda^{k'}_t$ is soluble for all $k' > k_0$ and $t \in \{0,1\}$, whereas $O_0$ and $O_1$ are not soluble.  Since in the hypothesis of the lemma we assumed that $\Lambda^k_\tau \ge O_\tau$, it follows that $k = k_0$ as claimed.

Combining Claims 3 and 4 completes the proof of the lemma.
\end{proof}

The critical case of Lemma~\ref{lem:insoluble_line_index} for the purposes of Theorem~\ref{intro:locally_prosoluble}, which ultimately leads to case (ii) of Theorem~\ref{intro:locally_prosoluble}, is case (d): this is the only situation in which we can have $\Lambda^k_t \ge O_t$ but $\Lambda^{k+1}_t \ngeq O_t$ for some $t \in \{0,1\}$.  The proof that this situation can actually occur will be given in Example~\ref{ex:exceptional} at the end of the article.

Note the following consequence, which completes a large part of the proof of Theorem~\ref{intro:locally_prosoluble} and takes care of a special case of Theorem~\ref{intro:end_stabilizer}.

\begin{cor}\label{cor:insoluble_line_index}
Let $G \in \mc{H}_{F_0,F_1}$ and write $O_t = O^\infty(F_t(\omega))$.  Suppose that Lemma~\ref{lem:insoluble_line_index}(d) does not occur for any $k$.  Then $\Delta^k_t \ge O_t$ for both $t \in \{0,1\}$ and all $k \ge 0$; consequently, $\Lambda_t \ge O_t$.  If in addition, we have $O_0 \neq \triv$ and $O_1 \neq \triv$, and neither $F_0$ nor $F_1$ is of exceptional $\SL_2(5)$ affine type, then $\Lambda_t = F_t(\omega)$ for both $t \in \{0,1\}$.
\end{cor}

\begin{proof}
Given $t \in \{0,1\}$ we can repeatedly apply Lemma~\ref{lem:insoluble_line_index} to restrict $\Delta^k_t$ and $\Lambda^{k+1}_t$ in terms of $\Lambda^k_t$, starting with $k=1$, where we have $\Lambda^1_t = F_t(\omega) \ge O_t$.  Because case (d) of Lemma~\ref{lem:insoluble_line_index} never occurs, we see by induction on $k$ that $\Delta^k_t \ge O_t$, and hence also $\Lambda^{k+1}_t \ge O_t$, for all $k \ge 1$ and both $t \in \{0,1\}$.  In particular, $\Lambda_t \ge O_t$ for both $t \in \{0,1\}$.

Now suppose that $O_0 \neq \triv$ and $O_1 \neq \triv$, and neither $F_0$ nor $F_1$ is of exceptional $\SL_2(5)$ affine type.  Then only case (b) of Lemma~\ref{lem:insoluble_line_index} can occur, so we see by Lemma~\ref{lem:line_index} and induction that $\Lambda^k_t = F_t(\omega)$ and $\delta^k_t = 1$ for all $k \ge 0$ and $t \in \{0,1\}$.  Thus $\Lambda_t = F_t(\omega)$ for both $t \in \{0,1\}$.
\end{proof}

We can also give the possible groups $\Lambda_t$ when Lemma~\ref{lem:insoluble_line_index}(d) does occur.

\begin{lem}\label{lem:end_stabilizer:excep}
Let $F_t = \PGaL_3(q_t)$ for $q_0=5$ and $q_1=4$, let $G \in \mc{H}_{F_0,F_1}$ be such that $\Lambda_t$ is properly contained in a point stabilizer of $F_t$ for some $t \in \{0,1\}$, and let $W_t$ be the socle of $F_t(\omega)$.  Then $\Lambda_t = W_t \rtimes A_t$, where
\[
A_0 \in \{\SL_2(3) \rtimes C_4, \SL_2(3) \rtimes C_2,  \SL_2(3), C_3 \rtimes C_8  \}  \; \text{ and } \; A_1 = \GaL_1(16).
\]
\end{lem}

\begin{proof}
Since $O_0 \neq \triv$ and $O_1 \neq \triv$, but we have $\Lambda_t \lneq F_t(\omega)$ for some $t \in \{0,1\}$, we see from Corollary~\ref{cor:insoluble_line_index} that Lemma~\ref{lem:insoluble_line_index}(d) must occur for some $k$, with the result that 
\[
\Lambda_0 \le W_0 \rtimes  (\SL_2(3) \rtimes C_4) \; \text{ and } \; \Lambda_1 \le W_1 \rtimes \GaL_1(16).
\]
Since $F_t$ has $2$-by-block-transitive action on the cosets of $\Lambda_t$, the possible groups for $\Lambda_0$ and $\Lambda_1$ can be deduced from Table~\ref{table:small_q}.  In particular, we see immediately that $\Lambda_1 = W_1 \rtimes \GaL_1(16)$.  One can check that among the groups occurring in Table~\ref{table:small_q} as point stabilizers of $2$-by-block-transitive actions of $F_0 = \PGaL_3(5)$, those that are subgroups of $W_0 \rtimes (\SL_2(3) \rtimes C_4)$ are as given in the statement.
\end{proof}

We have a few other special cases to deal with before the proofs of the main theorems.  First let us note some calculations that are easily verified using the list of possible actions in Table~\ref{table:small_q}.

\begin{lem}\label{lem:small_PSL3}
Let $G = \PGaL_3(p)$ for $p \in \{2,3\}$, acting $2$-by-block-transitively on the set $\Omega$ with point stabilizer $G(\omega)$, and with the set $\Omega/\sim$ of blocks corresponding to the projective plane $P_2(q)$ equipped with its usual $G$-action.  Let $N \unlhd G(\omega)$ and let $\delta$ be the number of orbits of $N$ on the set of blocks other than $[\omega]$.  If $p=2$ then $\delta \in \{1,3,6\}$.  If $p=3$ then $\delta \in \{1,2,4,12\}$; moreover, when the blocks are singletons, the case $\delta=2$ does not occur.
\end{lem}

We now obtain some restrictions on $\Lambda^k_t(G)$ when both local actions have socle $\PSL_3(q)$, or when one local action has such a socle and the other has soluble point stabilizers.

\begin{lem}\label{lem:leftover_psl}
Let $G \in \mc{H}_{F_0,F_1}$.  Suppose that $\PSL_3(q_0) \le F_0 \le \PGaL_3(q_0)$ with $F_0$ acting on $P_2(q_0)$, for some prime power $q_0 \ge 2$.  Suppose also that $F_1$ is of the same form for some prime power $q_1 \le q_0$, or that $F_1$ is some other $2$-transitive group such that $F_1(\omega)$ is soluble; in the latter case, write $q_1=1$.  Suppose also that $G$ does not satisfy the hypotheses of Lemma~\ref{lem:end_stabilizer:excep}.  Write $m_t=q_t$ if $q_t \le 3$; $m_t=1/3$ if $q_t > 3$ and $F_t \ngeq \PGL_3(q_t)$; and $m_t=1$ otherwise.  Then for all $k \ge 1$ and $t \in \{0,1\}$, we have $\Lambda^k_t \ge O^\infty(F_t(\omega))$.  Moreover, one of the following holds.
\begin{enumerate}[(i)]
\item We have $q_1=1$ and $|F_0(\omega):\Lambda^k_0|$ divides $m_0(q_0-1)$ for all $k \ge 1$.
\item We have $(q_0,q_1)=(3,2)$, $|F_0(\omega):\Lambda^k_0| \in \{1,3,6\}$ and $|F_1(\omega):\Lambda^k_1| \le 2$ for all $k \ge 1$.
\item We have $q_0 \equiv 1 \mod 9$, $q_1 = 2$ and $F_0 \ngeq \PGL_3(q_0)$.  In this case, for all $k \ge 1$ we have $\Lambda^k_1 = F_1(\omega)$ and $|\Lambda^k_0:\Lambda^{k+1}_0| \in \{1,3\}$.
\item We have $\Lambda^k_t = F_t(\omega)$ for all $k \ge 1$ and $t \in \{0,1\}$.
\end{enumerate}
\end{lem}

\begin{proof}
Write $O_t = O^\infty(F_t(\omega))$.  Note that $m_t(q_t-1)$ is always an integer; more specifically, if $q_t > 3$ then $m_t(q_t-1)$ is the index
\[
|F_0(\omega) \cap \PGL_3(q_0): O^\infty(\PGL_3(q_0)(\omega))|.
\]
By Corollary~\ref{cor:insoluble_line_index}, we have $\Lambda_t \ge O_t$ for both $t \in \{0,1\}$; moreover, if $q_1 > 3$ then we have $\Lambda_t = F_t(\omega)$ for both $t \in \{0,1\}$.  Thus we may assume $q_1 \le 3$, that is, $F_1(\omega)$ is soluble.

Let $Y_t$ be the set of divisors of $m_t(q_t-1)$ in $\Nb$.  We next claim that $|F_t(\omega):\Lambda^k_t| \in Y_t$ for all $k \ge 1$ and both $t \in \{0,1\}$.  If $q_t \le 3$ this follows from Table~\ref{table:small_q}.  If $q_t \ge 4$, then $t=0$ and since $\Lambda^k_0 \ge O_0$, either $\Lambda^k_0 = F_0(\omega)$ or we fall under case (a) of Theorem~\ref{thm:2bbtrans}.  In particular, we have the following restrictions:
\begin{enumerate}[(a)]
\item We have 
\[
|F_0(\omega):\Lambda^k_0| = |F_0(\omega) \cap \PGL_3(q_0):\Lambda^k_0\cap \PGL_3(q_0)|;
\]
since $\Lambda^k_0 \ge O_0$, this ensures $|F_0(\omega):\Lambda^k_0| \in Y_0$.
\item If $q_0 \neq 3$ and $F_0$ contains $\PGL_3(q_0)$, then $|F_0(\omega):\Lambda^k_0|$ is coprime to $3$.
\end{enumerate}
Notice that $|F_0(\omega):\Lambda^k_0|$ is coprime to $3$ unless $q_0 = 3$ or $q_0 \equiv 1 \mod 9$.  There is nothing more to prove if $q_1=1$, so from now on we assume $q_1 \in \{2,3\}$.

Consider next the case that $(q_0,q_1)=(3,2)$.  All that remains to show in this case is that $|F_0(\omega):\Lambda^k_0| \neq 2$.  We deduce this by observing that if $k$ is minimal such that $|F_0(\omega):\Lambda^k_0| = 2$, then $\Lambda^{k-1}_0 = F_0(\omega)$ and $|\Lambda^{k-1}_0:\Lambda^k_0| = 2$, so $\delta^{k-1}_{k-1} = 2$ by Lemma~\ref{lem:line_index}.  However, by Lemma~\ref{lem:small_PSL3}, this is not a valid value for either $\delta^{k-1}_0$ (since $\Lambda^{k-1}_0 = F_0(\omega)$) or $\delta^{k-1}_1$.

From now on we are assuming that $q_1 \in \{2,3\}$ and $(q_0,q_1) \neq (3,2)$.  We prove the remaining cases by induction on $k$.  For $k=1$ there is nothing to prove, so suppose $k \ge 2$ and that we have proved the lemma for all $1 \le k' < k$ and $t \in \{0,1\}$.  Note that by Lemmas~\ref{lem:small_PSL3} and~\ref{lem:insoluble_line_index}, together with the inductive hypothesis, for each $t \in \{0,1\}$ we have $\delta^{k-1}_t \in X_{q_t}$, where $X_2 = \{1,3,6\}$, $X_3 = \{1,4,12\}$ and $X_q = \{1\}$ for $q \ge 4$.  (For $q_t=3$, we are using the fact that our inductive hypothesis ensures $\Lambda^{k-1}_t = F_t(\omega)$.)  Given $t \in \{0,1\}$, we claim that $\Lambda^k_t$ is also compatible with the statement of the lemma.

Suppose that $q_t = q_{t+k-1}$ (in other words, $k$ is odd or $q_0=q_1$).  Then by Lemma~\ref{lem:line_index} and the previous arguments, we have 
\[
\delta^{k-1}_{t+k-1} = |\Lambda^{k-1}_t:\Lambda^k_t| \in X_{q_{t+k-1}} \cap Y_t = \{1\},
\]
so $\Lambda^k_t = \Lambda^{k-1}_t$.  Similarly if $t=1$, $q_0 > 3$ and $k$ is even, then $\delta^{k-1}_0=1$, so $\Lambda^k_t = \Lambda^{k-1}_t$.

From now on we may assume that $t=0$, $q_0 > 3$ and $k$ is even; thus $|\Lambda^{k-1}_0 :\Lambda^k_0| = \delta^{k-1}_1 \in X_{q_1}$, and by the inductive hypothesis, $\Lambda^{k-1}_1 = F_1(\omega)$.

By Lemma~\ref{lem:Goursat}, we have an isomorphism
\[
\theta: \frac{\Lambda^{k-1}_0}{\Delta^{k-1}_0} \rightarrow \frac{F_1(\omega)}{\Delta^{k-1}_1},
\]
so $F_1(\omega)/\Delta^{k-1}_1$ is isomorphic to a soluble quotient of $\Lambda^{k-1}_0$.  The largest soluble quotient of $\Lambda^{k-1}_0$ has derived length at most $2$, so $\Delta^{k-1}_1$ contains the second derived group of $F_1(\omega)$.  If $q_1=3$ we see that the second derived group of $F_1(\omega)$, which takes the form $C^2_3 \rtimes Q_8$, acts transitively on $\Omega_1 \setminus \{\omega\}$.  We then deduce that $\delta^{k-1}_1 = 1$ and hence $\Lambda^k_0 = \Lambda^{k-1}_0$.  This completes the proof of the lemma in the case $q_1=3$.

We may assume from now on that $q_1=2$.  Then the second derived group $N$ of $F_1(\omega)$ is nontrivial, which rules out the case $\delta^{k-1}_1=6$; hence $\delta^{k-1}_1 \in \{1,3\}$.  Moreover, we know that $|\Lambda^{k-1}_0 :\Lambda^k_0|$ is coprime to $3$ except possibly if $q_0 \equiv 1 \mod 9$ and $F_0 \ngeq \PGL_3(q_0)$.  This completes the inductive step in all cases and hence the proof of the lemma.
\end{proof}

\subsection{The main theorems}

With the results of the previous two subsections, we can immediately prove Theorem~\ref{intro:locally_prosoluble} and its corollaries.

\begin{proof}[Proof of Theorem~\ref{intro:locally_prosoluble}]
Note that the two cases are mutually exclusive: if (ii) holds, then $\Theta_0$ is soluble but $O^\infty(F_0(\omega))$ is nontrivial and perfect, so $O^\infty(F_0(\omega))$ cannot be contained in $\Theta_0$.

If Corollary~\ref{cor:soluble_dichotomy}(i) applies, then case (i) of the present theorem holds, so we may assume Corollary~\ref{cor:soluble_dichotomy}(ii) applies.  In particular, we deduce from Corollary~\ref{cor:insoluble_line_index} that Lemma~\ref{lem:insoluble_line_index}(d) must occur for some $k$.  Then $F_t = \PGaL_3(q_t)$ for $\{q_0,q_1\} = \{4,5\}$ and the possibilities for $\Lambda_t$ are given by Lemma~\ref{lem:end_stabilizer:excep}; thus (ii) is satisfied.
\end{proof}

\begin{proof}[Proof of Corollary~\ref{intro:soluble}]
Suppose that $G$ has a residually soluble compact open subgroup $U$; to prove (ii) we may assume that one of $F_0$ and $F_1$ has an insoluble point stabilizer.  Then $U$ contains rigid stabilizers of half-trees of both types, so Theorem~\ref{intro:locally_prosoluble}(i) is false; thus (ii) holds.

Conversely, suppose (ii) holds.  In case (ii)(a), clearly every arc stabilizer of $G$ is residually soluble.  In case (ii)(b), we fall into case (ii) of Proposition~\ref{prop:soluble_dichotomy}, so $G$ is locally prosoluble.  Thus (ii) implies (i), completing the proof that (i) and (ii) are equivalent.
\end{proof}

\begin{proof}[Proof of Corollary~\ref{intro:TMS}]
Suppose $G$ has local actions $F_0$ and $F_1$.  Since $G$ does not have a residually soluble compact open subgroup, the negation of Corollary~\ref{intro:soluble}(ii) holds, so we are in case (i) of Theorem~\ref{intro:locally_prosoluble} with $\Theta_t(G) \ge O^\infty(F_t(\omega)) \neq \triv$ for some $t \in \{0,1\}$, say for $t=0$.   In addition, $\Theta_0(G)$ is a nontrivial normal subgroup of the transitive group $F_0(\omega)$, so $\Theta_0(G)$ has no fixed points.  We see that $G$ and $G(\xi)$ both have at most two orbits on vertices of the tree, ensuring that they are compactly generated and do not preserve any proper subtree.  Consider now an arc $a = (v,w)$ directed away from $\xi$ and let $H = \rist_G(T_a)$.  If $w$ is of type $0$, then $H$ acts without fixed points on $S_v(w,1)$, whereas if $w$ is of type $1$, then every $x \in S_v(w,1)$ is of type $0$ and we see that $H$ acts without fixed points on $S_v(w,2)$.  In either case, $\rist_G(T_a)$ fixes only finitely many arcs of $T_a$.  By \cite[Proposition~4.6]{CapraceMarquisReid} it follows that $\rist_G(T_a)$ is a TMS subgroup of $G$ and of $G(\xi)$.  In particular, clearly $\rist_G(T_a)$ is nontrivial; from the freedom in the choice of $a$, we conclude that $G$ and $G(\xi)$ have micro-supported action on $\partial T$.
\end{proof}

\begin{proof}[Proof of Corollary~\ref{intro:perfect}]
Note that the hypothesis excludes the exceptional locally prosoluble type.  If $G$ has two orbits on vertices of the tree $T$, then $G$ is contained in an $\Aut(T)$-conjugate $U$ of $\mathbf{U}(F_0,F_1)$ by \cite[Theorem~6(iii)]{SmithDuke}.  If $G$ is vertex-transitive then it is contained in an $\Aut(T)$-conjugate $U$ of $\mathbf{U}(F_0)$ by \cite[Proposition~3.2.2]{BurgerMozes}.  On the other hand, given an arc $a$ of $T$, then $\Theta_t(G) = F_t(\omega)$ for all $t \in \{0,1\}$ by Theorem~\ref{intro:locally_prosoluble}, and we deduce that $G(a) = U(a)$; the local actions also ensure that $G$ acts transitively on each $U$-orbit of arcs, so we conclude that $G=U$.
\end{proof}

\begin{rem}\label{rem:lambda_restriction}
Theorem~\ref{intro:locally_prosoluble} reduces the cases of Lemma~\ref{lem:lambda_restriction}, as follows (recall also Table~\ref{table:small_q}).  If $G$ is of exceptional locally prosoluble type, then $F_t = \PGaL_3(5)$ falls under case (ii) or (vi) whereas $F_t = \PGaL_3(4)$ falls under case (ii).  If $G$ is not of exceptional locally prosoluble type, then cases (v) and (vi) cannot occur, and case (ii) can only occur for $q \in \{2,3\}$.  Example~\ref{ex:ldc_end} will provide an example of case (iv), an example of case (ii) for $q=2,3$, and an example of case (i) for infinitely many $n$ and $q$.  For each of the four families of $2$-transitive groups of rank $1$ Lie type, there could be (infinitely many) examples of case (iii), but constructing them is beyond the scope of this article.
\end{rem}

Theorem~\ref{intro:end_stabilizer:same} can be deduced from Theorem~\ref{thm:end_stabilizer} below, but it is easy enough to prove directly at this stage.

\begin{proof}[Proof of Theorem~\ref{intro:end_stabilizer:same}]
For $t \in \{0,1\}$, write $O_t = O^\infty(F_t(\omega))$.  We consider the cases of Lemma~\ref{lem:lambda_restriction} for $t = 0$.  Note that since $d_0=d_1$, the exceptional locally prosoluble case does not occur.

In case (vii) there is nothing to prove, and case (v) has been eliminated by Theorem~\ref{intro:locally_prosoluble}.  If $\PSL_3(q_0) \le F_0 \le \PGaL_3(q_0)$, with $F_0$ acting on $P_2(q_0)$, then the same description applies to $F_1$ with $q_0=q_1$, and by Lemma~\ref{lem:leftover_psl} we have $\Lambda_t = F_t(\omega)$ for all $t \in \{0,1\}$.  Thus we may assume $F_0$ is not of this form; this eliminates cases (ii) and (vi), and in case (i) we have $n \ge 3$.

In cases (i) and (iv) of Lemma~\ref{lem:lambda_restriction}, then $F_0(\omega) \cong F_1(\omega)$ is insoluble.  The conclusion follows by Corollary~\ref{cor:insoluble_line_index}.

In case (iii), $F_0 \cong F_1$ is defined over a field of characteristic $p$, and we see that $d_0-1=d_1-1$ is a power of $p$, so $|F_t(\omega): \Lambda_t|$ is a power of $p$ by Lemma~\ref{lem:line_index}.  On the other hand, by Theorem~\ref{thm:2bbtrans}, the index $|F_t(\omega): \Lambda_t|$ is coprime to $p$.

Thus $F_t(\omega) = \Lambda_t$ for all $t \in \{0,1\}$.
\end{proof}

We now give the more precise version of Theorem~\ref{intro:end_stabilizer}.  When combined with Lemma~\ref{lem:end_stabilizer:excep}, the next theorem is a strengthening of Lemma~\ref{lem:lambda_restriction} and also effectively incorporates Theorem~\ref{intro:end_stabilizer:same}.

\begin{thm}\label{thm:end_stabilizer}
Let $G \in \mc{H}_{F_0,F_1}$ and write $O_t = O^{\infty}(F_t(\omega))$.  Write $\lambda_t = |F_t(\omega):\Lambda_t(G)|$ and suppose that $\lambda_t > 1$ for some given $t \in \{0,1\}$, but that $G$ is not of exceptional locally prosoluble type.  Then one of the following holds.
\begin{enumerate}[(i)]
\item We have $\PSL_{n+1}(q_t) \unlhd F_t \le \PGaL_{n+1}(q_t)$ with $F_t$ acting on $P_n(q_t)$, with $n \ge 2$, $q_t > 2$; if $q_t=3$ we suppose $n \ge 3$.  In this case, $\Lambda_t(G)$ is a point stabilizer of a $2$-by-block-transitive action of $F_t$ as in case (a) of Theorem~\ref{thm:2bbtrans} and $\lambda_t$ divides $q_t-1$.
\item The socle of $F_t$ is of rank $1$ simple Lie type.  In this case $\Lambda_t(G)$ is a point stabilizer of a $2$-by-block-transitive action of $F_t$ as described in \cite[Proposition~3.21]{Reidkblock} (case (c) of Theorem~\ref{thm:2bbtrans}).  In particular, $\lambda_t$ divides $\mathrm{gcd}(q_t-1,e_{F_t})$, where $F_t$ has (projective) matrix coefficients in $\Fb_{q_t}$.
\item The tuple $(F_t,\Lambda_t(G),d_t,\lambda_t)$ is one of the first six lines of Table~\ref{table:small_q}; consequently, $\lambda_t \in \{2,3,6\}$.
\end{enumerate}
We recall the notation and conclusion of Lemma~\ref{lem:line_index} and write
\[
\lambda_t = \mu_t \nu_{1-t}, \text{ where } \mu_i = \prod_{k \ge 1} \delta^{2k}_i \text{ and } \nu_i = \prod_{k \ge 1} \delta^{2k-1}_i.
\]
Then $\mu_t = 1$ (in other words, $\lambda_t = \nu_{1-t}$) in all cases except possibly when $F_t = \PGaL_3(3)$ and $\lambda_t=6$, in which case $\mu_t \in \{1,2\}$.  Moreover, one of the following holds.
\begin{enumerate}[(a)]
\item $F_{1-t}$ is soluble, with $d_{1-t} \equiv 1 \mod p$ for all prime divisors $p$ of $\nu_{1-t}$.
\item $F_{1-t}$ is of exceptional $\SL_2(5)$ affine type and $\nu_{1-t}$ is a power of $\delta^*(F_{1-t}) \in \{2,3,7,29\}$.
\item The socle of $F_{1-t}$ is of rank $1$ simple Lie type over a field of characteristic $p$ and $\nu_{1-t}$ is a power of $p$.
\item $F_{1-t} = \PGaL_2(8)$ acting on $28$ points and $\nu_{1-t}$ is a power of $3$.
\item $F_{1-t} = \PGaL_3(p)$ with natural action for $p \in \{2,3\}$, $F_t \not\cong F_{1-t}$ and $\nu_{1-t} = 2^a 3^b$ for some $a,b \ge 0$.  If $p=2$ then $a \le b$, whereas if $p=3$ we have $a > b$.  If $F_t = \PGaL_3(3)$ then $\lambda_t \in \{1,3,6\}$.
\end{enumerate}
\end{thm}

\begin{proof}
Without loss of generality, let us assume $t=0$.  For $i \in \{0,1\}$, write $O_i = O^\infty(F_i(\omega))$.

By Theorem~\ref{intro:locally_prosoluble}, we now have $\Theta_i \ge O_i$ for $i =0,1$, so certainly $\Lambda_i \ge O_i$ and $\Delta^k_i \ge O_i$ for all $k \ge 1$.  In particular, for all $i \in \{0,1\}$ and $k \ge 1$, then the number of orbits $\delta^k_i$ of $\Delta^k_i$ must divide the number of orbits $\delta^*(F_i)$ of $O_i$.  By Lemma~\ref{lem:lambda_restriction}, $F_0$ is almost simple, so by Lemma~\ref{lem:soluble_residual_trans}(iii), if $F_0(\omega)$ is insoluble then $\delta^k_0=1$ for all $k$.  We have $\lambda_0 = \mu_0\nu_1$ by Lemma~\ref{lem:line_index}; in turn, $\mu_i$ and $\nu_i$ are divisors of $\delta^*(F_i)$.

We next prove that $\mu_0=1$, except for the special case where we have $\mu_0 \in \{1,2\}$; this will also lead to the cases (i)--(iii) of the present theorem.

As in the proof of Theorem~\ref{intro:end_stabilizer:same}, we consider the cases of Lemma~\ref{lem:lambda_restriction}; note that case (vii) is not applicable since $\lambda_0 > 1$.  We will deal with $F_0 \in \{\PGaL_3(2),\PGaL_3(3),\mathrm{M}_{11}\}$ in natural action separately.  Those cases aside, given Remark~\ref{rem:lambda_restriction}, we only need to consider cases (i) and (iii) of Lemma~\ref{lem:lambda_restriction}.

Suppose case (i) of Lemma~\ref{lem:lambda_restriction} applies with $q_0 \ge 4$.  In this case, $F_0(\omega)$ is insoluble, so $\delta^*(F_0)=1$ and hence $\mu_0=1$.  By Theorem~\ref{thm:2bbtrans} we see that $\lambda_0$ divides $q_0-1$,  so case (i) of the present theorem holds.

Suppose case (iii) of Lemma~\ref{lem:lambda_restriction} applies: $F_0$ has socle of rank $1$, with (projective) matrix coefficients in a field of order $q_0 = p^e_0$.  By Theorem~\ref{thm:2bbtrans}, $\lambda_0$ also divides $q_0-1$ and $e_{F_0}$.  In this case, $\delta^*(F_0) = d_0-1$ is a power of $p_0$, which is coprime to $\lambda_0$, so $\mu_0 = 1$ and hence $\lambda_0$ divides a power of $\delta^*(F_1)$.    So case (ii) of the present theorem holds.

We are left with $F_0 \in \{\PGaL_3(2),\PGaL_3(3),\mathrm{M}_{11}\}$ in natural action.  The first six lines of Table~\ref{table:small_q} account for all possible values of $\Lambda_0(G)$ and hence of $\lambda_0$; one observes that $\lambda_0 \in \{2,3,6\}$ in all cases.  Thus case (iii) of the present theorem applies.  We also note that if $F_i = \PGaL_3(p_i)$ for $p_i \in \{2,3\}$, then by Lemma~\ref{lem:small_PSL3}, for all $k$ we have $\delta^k_i \in Z_{p_i}$ where $Z_2 = \{1,3,6\}$ and $Z_3 = \{1,2,4,12\}$; on the other hand if $F_i = \mathrm{M}_{11}$ then $O_i$ is transitive, so $\delta^*(F_i)=1$.

If $F_0 = \PGaL_3(2)$ then $\lambda_0=2$, but $2$ is not a possible value of $\delta^k_0$; we deduce that $\mu_0=1$.

If $F_0 = \PGaL_3(3)$ then $\lambda_0 \in \{2,3,6\}$, whilst $\delta^k_0 \in \{1,2,4,12\}$ for all $k$; hence $\mu_0 \in \{1,2\}$.  If $\mu_0 = \lambda_0=2$, then $\nu_1=1$, and taking $k$ minimal such that $|F_0(\omega):\Lambda^k_0|>1$, we would need $|\Lambda^{k-1}_0:\Lambda^k_0| =2$ and $\delta^{k-1}_0=2$; by Lemma~\ref{lem:small_PSL3} we then have $|F_0(\omega): \Lambda^{k-1}_0| > 1$, which is impossible.  The only remaining possibility to have $\mu_0 >1$ is if $\mu_0=2$ and $\lambda_0=6$.

We now consider the possibilities for $F_1$ in terms of the value of $\nu_1$, which is necessarily a divisor of a power of $\delta^*(F_1)$.

If $F_1$ is soluble, we see that case (a) must apply.

If $F_1(\omega)$ is insoluble, then we must fall under case (ii) of Lemma~\ref{lem:soluble_residual_trans}, which yields case (b) of the present theorem.

From now on we may assume $F_1$ is almost simple and $F_1(\omega)$ is soluble; in particular, $\delta^*(F_1) = d_1-1$.

If $F_1$ has socle of rank $1$ Lie type, then $\delta^*(F_1)$ is a power of the defining characteristic $p$, so $\nu_1$ is also a power of $p$.  Thus case (c) holds.

If $F_1$ is $\PGaL_2(8)$ acting on $28$ points, then $\delta^*(F_1) = 3^3$ and we are in case (d) of the theorem.

If $F_1 = \PGaL_3(p)$ for $p \in \{2,3\}$, then $\nu_1$ is a product of elements of $Z_p$; in turn, any product of elements of $Z_p$ will satisfy the given restrictions on the prime factorization of $\nu_1$.  By Lemma~\ref{lem:leftover_psl}, if $F_0 = \PGaL_3(q)$ for some $q \in \{2,3\}$, then $q_0 = 3$, $q_1 = 2$ and $\lambda_0 \in \{3,6\}$.  Thus case (e) holds.

There are no other possibilities for $F_1$: see for instance the classification of $2$-transitive groups given in \cite[\S7.7]{DixonMortimer}.
\end{proof}

\section{A construction of some examples of boundary-$2$-transitive groups}\label{sec:examples}

In this last section, we give a construction to show that the exceptional case of Theorem~\ref{intro:locally_prosoluble} occurs, as well as examples of $G \in \mc{H}_{F_0,F_1}$ (for several different pairs $(F_0,F_1)$) that are not in the exceptional case, but still have $\Lambda_0(G) \neq F_0(\omega)$.  The examples use several of the $2$-by-block-transitive actions listed in Table~\ref{table:small_q}; they are also inspired by a (more general) unpublished construction of Florian Lehner, R\"{o}gnvaldur M\"{o}ller and Wolfgang Woess.

Before getting into specific examples, let us describe the construction in general terms.

Given a graph $\Gamma$, let $A\Gamma$\index{A@$A\Gamma, \; AT$} be the set of arcs of $\Gamma$; given a vertex $v \in V\Gamma$, write $A_\Gamma(v)$ for the set of arcs with origin $v$.  Fix a finite $2$-transitive permutation group $F_t \le \Sym(\Omega_t)$ for $t = 0,1$ with $|\Omega_t|=d_t \ge 3$; throughout this section we will assume $\Omega_0$ and $\Omega_1$ are disjoint to simplify notation.  We can then form the group $\mathbf{U}(F_0,F_1)$\index{U@$\mathbf{U}(F_0,F_1)$} acting on a $(d_0,d_1)$-semiregular tree $T$.  The construction of $\mathbf{U}(F_0,F_1)$ (see \cite{SmithDuke} for more details) entails a colouring $\mc{L}:AT \rightarrow \Omega_0 \cup \Omega_1$ of the arcs of $T$ with the following properties: all arcs terminating at a given vertex have the same colour, and for each vertex $v$ of type $t$, $\mc{L}$ restricts to a bijection from $A_T(v)$ to $\Omega_t$.  Write $\Aut^+(T)$ for the group of automorphisms of $T$ that are type-preserving on vertices.  The \defbold{($\mc{L}$-)local action} of an automorphism $g \in \Aut^+(T)$ at a vertex $v$ of type $t$ is as follows:
\[
\sigma_{\mc{L},v}(g): \Omega_t \rightarrow \Omega_t; \; c \mapsto  \mc{L}g(\mc{L}|_{A_T(v)})\inv c.
\]
The group $\mathbf{U}_{\mc{L}}(F_0,F_1)$ is then defined as the set of all $g \in \Aut(T)$ such that $\sigma_{\mc{L},v}(g) \in F_t$ for all $v \in V_tT$ and $t \in \{0,1\}$.  By \cite[Proposition~11]{SmithDuke}, up to $\Aut(T)$-conjugacy, $\mathbf{U}_{\mc{L}}(F_0,F_1)$ does not depend on the choice of $\mc{L}$, so we can just write $\mathbf{U}(F_0,F_1) = \mathbf{U}_{\mc{L}}(F_0,F_1)$.   It is then straightforward to check that $U:= \mathbf{U}(F_0,F_1)$ is a closed subgroup of $\Aut(T)$ that belongs to $\mc{H}_{F_0,F_1}$.  More generally, given $G \le \Aut(T)$, the \defbold{local action of $G$ at $v$}\index{local action} is the set of all local actions achieved by elements of $G(v)$; the group $U$ itself has local action $F_t$ at vertices of type $t$ by \cite[Lemma~13]{SmithDuke}.

The following basic observation is key to the local actions approach.

\begin{lem}\label{lem:local_action:determined}
Let $g,h \in \Aut^+(T)$.  If $gv = hv$ for some $v \in VT$, and for all $w \in VT$ we have $\sigma_{\mc{L},w}(g) = \sigma_{\mc{L},w}(h)$, then $g=h$.

More precisely, if $gv = hv$ for some $v \in VT$ and $\sigma_{\mc{L},v}(g) = \sigma_{\mc{L},v}(h)$, then also $ga = ha$ for all $a \in A_T(v)$, and hence $gw = hw$ for all $w \in S(v,1)$.
\end{lem}

\begin{proof}
The second statement is clear from the fact that $\mc{L}$ restricts to a bijection from $A_T(v)$ to $\Omega_t$ (where $v$ has type $t$).  The first statement then follows from the second by considering the distance $d$ from $v$ to some vertex/arc $a$ on which $g$ and $h$ might disagree, and arguing by induction on $d$.
\end{proof}

Our next goal is to build a subgroup of $U = \mathbf{U}(F_0,F_1)$ that still has the same local actions, but such that if we fix a pair of adjacent vertices $x$ and $y$ of types $0$ and $1$ respectively, then $G(x,y)/G_1(x,y)$ will be strictly smaller than $U(x,y)/U_1(x,y)$.  The set of all pairs of vertices $(x',y')$, with $x' \neq y$ adjacent to $x$ and $y' \neq x$ adjacent to $y$, is $S_y(x,1) \times S_x(y,1)$.  The group $U(x,y)$ acts as the full direct product $F_0(\omega_0) \times F_1(\omega_1)$, where $\omega_0$ and $\omega_1$ are respectively the colours of the arcs from $x$ to $y$ and from $y$ to $x$.  Since $G$ has local actions $F_t$, the group $G(x,y)$ will project onto $F_0(\omega_0)$ and $F_1(\omega_1)$, but it need not be the full direct product; in fact, we will force it to be a given subdirect product.  By Goursat's lemma (recall also Lemma~\ref{lem:Goursat}), this means $G(x,y)/G_1(x,y)$ will occur as the pullback of a pair of surjective homomorphisms $\psi_0: F_0(\omega_0) \rightarrow B$ and $\psi_1: F_1(\omega_1) \rightarrow B$ to some common quotient $B$.

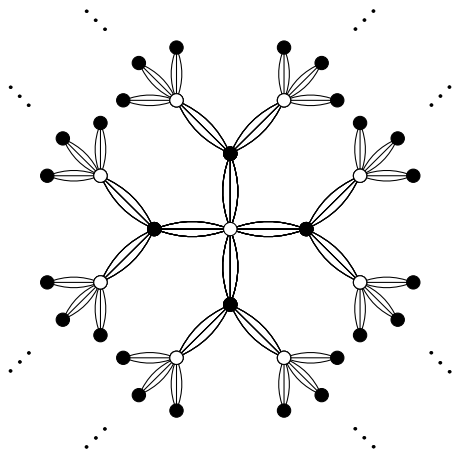
\begin{figure}[h]
\caption{The graph $T^*$ with $|B|=3$}
\label{fig:Tstar}
\begin{center}
\begin{tikzpicture}[scale=1, every loop/.style={}, square/.style={regular polygon,regular polygon sides=4}]

\tikzstyle{every node}=[circle,
                        inner sep=0pt, minimum width=5pt]
                     
\foreach \x in {0,1,2,3} {
\foreach \y in {-1,1} {
\foreach \z in {-1,0,1} {
\draw (0,0) to [bend left=20] ++ (90*\x:1) to [bend left=20] ++ (90*\x+45*\y:1) to [bend left=20] ++ (90*\x+45*\y+45*\z:0.7);
\draw (0,0) to [bend right=20] ++ (90*\x:1) to [bend right=20] ++ (90*\x+45*\y:1) to [bend right=20] ++ (90*\x+45*\y+45*\z:0.7);
\draw (0,0) to ++ (90*\x:1) node[draw=black, fill=black]{} to ++ (90*\x+45*\y:1) node[draw=black, fill=white]{} to ++ (90*\x+45*\y+45*\z:0.7) node[draw=black, fill=black]{};
\draw (0,0) ++ (90*\x:1) ++ (90*\x+45*\y:2.5) node[rotate=90*\x+45*\y]{$\ldots$};
}
}
};

\draw (0,0) node[draw=black, fill=white]{} to (0,0);
     
\end{tikzpicture}
\end{center}
\end{figure}

\begin{const}\label{const:edge}

Recall the more detailed permutation groups terminology set up in Definition~\ref{def:permutation} and Remark~\ref{rem:standard_extension}.  We suppose that we are given a pair of surjective homomorphisms $\psi_0: F_0(\omega_0) \rightarrow B$ and $\psi_1: F_1(\omega_1) \rightarrow B$ for $\omega_t \in \Omega_t$.  We can now form a standard extension $\Omega_t \times B$ of the $F_t$-set $\Omega_t$ in the manner of Remark~\ref{rem:standard_extension}, as follows.  We identify $\Omega_t$ with the coset space $F_t/F_t(\omega_t)$.  Take a set $A_t$ of left coset representatives of $F_t(\omega_t)$, with $a_\omega$ being the representative of the coset $\omega$; we then map $(\omega,b)$ to the left coset $a_\omega\psi\inv_t(b)$ of $R_t:=\ker\psi_t$ in $F_t$, thus obtaining a bijection from $\Omega_t \times B$ to $F_t/R_t$.  The induced action $\alpha_t$ of $F_t$ on $\Omega_t \times B$ is then a standard extension of $\Omega_t$.  For each $\omega \in \Omega_t$, we have a function $\psi_\omega: F_t \rightarrow B$ specified by
\[
\forall f \in F_t, b \in B: f(\omega,b) = (f\omega,\psi_\omega(f)b).
\]
The function $\psi_\omega$ is an instance of a block restriction map as in Remark~\ref{rem:standard_extension}; in the present context, the associated block action is simply the regular action of $B$ on itself.  Note that $\psi_\omega$ restricts to a surjective homomorphism from $F_t(\omega)$ to $B$; in particular, if $\omega = \omega_t$ then $\psi_\omega$ restricts to $\psi_t$.  Given $\omega \in \Omega_t$, set $R_\omega:= \{f \in F_t(\omega) \mid \psi_\omega(f) =1\}$; in particular, $R_{\omega_t} = R_t$.  We can also interpret $R_\omega$ as the stabilizer of $(\omega,b)$ in the action of $F_t$ on $\Omega_t \times B$, where $b$ can be taken to be any element of $B$.

Let $T$ be the $(d_0,d_1)$-semiregular tree, and given $v \in VT$, write $t(v)$\index{t@$t(v)$} for the type of $v$.  We now introduce an auxiliary graph $T^*$, which has its own arc colouring
\[
\mc{L}^*: AT^* \rightarrow (\Omega_0 \cup \Omega_1) \times B.
\]
The graph $T^*$ has the same vertices as $T$, but each arc of $T$ coloured $\omega$ is replaced with a set of arcs coloured $(\omega,b)$, one for each $b \in B$, in such a way that if the arc $a$ has colour $(\omega_a,b)$, then the reverse arc $\ol{a}$ has some colour $(\omega_{\ol{a}},b)$ (see Figure~\ref{fig:Tstar}).  There is then a natural projection from $T^*$ to $T$ given by fixing vertices and sending each arc of $T^*$ to the arc of $T$ with the same origin and terminus.

We define the local action of an automorphism $g \in \Aut^+(T^*)$ at a vertex $v$ of type $t$ as follows:
\[
\sigma_{\mc{L}^*,v}(g): \Omega_t \times B \rightarrow \Omega_t \times B; \; c \mapsto  \mc{L}^*g(\mc{L}^*|_{A_{T^*}(v)})\inv c.
\]
Lemma~\ref{lem:local_action:determined} still applies in this context, replacing $T$ with $T^*$ and $\mc{L}$ with $\mc{L}^*$.  We equip $\Aut^+(T^*)$ with the permutation topology on arcs of $T^*$.  Define $G^*$ to be the set of $g \in \Aut^+(T^*)$ such that $\sigma_{\mc{L}^*,v}(g) \in \alpha_t(F_t)$ for all $v \in V_tT$; given $g \in G^*$, write $\sigma^*_v(g)$ for the element of $F_t$ such that $\alpha_t(\sigma^*_v(g)) = \sigma_{\mc{L}^*,v}(g)$.  Finally, the projection of $T^*$ to $T$ induces a homomorphism $\pi: \Aut^+(T^*) \rightarrow \Aut(T)$ and define
\[
G:= \mathbf{U}(F_0,F_1,\alpha_0,\alpha_1) := \pi(G^*).
\]
\end{const}

For the rest of this section we assume $G$ has been constructed as in Construction~\ref{const:edge}.  The next proposition establishes some general features of $G$.

\begin{prop}\label{prop:edge_construction}
Let $G$ be as in Construction~\ref{const:edge}.  Given adjacent vertices $x$ and $y$ of $T$, write $\omega_{xy} = \mc{L}(a)$ where $a$ is the arc of $T$ from $x$ to $y$.
\begin{enumerate}[(i)]
\item $G$ is a closed subgroup of $\mathbf{U}(F_0,F_1)$.
\item The restriction of $\pi$ to $G$ is injective.
\item Let $T'$ be a nonempty subtree of $T$.  Define a \defbold{$T'$-portrait} to be a function $p: VT' \rightarrow F_0 \sqcup F_1$, where $p(v) \in F_t$ if $v$ has type $t$, such that for any two adjacent vertices $x$ and $y$ of $T'$ of types $0$ and $1$ respectively, we have
\[
p(x) \in F_0(\omega_{xy}), \; p(y) \in F_1(\omega_{yx}), \; \psi_{\omega_{xy}}(p(x)) = \psi_{\omega_{yx}}(p(y)).
\]
Then for every $T'$-portrait $p$, there is $g \in G^*$ fixing $T'$ pointwise, such that $\sigma^*_{v}(g) = p(v)$ for all $v \in VT'$.  Conversely, given any $g \in G^*$ fixing $T'$ pointwise, then the map $v \mapsto \sigma^*_v(g)$ defines a $T'$-portrait.
\item Let $x$ and $y$ be adjacent vertices of type $0$ and $1$ respectively.  Then $G(x,y)$ acts on $S_y(x,1) \times S_x(y,1)$ as the following subdirect product:
\[
\{(f_0,f_1) \in F_0(\omega_{xy}) \times F_1(\omega_{yx}) \mid \psi_{\omega_{xy}}(f_0) = \psi_{\omega_{yx}}(f_1)\}.
\]
\item Given $t \in \{0,1\}$ then $G$ has local action $F_t$ at vertices of type $t$ and $G$ is $2$-type-distance-transitive.
\item $G$ has the independence property $\propP{2}$ (recall Remark~\ref{rem:Pk}).
\end{enumerate}
\end{prop}

\begin{proof}
For all $g,h \in \Aut^+(T^*)$ and $v \in VT$, the following holds:
\[
\sigma_{\mc{L}^*,v}(gh\inv) = \sigma_{\mc{L}^*,h\inv v}(g) \sigma_{\mc{L}^*,v}(h\inv) = \sigma_{\mc{L}^*,h\inv v}(g) \sigma_{\mc{L}^*,h\inv v}(h)\inv.
\]
Thus $G^*$ is a subgroup of $G$, so $\pi(G^*)$ is a subgroup of $\Aut(T)$.  Since $\alpha_t$ is compatible with the action of $F_t$ on $\Omega_t$, it is easy to see that $G \le \mathbf{U}(F_0,F_1)$.

From the definition of $G^*$, it is clear that $G^*$ is closed in the permutation topology, so $G^*$ has compact arc stabilizers.  It follows that $G$ has compact arc stabilizers, so $G$ is closed in $\Aut(T)$, completing the proof of (i).

For (ii), consider $g \in G^*$ such that $\pi(g) = 1$.  Then $g$ fixes every vertex; in particular, given $t \in \{0,1\}$ and $v \in V_tT$, then $g$ fixes $v$ and all neighbours of $v$.  Writing $\sigma_{\mc{L}^*,v}(g) = \alpha_t(f)$, then $\alpha_t(f)$ preserves each fibre of the projection map $\phi$ from $\Omega_t \times B$ to $\Omega_t$.  Since $\phi$ intertwines $\alpha_t$ with the action of $F_t$ on $\Omega_t$, it follows that $f$ acts trivially on $\Omega_t$, and since $F_t \le \Sym(\Omega_t)$ we thus have $f=1$, so $\alpha_t(f)=1$.  Thus $g$ has trivial local action at every vertex, from which we conclude that $g$ is the trivial automorphism of $T^*$, as required.

For (iii), given $g \in G^*$ fixing $T'$ pointwise, and given adjacent vertices $x$ and $y$ as in the statement, we certainly need to have
\[
\sigma^*_x(g) \in F_0(\omega_{xy}), \; \sigma^*_y(g) \in F_1(\omega_{yx}), \; \psi_{\omega_{xy}}(\sigma^*_x(g)) = \psi_{\omega_{yx}}(\sigma^*_y(g)).
\]
Say $g$ sends the arc $a$ from $x$ to $y$ coloured $(\omega_{xy},1)$ to an arc coloured $(\omega',b)$, then $\omega' = \omega_{xy}$ since $ga$ is an arc from $x$ to $y$.  Moreover, $\ol{a}$ has colour $(\omega_{yx},1)$ and $\ol{(ga)}$ has colour $(\omega_{yx},b)$, so we have
\[
\psi_{\omega_{xy}}(\sigma^*_x(g)) = b = \psi_{\omega_{yx}}(\sigma^*_y(g)).
\]
Thus the map $v \mapsto \sigma^*_v(g)$ defines a $T'$-portrait.
For the other direction, suppose we are given a $T'$-portrait $p$; we aim to construct an element $g \in G^*$ yielding this $T'$-portrait.  We decide where $g$ sends vertices and arcs working outwards from the subtree $T'$, which is fixed pointwise by assumption; we will choose local actions $\sigma^*_w(g) = f_w$ for every vertex $w$, where $f_w = p(w)$ if $w \in VT'$.  We suppose that we have chosen $gw$ and the local action $\sigma^*_w(g)=f_w$ exactly when $w \in VT''$, where $T''$ is a subtree of $T$ containing $T'$.  If $T'' = T$ we are done; otherwise, let $w \in VT''$ and suppose $w'$ is some neighbour of $w$ that does not belong to $T''$.  There is a set of $T^*$-arcs from $w$ to $w'$ with colours $\{\omega\} \times B$ for some $\omega \in \Omega_{t(w)}$; the local action $f_w$ then specifies that if $a$ is the arc coloured $(\omega,b)$ originating at $w$, then $ga$ must be the arc coloured $f_w(\omega,b)$ originating at $gw$.  All such arcs $ga$ terminate at the same vertex $w''$, which we set as $gw'$.  The $T^*$-arcs going from $w'$ to $w$ have colours $\{\omega_1\} \times B$ and those from $gw$ to $gw'$ have colours $\{\omega_2\} \times B$ for some $\omega_1,\omega_2 \in \Omega_{t(w')}$.  Let $h \in F_{t(w')}$ be such that $h(\omega_1,b) = (\omega_2,b)$ for all $b \in B$ (using the fact that the action on $\Omega_{t(w')} \times B$ is a standard extension), and let $h' \in F_{t(w')}(\omega_2)$ be such that $\psi_{\omega_2}(h') = \psi_{\omega}(f_w)$ (which is possible as $\psi_{\omega_2}(F_{t(w')}(\omega_2)) = B$).  Now set $f_{w'} = h'h$; thus if $a$ is the arc originating at $w'$ coloured $(\omega_1,b)$, then we set $ga$ to be the arc originating at $gw'$ with the colour
\[
f_{w'}(\omega_1,b) = (\omega_2,\psi_{\omega}(f_w)b).
\]
The choice of $f_{w'}$ in terms of $f_w$ ensures that we have respected edge reversal, in other words, given an arc $a$ from $w$ to $w'$ and its reverse $\ol{a}$ from $w'$ to $w$, then we have set $g(\ol{a}) = \ol{(ga)}$.  We now have compatible choices of local actions $f_z$ for $z \in VT'' \cup \{w'\}$.

By iterating this procedure, we can choose suitable $gw$ and $\sigma^*_w(g)$ for all $w \in VT''$, where $T''$ is any tree containing $T'$ such that $VT'' \setminus VT'$ is finite.  Since $G^*$ is closed in $\Aut(T^*)$ with compact vertex stabilizers, a compactness argument then allows us to make consistent choices of images of vertices and local actions over the whole tree $T$.

Part (iv) is a special case of (iii).

For (v), we note that it is clear from (i) and (iii) that $G$ has local action $F_t$ at vertices of type $t$.  Since $F_t$ is transitive on $\Omega_t$, it follows by Lemma~\ref{lem:edge-walk} that $G$ is $1$-type-distance-transitive.  Moreover, given a pair of adjacent vertices $x$ and $y$ as in part (iv), we see that $G(x,y)$ acts transitively on both $S_y(x,1)$ and $S_x(y,1)$; thus in fact $G$ is $2$-type-distance-transitive (by a similar argument to the proof of Proposition~\ref{prop:type-distance-trans}(iii)).

For part (vi) we must show $G = G^{\propP{2}}$, that is, if $g \in \Aut(T)$ acts as some element of $G$ on each ball of $T$ radius $2$, then $g \in G$.  It is clear that $G^{\propP{2}} \le \mathbf{U}(F_0,F_1)$.  Since $G$ is transitive on vertices of a given type, it is enough to take a vertex $v_0 \in VT$ and show that $G^{\propP{2}}(v_0) \le G$.  Let $g \in G^{\propP{2}}(v_0)$, and let $p: VT \rightarrow F_0 \sqcup F_1$ be given by setting $p(v) = \sigma_{\mc{L},v}(g) \in F_t$ for $v \in VT$ of type $t$.  We now construct $h \in \Aut^+(T^*)$ such that $\pi(h) = g$.  Certainly, if we set $hv = gv$ for every $v \in VT^*$, then $h$ will respect adjacency of vertices in $T^*$.  Let $x$ be a vertex of type $t$ and $y$ a neighbour of $x$, and suppose the colour of the $T$-arc from $x$ to $y$ and from $y$ to $x$ is $\omega_{xy}$, respectively $\omega_{yx}$.  Then in $T^*$, we set the image under $h$ of the arc from $x$ to $y$ coloured $(\omega_{xy},b)$ to be the arc from $hx$ to $hy$ coloured $p(x)(\omega_{xy},b)$ (where now $p(x) \in F_t$ is acting on $\Omega_t \times B$); we do this for all pairs of adjacent vertices $x$ and $y$ of $T^*$.  The only way this can fail to define an automorphism of $T^*$ is if we fail to respect edge-reversal: in other words, if there is some pair of adjacent vertices $x$ and $y$ and some $b \in B$ such that
\[
\psi_{\omega_{xy}}(p(x)) \neq \psi_{\omega_{yx}}(p(y)).
\]
Suppose this occurs, and consider what $g$ does to $B_T(x,2)$, which we define as the subtree of $T$ spanned by all vertices at distance at most $2$ from $x$.  Since $g \in G^{\propP{2}}$ there is $g' \in G^*$ such that $\pi(g')a = ga$ for every arc $a$ of $B_T(x,2)$.  By comparing the colours of arcs of $B_T(x,2)$ with the colours of the images of the arcs, we see that 
\[
\sigma_{\mc{L},x}(g) = \sigma_{\mc{L},x}(\pi(g')) \text{ and } \sigma_{\mc{L},y}(g) = \sigma_{\mc{L},y}(\pi(g')).
\]
We then have $\sigma_{\mc{L},x}(\pi(g')) = \sigma^*_x(g')$ as elements of $F_t$, and similarly for $y$ in place of $x$.  But then
\[
\psi_{\omega_{xy}}(\sigma^*_x(g')) \neq \psi_{\omega_{yx}}(\sigma^*_y(g')),
\]
contradicting part (iii).  From this contradiction, we deduce that $h$ is indeed an automorphism of $T^*$; it is then clear from how $h$ is constructed that $h \in G^*$ and $\pi(h) = g$.  Thus $G^{\propP{2}}(v_0) \le G$, which completes the proof.
\end{proof}

In particular, for any subdirect product $D$ of $F_0(\omega_0)$ and $F_1(\omega_1)$, we have constructed a closed subgroup $G$ of $\mathbf{U}(F_0,F_1)$ such that $G(x,y)$ acts on $S_y(x,1) \times S_x(y,1)$ as $D$ and such that $G$ still has local actions $F_t$.  However, whether or not $G$ actually acts $2$-transitively on $\partial T$ is more sensitive to the choice of subdirect product.  By Corollary~\ref{cor:distance_transitivity}, we have $G \in \mc{H}_T$ if and only if $G$ is $k$-type-distance-transitive for all $k$; determining whether the latter property holds requires a more detailed consideration of portraits in the sense of Proposition~\ref{prop:edge_construction}(iii).

Note that given $\omega,\omega' \in \Omega_t$, then $fR_\omega f\inv = R_{\omega'}$ for any $f \in F_t$ such that $f\omega = \omega'$.  It will be useful to also consider the following subgroups of $B$ for $\omega,\omega' \in \Omega_t$ with $\omega \neq \omega'$:
\[
K_{t,\omega,\omega'} := \psi_\omega(F_t(\omega,\omega')); \; K'_{t,\omega,\omega'} = \psi_\omega(F_t(\omega) \cap R_{\omega'}).
\]
Since $F_t$ is $2$-transitive on $\Omega_t$, one sees that up to conjugacy in $B$, the groups $K_{t,\omega,\omega'}$ and $K'_{t,\omega,\omega'}$ do not depend on the choice of the pair $(\omega,\omega')$.  For the purposes of our argument it will often be enough to understand these subgroups up to conjugacy, so we write $K_t = K_{t,\omega,\omega'}$ and $K'_t = K'_{t,\omega,\omega'}$.  Given $\omega'' \in \Omega_{1-t}$ (by default we will take $\omega'' = \omega_{1-t}$), we can then define 
\[
L_{1-t} = F_{1-t}(\omega'') \cap \psi\inv_{\omega''}(K_t) \text{ and } L'_{1-t} = F_{1-t}(\omega'') \cap \psi\inv_{\omega''}(K'_t);
\]
one sees that up to conjugacy, $L_{1-t}$ and $L'_{1-t}$ do not depend on the choice of $(\omega,\omega',\omega'')$.  In particular, we consider $L_{1-t}$ and $L'_{1-t}$ as permutation groups acting on $\Omega_{1-t} \setminus \{\omega''\}$; the permutational properties of these actions do not depend on $(\omega,\omega',\omega'')$.  Likewise the permutational properties of $F_t$ acting on $F_t/L_t$ and $F_t/L'_t$ do not depend on $(\omega,\omega',\omega'')$.  We remark that the index $|F_{1-t}(\omega''):L_{1-t}|$ counts the number of orbits of $R_\omega$ on $\Omega_t \setminus \{\omega\}$: writing $\sim_{R_\omega}$ for the $R_\omega$-orbit relation on $\Omega_t \setminus \{\omega\}$, then
\[
|F_{1-t}(\omega''):L_{1-t}| =  |B:K_t| = |F_t(\omega): F_t(\omega,\omega')R_\omega| = |(\Omega_t \setminus \{\omega\})/\sim_{R_\omega}|.
\]

The next lemma provides a necessary condition to have $G \in \mc{H}_T$.

\begin{lem}\label{lem:edge_construction:necessary}
Let $G$ be as in Construction~\ref{const:edge}.  For each $t \in \{0,1\}$ then $\Lambda^2_t = L_t$.  Moreover, letting $x$ and $y$ be adjacent vertices of $T$, then the following are equivalent:
\begin{enumerate}[(i)]
\item $G$ is $3$-type-distance-transitive;
\item $L_0$ and $L_1$ are transitive;
\item Given $z \in S_x(y,1)$ then $G(x,y,z)$ is transitive on $S_y(x,1)$, and given $z' \in S_y(x,1)$ then $G(z',x,y)$ is transitive on $S_x(y,1)$;
\item $G(x,y)$ acts transitively on $S_y(x,1) \times S_x(y,1)$.
\end{enumerate}
\end{lem}

\begin{proof}
Let $x$ and $y$ be adjacent vertices of $T$; without loss of generality, suppose $x$ is of type $0$, so that $G(x,y,z)$ acts as a copy of $\Lambda^2_0$ on $S_y(x,1)$.  We recover the action of $G(x,y,z)$ on $S_y(x,1)$ by considering $T'$-portraits $p$ where $T'$ is the tree spanned by $\{x,y\}$, such that the image of $p(y)$ in $F_1$ fixes $\omega_{yx}$ and $\omega_{yz}$.  Given that
\[
\psi_{\omega_{xy}}(p(x)) = \psi_{\omega_{yx}}(p(y)),
\]
we see that the possible values of $p(x)$ range over a copy of $L_0$.  The argument to show $\Lambda^2_1 = L_1$ is similar.

We now see that (i) and (ii) are equivalent via Corollary~\ref{cor:type-distance-trans}; in turn, given the description of $\Lambda^2_t$, we see that (ii) are (iii) are equivalent.  Finally, since $G(x,y)$ is transitive on each of $S_y(x,1)$ and on $S_x(y,1)$ taken separately, we see that (iii) and (iv) are equivalent.
\end{proof}

\begin{cor}
Let $G$ be as in Construction~\ref{const:edge} and suppose $G \in \mc{H}_T$.  Then $F_t$ acts $2$-by-block-transitively on $F_t/L_t$ for $t \in \{0,1\}$.
\end{cor}

\begin{proof}
Let $t \in \{0,1\}$.  By Lemma~\ref{lem:edge_construction:necessary} and the definition of $\Lambda_t$, we see that $L_t = \Lambda^2_t \ge \Lambda_t$.  By Corollary~\ref{cor:tree_2bbtrans}, $F_t$ acts $2$-by-block-transitively on $F_t/\Lambda_t$; thus $F_t$ also acts $2$-by-block transitively on the factor space $F_t/L_t$. (In both cases, the blocks are the orbits of $F_t(\omega_t)$, which contains $L_t$.)
\end{proof}

In particular, suppose our construction yields $G \in \mc{H}_T$.  Then for each $t \in \{0,1\}$, either $L_t$ is a point stabilizer for the action of $F_t$ on $\Omega_t$ (equivalently, $\ker\psi_{1-t}$ is transitive on $\Omega_{1-t} \setminus \{\omega_{1-t}\}$) or else $F_t/L_t$ is $2$-by-block-transitive with nontrivial blocks; in the latter case the possibilities for $F_t$ are restricted, since we must witness one of the cases of Theorem~\ref{thm:2bbtrans}.
%

Determining whether or not $G$ is $k$-type-distance-transitive for large $k$ is more complicated in general.  Let $t \in \{0,1\}$ and let $\rho$ be a ray with vertices $(v_t,v_{t+1},\dots)$, with $v_t$ of type $t$; let $\omega^+_i$ be the colour of the arc of $T$ from $v_i$ to $v_{i+1}$ and let $\omega^-_i$ be the colour of the arc of $T$ from $v_{i+1}$ to $v_i$.  Let $P = P_{t,k}$ be the path from $v_t$ to $v_{t+k}$.  Then a $P$-portrait assigns to each $v_i$ for $i \ge t$ a group element $h_i$ such that 
\begin{enumerate}[(a)]
\item We have $h_t \in F_t(\omega^+_t)$ and $h_{t+k} \in F_{t+k}(\omega^-_{t+k})$, and for $t < i < t+k$ we have $h_i \in F_i(\omega^+_i,\omega^-_i)$;
\item For $t \le i < t+k$, the following compatibility condition is satisfied:
\[
\psi_{\omega^+_i}(h_i) = \psi_{\omega^-_{i+1}}(h_{i+1}).
\]
\end{enumerate}
Interpreting Corollary~\ref{cor:type-distance-trans} in this context, the criterion for $G$ to be $(k+1)$-type-distance-transitive is as follows: for both types $t$, and for all $\omega^*_1,\omega^*_2 \in \Omega_t \setminus \{\omega^+_t\}$, there is a $P_{t,k}$-portrait $p$ such that $p(v_t)\omega^*_1 = \omega^*_2$.

The following gives a simpler sufficient condition to obtain $G \in \mc{H}_T$; we will see in Example~\ref{ex:ldc_end} however that this condition is not necessary.

\begin{lem}\label{lem:edge_construction:sufficient}
Let $G$ be as in Construction~\ref{const:edge} and suppose that $L'_0$ and $L'_1$ are transitive.  Then $G \in \mc{H}_T$.
\end{lem}

\begin{proof}
Retain the notation for the ray $\rho$ from above.  We consider $\rho$-portraits $p$ such that $p(v_i)=1$ for $i \ge t+2$; write $h_i = p(v_i)$.  Clearly the trivial local actions $h_{t+2},h_{t+3},\dots$ are all compatible with each other.  For $h_{t+1}$ to be compatible with $h_{t+2}$, we need
\[
h_{t+1} \in F_{1-t}(\omega^-_{t+1},\omega^+_{t+1}); \; \psi_{\omega^+_{t+1}}(h_{t+1}) = 1,
\]
in other words, $h_{t+1}$ can range over a copy of $K'_{1-t}$.  Then for $h_t$ to be compatible with $h_{t+1}$, we need
\[
\psi_{\omega^+_t}(h_t) = \psi_{\omega^-_{t+1}}(h_{t+1}),
\]
which allows $h_t$ to range over a copy of $L'_t$.  Since $L'_t$ is transitive by hypothesis, we know see that for all $\omega^*_1,\omega^*_2 \in \Omega_t \setminus \{\omega^+_t\}$, there is a $\rho$-portrait $p$ such that $p(v_t)\omega^*_1 = \omega^*_2$.  By Proposition~\ref{prop:edge_construction}(iii), all such portraits are witnessed by elements of $G$, with the result that the pointwise stabilizer of $\rho$ in $G$ acts transitively on $S_{v_{t+1}}(v_t,1)$.  Given Proposition~\ref{prop:edge_construction}(v), we now see that $G$ satisfies Corollary~\ref{cor:distance_transitivity}(v) and hence $G \in \mc{H}_T$.
\end{proof}

We conclude this section with two special cases of Construction~\ref{const:edge}.  In these examples, if we treat a vector space or field as a group, it should be read as the additive group of that vector space or field, whereas $\Fb^*$ is the multiplicative group of the field $\Fb$.

\begin{ex}\label{ex:ldc_end}
We will construct examples of $G \in \mc{H}_{F_0,F_1}$ for suitable $(F_0,F_1)$ such that $\Lambda_0$ is strictly smaller than $F_0$, and which do not satisfy Lemma~\ref{lem:edge_construction:sufficient}.  In Construction~\ref{const:edge} we have surjective homomorphisms $\psi_0: F_0(\omega_0) \rightarrow B$ and $\psi_1: F_1(\omega_1) \rightarrow B$; for this example, we will simplify the analysis somewhat by taking $\psi_1$ to be the identity map.

For $t = \{0,1\}$ let $F_t$ be a finite group with a $2$-transitive action on the coset space $\Omega_t = F_t/M_t$; let $d_t = |\Omega_t|$ and let $B = M_1$.  Write $\omega_t = M_t$ considered as a point in $\Omega_t$.  We suppose that there is some surjective homomorphism $\psi_0$ from $M_0$ to $B$, with kernel $R_0$; set $\psi_1 = \mathrm{id}_B$ and $R_1 = \triv$.  To make some transitivity arguments work later, we also suppose that the action of $F_0$ on $F_0/R_0$ is $2$-by-block-transitive (where the blocks are the $M_0$-orbits); this means that we do not need to specify $M_0$ directly, as it can be recovered as the unique largest proper subgroup of $F_0$ containing $R_0$.  One can for example take $(F_0,F_1,B,R_0)$ to be one of the following:
\begin{align*}
&(\mathrm{M}_{11},\Sym(3),\Sym(2),\Alt(6)); \; (\PGaL_3(2),\Sym(3),\Sym(2),C^2_2 \rtimes C_3);\\
&(\PGaL_3(3),\Sym(4),\Sym(3),C^2_3 \rtimes Q_8); \; (\PGL_{n+1}(q),\Fb_{q} \rtimes \Fb^*_{q},\Fb^*_{q},\Fb^n_{q} \rtimes \SL_n(q)),
\end{align*}
where for the last family of examples we take $q$ to be a prime power, $q \neq 2$ and $n+1$ to be coprime to $q-1$.

We then go through the construction of $G = \mathbf{U}(F_0,F_1,\mc{L}^*,\alpha_0,\alpha_1)$ as in Construction~\ref{const:edge}.  The fact that $F_0$ acts $2$-by-block-transitively on $F_0/R_0$ implies that $R_0$ acts transitively on $\Omega_0 \setminus \{\omega_0\} \times B$, and hence also on $\Omega_0 \setminus \{\omega_0\}$.  We then consider the associated groups $L_t,L'_t \le F_t(\omega_t)$.

In the present case, we see that $L_0 \ge L'_0 = R_0$.  Since $R_0$ is transitive on $\Omega_0 \setminus \{\omega_0\}$, we have $L_1 = B$, and we see that $L_t$ acts transitively on $\Omega_t \setminus \{\omega_t\}$ for $t \in \{0,1\}$.  On the other hand, $L'_1$ is trivial.  In particular, $L'_1$ is not transitive on $\Omega_1 \setminus \{\omega_1\}$.

We retain the notation for the ray $\rho = (v_t,v_{t+1},\dots)$ from earlier; let us suppose without loss of generality that $\omega^+_t$ is the trivial coset $M_t$.  To show $G \in \mc{H}_T$ it is sufficient to show the following: for all $\omega^*_1,\omega^*_2 \in \Omega_t \setminus \{\omega^+_t\}$, there is a $\rho$-portrait $p$ such that $p(v_t)\omega^*_1 = \omega^*_2$.  Indeed, we make the stronger claims that if $t=0$ we can take $p(v_t)$ to be any element of $R_0$, whereas if $t=1$, we can take $p(v_t)$ to be any element of $F_1(\omega^+_1)$.  We divide the argument according to the type $t$.
\begin{enumerate}[(i)]
\item If $t=0$, we take $p(v_t) \in R_0$, so that $\psi_{\omega^+_0}(p(v_0))=1$.  This is compatible with trivial local action at $v_1$, so for $i \ge 2$ we set $p(v_i)=1$.
\item If $t=1$, suppose we are given some $h_1 \in F_1(\omega^+_1)$; set $p(v_1) = h_1$ and write $b = \psi_{\omega^+_1}(h_1)$.  The compatibility condition for $p(v_2)$ is that $\psi_{\omega^-_2}(p(v_2)) = b$, in other words, $p(v_2)$ is confined to a left coset of $R_{\omega^-_2}$.  Since the action of $F_0$ on $\Omega_0 \times B$ is $2$-by-block-transitive, we have a transitive action of $R_{\omega^-_2}$ on $(\Omega_0 \setminus \{\omega^-_2\}) \times B$.  We can thus find $h_2 \in F_0$ such that
\[
\psi_{\omega^-_2}(h_2) = b; \; h_2(\omega^+_2,b) = (\omega^+_2,b).
\]
Since $R_{\omega^-_2}$ is normal in $F_0(\omega^-_2)$, in fact $h_2$ fixes $\{\omega^+_2\} \times B$ pointwise, so $\psi_{\omega^+_2}(h_2) = 1$.  We can then set $p(v_2) = h_2$ and $p(v_i) = 1$ for all $i \ge 3$.
\end{enumerate}
In either case, by Proposition~\ref{prop:edge_construction}(iii), there exists $g \in G$ with the given portrait, and we deduce that $G \in \mc{H}_T$.

Finally, we argue that $\Lambda^k_0$ is properly contained in $F_0(\omega)$ for all $k \ge 2$.  Indeed, we have $|F_0(\omega_0):L_0| = d_1-1 > 1$, since the kernel of $\psi_1$ is trivial, and then $\Lambda^2_0 = L_0$ by Lemma~\ref{lem:edge_construction:necessary}; in turn, $\Lambda^k_0 \le \Lambda^2_0$ for all $k \ge 2$.
\end{ex}


\begin{ex}\label{ex:exceptional}
We will construct an example of the exceptional case (ii) from Theorem~\ref{intro:locally_prosoluble}.

Let $F_t = \PGaL_3(q_t)$ acting on $\Omega_t = P_2(q_t)$, where $q_0=5$ and $q_1=4$, fix $\omega_t \in F_t$ and let $B = \Sym(5)$.  Then for all $t \in \{0,1\}$, we have $B \cong \PGaL_2(q_t)$; as permutation groups, $\PGaL_2(4) = \Sym(P_1(4))$, while $\PGaL_2(5) \hookrightarrow \Sym(P_1(5))$ is an instance of the exceptional embedding of $\Sym(5)$ into $\Sym(6)$.  Hence there exists a surjective homomorphism $\psi_t: F_t(\omega_t) \rightarrow B$; the kernel $R_t = \Fb^2_{q_t} \rtimes \Fb^*_{q_t}$ of $\psi_t$ is the soluble radical of $F_t(\omega_t)$.  (Recall similar observations made in the proof of Lemma~\ref{lem:insoluble_line_index}.)  In particular, since $B$ is insoluble, we note that the group $F_t(\omega_t)$ is insoluble for $t \in \{0,1\}$.  Up to a choice of isomorphism between $\PGaL_2(q_t)$ and $B$, we can think of $\psi_t$ as being specified by the action of $F_t(\omega_t)$ on $P_1(q_t)$ given by Remark~\ref{rem:linear}.

We again follow the construction of the $F_t$-set $\Omega_t \times B$ and the group $G = \mathbf{U}(F_0,F_1,\mc{L}^*,\alpha_0,\alpha_1)$ as in Construction~\ref{const:edge}.  In this case we can describe the groups $K_t,K'_t,L_t,L'_t$ up to conjugacy.  The group $K_t$ corresponds to a point stabilizer of the action of $\PGaL_2(q_t)$ on $P_1(q_t)$: specifically, $K_0$ is a subgroup of $B$ of index $6$ of the form $C_5 \rtimes C_4$, whereas $K_1$ is conjugate to the point stabilizer $\Sym(4)$ of $B$ in its natural action on $5$ points.  For the subgroups $K'_t$, we find that $K'_0 = K_0$, whereas $K'_1$ has index $2$ in $K_1$ and arises as $K_1 \cap \Alt(5) = \Alt(4)$.  Via the homomorphism $\psi_{\omega_t}$ and the actions with which we have equipped $B$, we have an action of $F_t(\omega_t)$ on the projective line $P_1(q_{1-t})$; the group $L_t$ is a point stabilizer of this action, and we have $|L_0:L'_0|=2$ and $L_1 = L'_1$.  Specifically, we find that as abstract groups, $L_0$ is of the form $C^2_5 \rtimes (\SL_2(3) \rtimes C_4)$; $L'_0$ is of the form $C^2_5 \rtimes (\SL_2(3) \rtimes C_2)$; and $L_1 = L'_1$ is of the form $\Fb_{16} \rtimes \GaL_1(16)$.  Note that the groups $L_0,L'_0,L_1$ are all soluble.

One can check that $F_t$ acts $2$-by-block-transitively on the cosets of $L'_t$ for $t \in \{0,1\}$.  More precisely, in the terminology of \cite{Reidkblock}, the action of $F_0$ on $L'_0$ is the exceptional action with block size $10$ of $\PSL_3(5) = \PGaL_3(5)$ displayed in \cite[Table~1]{Reidkblock}, whereas the action of $F_1$ on $L'_1$ is the type QP action with parameters $d_x=d_y=1$ in the sense of \cite[Proposition~3.27]{Reidkblock}.  In particular, this ensures that $L'_t$ acts transitively on $\Omega_t \setminus \{\omega_t\}$, and hence we have $G \in \mc{H}_T$ by Lemma~\ref{lem:edge_construction:sufficient}.

We now argue that $G(x,z)$ is prosoluble, where $x$ and $z$ are any two vertices such that $d(x,z) = 2$.  By Lemma~\ref{lem:edge_construction:necessary}, $G(x,z)$ acts on $S_z(x,1)$ as the soluble group $L_{t(x)}$.  More generally, if $\mathbf{x}$ is any finite collection of vertices such that for all $y \in \mathbf{x}$ there exists $y' \in \mathbf{x}$ with $d(y,y')=2$, then $G(\mathbf{x})$ acts on $S(y,1)$ as a subgroup of $L_{t(y)}$ for all $y \in \mathbf{x}$; thus $G(\mathbf{x})/G_1(\mathbf{x})$ is contained in a finite direct product of copies of $L_0$ and $L_1$, so it is soluble.  In particular, $G_{k-1}(x,z)/G_k(x,z)$ is soluble for all $k \ge 1$, so $G(x,z)$ is prosoluble.

In particular, the rigid stabilizer of every half-tree is prosoluble, so $\Theta_t$ is soluble, and hence $\Theta_t \le O_\infty(F_t(\omega))$ for all $t \in \{0,1\}$.  Actually for this example we get exactly $\Theta_t = O_\infty(F_t(\omega))$.  The kernel of $\psi_{\omega}$ is exactly the soluble residual of $F_t(\omega)$, so given a half-tree $T'$ with root $x$ of type $t$, and $T''$ the tree spanned by all edges of $T$ that are not in $T'$, for all $g \in O_\infty(F_t(\omega))$ we can form a $T''$-portrait $p_g$ such that $p_g(x) = g$ and $p_g(y) = 1$ for all $y \in VT'' \setminus \{x\}$.  We then obtain suitable elements of $G$ to witness that $\Theta_t(G) \ge O_\infty(F_t(\omega))$ via Proposition~\ref{prop:edge_construction}(iii).

It remains to note that $\Lambda_t$ is one of the possibilities listed in Theorem~\ref{intro:locally_prosoluble} for $t=0,1$.  By Lemma~\ref{lem:edge_construction:necessary} we have $\Lambda_t \le L_t$.  On the other hand, we are in a situation where $L'_t$ is transitive for all $t \in \{0,1\}$, so the proof of Lemma~\ref{lem:edge_construction:sufficient} ensures that $\Lambda_t \ge L'_t$.  Thus $\Lambda_1 = L_1 = \Fb_{16} \rtimes \GaL_1(16)$ and $\Lambda_0 = C^2_5 \rtimes A(5)$ where
\[
A(5) \in \{\SL_2(3) \rtimes C_4, \SL_2(3) \rtimes C_2\}.
\]

To summarize, we have obtained a group $G \in \mc{H}_{\PGaL_3(5),\PGaL_3(4)}$ acting on the $(31,21)$-semiregular tree, such that the local actions both have insoluble point stabilizers, so the arc stabilizers are not prosoluble, but such that the stabilizer of any path of length $2$ is prosoluble.
\end{ex}

\newpage

\printindex

\end{document}